\documentclass[a4paper,twoside,10pt]{article}

\usepackage[a4paper,left=3cm,right=3cm, top=3cm, bottom=3cm]{geometry}

\usepackage{amsthm}%
\usepackage{mathrsfs}%
\usepackage[title]{appendix}%
\usepackage{xcolor}%
\usepackage{textcomp}%
\usepackage{manyfoot}%
\usepackage{booktabs}%
\usepackage{algorithm}%
\usepackage{algorithmicx}%
\usepackage{algpseudocode}%
\usepackage{listings}%
\usepackage{amsfonts} 
\usepackage{graphicx}
\usepackage{epstopdf}
\usepackage{scalerel}
\usepackage{scalerel} 
\usepackage{framed,multirow}
\usepackage{etoolbox}
\usepackage{subfloat}
\usepackage{subcaption}
\usepackage{stmaryrd}
\usepackage{amscd,amsmath,amssymb,mathrsfs,bbm,listings} 
\usepackage{amsthm}
\usepackage{tikz} 
\usetikzlibrary{arrows,automata,calc} 
\usepackage{multirow}        
\usepackage[hidelinks]{hyperref}
\hypersetup{
    colorlinks=true,
    linkcolor=red,
    citecolor=blue,
    pdfborder={0 0 0}
}

\def\dx{\,\mathrm{d}\bx} 
\def\dS{\,\mathrm{d}S}
\def\dV{\,\mathrm{d}V}

\def\Deltax{\Delta_{\bx}}
\def\nablax{\nabla_{\bx}}
\def\QT{Q_T}
\newcommand\QTrefftz[2]{\mathbb{Q}\mathbb{T}^{#1}\left(#2\right)}

\def\calS{\mathcal{S}}

\newcommand\REL[1]{\mathfrak{Re}\left({#1}\right)}

\newcommand{\dfdt}[1]{\partial_t{#1}}
\newcommand{\dfhdt}[1]{\frac{d {#1}}{d t}}

\newcommand{\EFC}[2]{\calC^{#1}\left({#2}\right)}
\newcommand{\ESOBOLEV}[2]{H^{#1}\left({#2}\right)}
\newcommand{\PartDer}[2]{\partial^{#1}_{#2}}
\newcommand{\DerA}[2]{D^{#1}{#2}}

\newcommand{\Taylor}[3]{T_{#1}^{#2}\left[#3\right]}

\newcommand{\AvTaylor}[2]{\calQ^{#1}\left[#2\right]}

\newcommand{\Conjugate}[1]{\overline{#1}}
\newcommand{\jump}[1]{\llbracket#1\rrbracket}

\newcommand{\ORDER}[1]{{\mathcal O}\left(#1\right)}

\numberwithin{equation}{section}

\newcommand{\vnSpace}[1]{{{\vn}_{#1}^{\bx}}}
\newcommand{\vnTime}[1]{{n_{#1}^t}}
\newlength{\dhatheight}

\allowdisplaybreaks[4]
\newcommand{\normalDer}{\partial_{\mathbf{n_x}}}
\newcommand{\nF}{\vec{\bn}_F}
\newcommand{\nFt}{n_F^t}
\newcommand{\nFx}{\vec{\bn}_F^{\bx}}
\newcommand{\mvec}[1]{{\vec{\mathbf{#1}}}}
\newcommand{\vn}{{\mvec n}}

\newcommand{\Uu}[1]{{\mathbf{#1}}}

\newcommand{\bj}{{\boldsymbol j}}
\newcommand{\bjx}{{\boldsymbol j_\bx}}
\newcommand{\bjxo}{{\boldsymbol{j}_{\bx_1}}}
\newcommand{\bjxl}{{\boldsymbol{j}_{\bx_\ell}}}

\newcommand{\bV}{{\Uu V}}

\newcommand{\calC}{{\mathcal C}}

\newcommand{\calT}{{\mathcal T}}
\newcommand{\calE}{{\mathcal E}}
\newcommand{\calF}{{\mathcal F}}

\newcommand{\calQ}{{\mathcal Q}}
\newcommand{\Fh}{\calF_h}
\newcommand{\Kx}{K_{\bx}}
\newcommand{\Kt}{K_t}
\newcommand{\hK}{h_{K}}
\newcommand{\hKx}{h_{K_\bx}}
\newcommand{\hKt}{h_{K_t}}
\newcommand{\Th}{{\calT_h}}

\newcommand{\deK}{{\partial K}}
\newcommand{\GD}{{\Gamma_{\mathrm D}}}
\newcommand{\GN}{{\Gamma_{\mathrm N}}}
\newcommand{\GR}{{\Gamma_{\mathrm R}}}
\newcommand{\uhp}{{\psi_{hp}}}
\newcommand{\uhpT}{{\psi_{hp}^{-}}}
\newcommand{\shp}{{v_{hp}}}
\newcommand{\cs}{{\conj{\shp}}}
\newcommand{\cv}{{\conj{v}}}

\newcommand{\IN}{\mathbb{N}}
\newcommand{\IR}{\mathbb{R}}
\newcommand{\po}{\partial \Omega}
\newcommand{\gD}{{g_{\mathrm D}}}
\newcommand{\gN}{{g_{\mathrm N}}}
\newcommand{\gR}{{g_{\mathrm R}}}

\newcommand{\oon}{\;\text{on}\;}
\newcommand{\bx}{{\Uu x}}
\newcommand{\by}{{\Uu y}}
\newcommand{\bn}{{\Uu n}}

\newcommand{\bN}{{\Uu N}}

\newcommand{\bz}{{\Uu z}}
\newcommand{\btau}{{\boldsymbol\tau}}		
\newcommand{\IC}{\mathbb{C}}
\newcommand{\IP}{\mathbb{P}}
\newcommand{\IT}{\mathbb{T}}
\newcommand{\IV}{\mathbb{V}}

\newcommand*{\jmp}[1]{[\![#1]\!]}                     
 \newcommand{\mvl}[1]{\left\{\!\!\left\{#1\right\}\!\!\right\}}  
 \newcommand{\FT}{{\Fh^T}}
\newcommand{\FO}{{\Fh^0}}

\newcommand{\FD}{{\Fh^{\mathrm D}}}
\newcommand{\FN}{{\Fh^{\mathrm N}}}
\newcommand{\FR}{{\Fh^{\mathrm R}}}
\newcommand{\rtime}{{\mathrm{time}}}
\newcommand{\rspace}{{\mathrm{space}}}
\newcommand{\Fspa}{{\Fh^\rspace}}
\newcommand{\Ftime}{{\Fh^\rtime}}
\newcommand{\bVp}[1]{{\IT_{#1}}} 
\newcommand{\bVhp}{{\IV_{hp}}} 
\newcommand*{\conj}[1]{\overline{#1}}

\DeclareMathOperator{\im}{\mathfrak{Im}} 
\DeclareMathOperator{\diam}{diam} 
\DeclareMathOperator{\esssup}{ess\,sup}
\DeclareMathOperator{\essinf}{ess\,inf}

\newcommand{\cbA}[2]{{{\mathcal{A}}}\left({#1};\ {#2}\right)}

\newcommand*{\abs}[1]{\left|#1\right|}
\newcommand{\Tnorm}[2]{|||#1|||_{#2}}

\newcommand*{\Norm}[2]{\left\|#1\right\|_{#2}}
\newcommand{\DG}{_{\mathrm{DG}}}
\newcommand{\DGp}{_{\mathrm{DG^+}}}

\newcommand{\deO}{{\partial\Omega}}

\newcommand{\bv}{{\Uu v}}

\newcommand{\tr}{\mathrm{tr}}

\theoremstyle{thmstyleone}%
\newtheorem{theorem}{Theorem}

\theoremstyle{thmstyletwo}%
\newtheorem{remark}{Remark}%
\newtheorem{lemma}{Lemma}%
\newtheorem{proposition}{Proposition}%

\theoremstyle{thmstylethree}%
\title{A space--time DG method for the Schr\"odinger equation with variable potential\thanks{The authors acknowledge support from GNCS-INDAM, from PRIN projects ``NA-FROM-PDEs'' and ``ASTICE", and from PNRR-M4C2-I1.4-NC-HPC-Spoke6.}}

\author{\large{Sergio G\'omez\thanks{Department of Mathematics and Applications, University of Milano-Bicocca, Via Cozzi 55, 20125, Milan, Italy (\href{mailto:sergio.gomezmacias@unimib.it}{sergio.gomezmacias@unimib.it})}\; and Andrea Moiola\thanks{Department of Mathematics, University of Pavia, Via Ferrata 5, 27100, Pavia, Italy (\href{mailto:andrea.moiola@unipv.it}{andrea.moiola@unipv.it})}}} 

\begin{document}

\maketitle 

\begin{abstract} 
We present a space--time ultra-weak discontinuous Galerkin discretization of  
the linear Schr\"odinger equation with variable potential. The proposed method is well-posed and quasi-optimal in mesh-dependent norms for very general discrete spaces.
Optimal~$h$-convergence error estimates are derived for the method when test and trial spaces are chosen either as piecewise  polynomials, or as a novel quasi-Trefftz polynomial space.
The latter allows for a substantial reduction of the number of degrees of freedom and admits piecewise-smooth potentials.
Several numerical experiments validate the accuracy and advantages of the proposed method. 
\end{abstract}


\maketitle


\section{Introduction}\label{SEC::INTRODUCTION}
In this work we are interested in the approximation of the solution to the time-dependent Schr\"odinger equation on a space--time cylinder~$\QT = \Omega \times I$, where~$\Omega \subset \IR^d\ (d \in \IN)$ is an open, 
bounded polytopic domain with Lipschitz boundary~$\po$, and~$I =
(0, T)$ for   
some final time~$T > 0$:
\begin{equation}
\label{EQN::SCHRODINGER-EQUATION}
\begin{split}
\calS \psi := i \dfdt{\psi} + \frac{1}{2}\Deltax \psi - V \psi & = 0 \quad \quad\;\; \mbox{ in }\ \QT, \\
\psi& =\gD  \quad\quad \oon\ \GD\times I,\\
\normalDer \psi & = \gN \quad\quad \oon\ \GN\times I,\\
\normalDer \psi - i \vartheta\psi & = \gR \quad \quad \oon\ 
\GR\times I,\\
\psi(\bx, 0) & = \psi_0(\bx) \ \  \oon\ \Omega.
\end{split}
\end{equation}
Here~$i$ is the imaginary unit;  $\normalDer(\cdot)$ is the normal derivative-in-space operator; $V: \QT \rightarrow \IR$ is the potential energy function; $\vartheta \in L^\infty(\GR \times I)$ is a positive ``impedance" function; the Dirichlet~($\gD$), Neumann~($\gN$), Robin~($\gR$) and initial condition~($\psi_0$) data are given functions; $\GD,\GN,\GR$ are a polytopic partition of~$\deO$. 
 
The model problem~\eqref{EQN::SCHRODINGER-EQUATION} has a wide range of applications.
In quantum physics~\cite{Lifshitz_Landau_1965}, 
the solution~$\psi$  is a quantum-mechanical wave function determining the dynamics of one or multiple particles in a potential~$V$.
In electromagnetic wave propagation~\cite{Levy_2000}, it is called ``paraxial wave equation" and $\psi$ is a function associated with the field component in a two-dimensional electromagnetic problem where the energy propagates at small angles from a preferred direction. In such problems, the function $V$ depends on the refractive index and the wave number.
In underwater sound propagation~\cite{Keller_Papadakis_1977}, it is referred to as ``parabolic equation" and~$\psi$ describes a time harmonic wave propagating primarily in one direction.
In molecular dynamics~\cite{Born_Oppenheimer_2000}, by neglecting the motion of the atomic nuclei, the Born-Oppenheimer approximation leads to a Schr\"odinger equation in the \emph{semi-classical} regime.

Space--time Galerkin methods discretize all the variables in a time dependent PDE at once; this is in contrast with the method of lines, which combines a spatial discretization and a time-stepping scheme. 
Space--time methods can achieve high convergence rates in space and time, and provide discrete solutions that are available on the whole space--time domain.

The literature on space--time Galerkin methods for the Schr\"odinger equation is very scarce. 
In fact, the standard Petrov-Galerkin formulation for the Schr\"odinger equation, i.e.,\ the analogous formulation to that proposed in~\cite{Steinbach:2015} for the heat equation, is not inf-sup stable, see~\cite[Sect. 2.2]{Hain_Urban_2022}.
In~\cite{Karakashian_Makridakis_1998}, Karakashian and
Makridakis proposed a space--time method for the Schr\"odinger equation with nonlinear potential, combining a conforming Galerkin discretization in space and an upwind DG time-stepping. This method reduces to a Radau IIA Runge-Kutta time discretization in the case of constant potentials.
Moreover, under some restrictions on the mesh that are necessary to preserve the accuracy of the method, it allows for changing the spatial mesh on each time-slab, but not for local time-stepping.
A second version of the method, obtained by enforcing the transmission of information from the past through a projection, was proposed in~\cite{Karakashian_Makridakis_1999}. This version reduces to a Legendre Runge-Kutta time discretization in the case of constant potentials.
Recently, some space--time methods based on ultra-weak formulations of the Schr\"odinger equation have been designed. The well-posedness of such formulations requires weaker assumptions on the mesh. 
Demkowicz
et al., in~\cite{Demkowicz_ETAL_2017}, the authors proposed a discontinuous Petrov-Galerkin (DPG) formulation for the linear Schr\"odinger equation. The method is a conforming discretization of an ultra-weak formulation of the Schr\"odinger equation in graph spaces. Well-posedness and quasi-optimality of the method follow directly from the inf-sup stability (in a graph norm) of the continuous Petrov-Galerkin formulation. 
In~\cite{Hain_Urban_2022}, Hain and Urban proposed a space--time ultra-weak variational formulation for the Schr\"odinger equation with optimal inf-sup constant.
The formulation in~\cite{Hain_Urban_2022} is closely related to the DPG method in~\cite{Demkowicz_ETAL_2017}, but differs in the choice of the test and trial spaces. While for the method in~\cite{Demkowicz_ETAL_2017} one first fixes a trial space and then construct a suitable test space, the method in~\cite{Hain_Urban_2022} requires the choice of a conforming test space and then the trial space is defined accordingly.
We are not aware of publications proposing space--time DG methods for the Schr\"odinger equation other than~\cite{Demkowicz_ETAL_2017,Hain_Urban_2022,Karakashian_Makridakis_1998,Karakashian_Makridakis_1999}, outlined in this paragraph, and the space--time Trefftz-DG method in~\cite{Gomez_Moiola_2022,Gomez_Moiola_Perugia_Stocker_2023}, which motivated the present paper.

Trefftz methods are Galerkin discretizations with test and trial spaces spanned by local solutions of the considered PDE.
Trefftz methods with lower-dimensional spaces than standard finite element spaces, but similar approximation properties,
have been designed for many problems, e.g.,
Laplace and solid-mechanics problems~\cite{Qin05}; the Helmholtz equation~\cite{Hiptmair_Moiola_Perugia_2016}; the time-harmonic~\cite{Hiptmair_Moiola_Perugia_2013}, and time-dependent~\cite{Egger_Kretzchmar_Schnepp_Weiland_2015} Maxwell's equations;
the acoustic wave equation in second-order~\cite{Banjai_Georgoulis_Lijoka_2017} and first-order~\cite{Moiola_Perugia_2018} form; the Schr\"odinger equation~\cite{Gomez_Moiola_2022}; among others. Nonetheless, pure Trefftz methods are essentially limited to problems with piecewise-constant coefficients, as for PDEs with varying coefficients the design of ``rich enough" finite-dimensional Trefftz spaces 
is in general not possible.
A way to overcome this limitation is the use of quasi-Trefftz methods, which are based on spaces containing functions that are just approximate local solutions to the PDE. In essence, the earliest quasi-Trefftz spaces are the generalized plane waves used   in~\cite{ImbertGerard_Desperes_2014} for the discretization of the Helmholtz equation with smoothly varying coefficients. More recently, a quasi-Trefftz DG method 
for the acoustic wave equation with piecewise-smooth material parameters was proposed in~\cite{ImbertGerard_Moiola_Stocker}, where some polynomial quasi-Trefftz spaces were introduced. As an alternative idea, the embedded Trefftz DG method proposed in~\cite{Lehrenfeld_Stocker_2022} does not require the local basis functions to be known in advance, as they are simply taken as a basis for the kernel of the local discrete operators in a standard DG formulation. This corresponds to a Galerkin projection of a DG formulation with a predetermined discrete space onto a Trefftz-type subspace. In practice, it requires the computation of singular or eigenvalue decompositions of the local matrices.

In~\cite{Gomez_Moiola_2022}, the authors proposed a space--time Trefftz-DG method for the Schr\"odinger equation with piecewise-constant potential, whose well-posedness and quasi-optimality in mesh-dependent norms were proven for  general discrete Trefftz spaces. Optimal~$h$-convergence estimates were shown for a Trefftz space consisting of complex-exponential wave functions.

In this work we propose a space--time DG method for the discretization of the Schr\"odinger equation with variable potentials, extending the formulation of \cite{Gomez_Moiola_2022} to more general problems and discrete spaces.
The main advantages of the proposed method
are the following:
\begin{itemize}
\item The proposed ultra-weak DG variational formulation of~\eqref{EQN::SCHRODINGER-EQUATION} is well-posed, stable, and quasi-optimal in any space dimension for an almost arbitrary choice of piecewise-defined discrete spaces and variable potentials. 
\item \emph{A priori} error estimates in a mesh-dependent norm can be obtained by simply analyzing the approximation properties of the local spaces.
\item The method naturally allows for non-matching
space-like and time-like facets and all our theoretical results hold under standard assumptions on the space--time mesh, which make
the method suitable for adaptive versions and local time-stepping.
\item Building on~\cite{ImbertGerard_Moiola_Stocker}, for elementwise smooth potentials, we design and analyze a quasi-Trefftz polynomial space with similar approximation properties of full polynomial spaces but with much smaller dimension, thus substantially reducing the total number of degrees of freedom required for a given accuracy.
\end{itemize}

\textbf{Structure of the paper:} In Section \ref{SECT::ULTRA-WEAK-DG} we introduce some notation on the space--time meshes to be used and the proposed ultra-weak DG variational formulation on abstract spaces. Section \ref{SECT::WELL-POSEDNESS} is devoted to the analysis of well-posedness, stability and quasi-optimality of the method. In Sections \ref{SECT::ERROR-ESTIMATE-FULL-POLYNOMIAL} and \ref{SECT::ERROR-ESTIMATE-QUASI-TREFFTZ} we prove optimal $h$-convergence estimates for the method when the test and trial spaces are taken as the space of piecewise polynomials or a novel quasi-Trefftz space, respectively. In Section \ref{SECT::NUMERICAL-RESULTS} we present some numerical experiments that validate our theoretical results  and illustrate the advantages of the proposed method. We end with some concluding remarks in Section \ref{SECT::CONCLUSIONS}.


\section{Ultra-weak discontinuous Galerkin formulation\label{SECT::ULTRA-WEAK-DG}}
\subsection{Space--time mesh and DG notation}\label{S:Mesh}
\noindent Let~$\Th$ be a non-overlapping prismatic partition of~$\QT$, i.e., each element $K \in \Th$ can be written as~$K = \Kx \times \Kt$ for a~$d$-dimensional polytope~$\Kx \subset \Omega$ and a time interval~$\Kt \subset I$.
We use the notation~$\hKx = \diam(\Kx)$, $\hKt = \abs{\Kt}$ and $\hK = \diam(K)=(\hKx^2+\hKt^2)^{1/2}$.
We call ``mesh facet'' any intersection $F=\deK_1\cap \deK_2$ or $F=\deK_1\cap\partial Q_T$, for
$K_1,K_2\in\Th$, that has positive $d$-dimensional measure and is contained in a $d$-dimensional hyperplane.
We denote by~$\nF = (\nFx, \nFt) \in \IR^{d+1}$ one of the two unit normal vectors orthogonal to $F$ with $\nFt = 0$ or $\nFt = 1$.
We assume that each internal mesh facet $F$ is either
\begin{equation*}
\text{a space-like facet} \quad \text{if } \nFx = 0, \quad \text{ or}\quad
\text{a time-like facet} \quad \text{if } \nFt = 0.
\end{equation*}
We further denote the mesh skeleton and its parts as

\begin{align*}
\Fh :=& \bigcup_{K \in \Th} \deK,\qquad
\FO := \Omega \times \left\{0\right\}, \quad 
\FT := \Omega \times \left\{T\right\}, \\
\FD :=& \GD \times (0, T), \quad \FN := \GN \times (0, T), \quad \FR := \GR \times (0, T),
\\
\Ftime :=& \mbox{ the union of all the internal time-like facets},\\
\Fspa  :=& \mbox{ the union of all the internal space-like facets}.
\end{align*}
We employ the standard DG notation for the averages $\mvl{\cdot}$ and 
space $\jump{\cdot}_{\bN}$ and time $\jump{\cdot}_t$ jumps for 
piecewise complex scalar $w$ and vector $\btau$ fields:
\begin{align*}
&\begin{cases}
\mvl{w} : = \frac{1}{2} \left(w|_{K_1} + w|_{K_2}\right)\\
\mvl{\btau} : = \frac{1}{2} \left(\btau|_{K_1} + \btau|_{K_2}\right)
\end{cases}
&&\oon \deK_1\cap\deK_2\subset\Ftime,
\\
&\begin{cases}
\jump{w}_\bN : = w|_{K_1} \vnSpace{K_1} + w|_{K_2} \vnSpace{K_2}\\
\jump{\btau}_\bN : = \btau|_{K_1} \cdot \vnSpace{K_1} + \btau|_{K_2} \cdot \vnSpace{K_2}
\end{cases}
&&\oon\deK_1\cap\deK_2\subset\Ftime,
\\
&
\ \jump{w}_t : = w|_{K_1} \vnTime{K_1} + w|_{K_2} \vnTime{K_2} = w^- - w^+,
&&\oon\deK_1\cap\deK_2\subset\Fspa,
\end{align*}
where $\vnSpace{K}\in\IR^d$ and $\vnTime{K} \in \IR$ are the space and time components of the outward-pointing unit 
normal vectors on $\deK\cap\Ftime$ and $\deK \cap \Fspa$, respectively.  The superscripts ``$-$" and ``$+$" are used to denote the traces of a function on a space-like facet from the elements ``before" ($-$) and ``after" ($+$) the facet. 

The space--time prismatic meshes described in this section may include hanging space-like and time-like facets, so the proposed method allows for local time-stepping and local space--time refinements. Tent-pitched meshes are popular in space--time methods for wave propagation problems; see e.g., \cite{Perugia_Schoeberl_Stocker_Wintersteiger_2020} and~\cite[Eq.~3]{Moiola_Perugia_2018}. However, such meshes
do not lead to a semi-implicit discretization of
the Schr\"odinger equation because the propagation speed of its
solutions, which dictates the slope of space-like facets of the tents,
is infinite.

We denote space--time broken function spaces as $H^s(\Th):=\{v\in L^2(Q_T),\;v|_K\in H^s(K)\; \forall K\in \Th\}$, $\EFC{s}\Th:=\{v:Q_T\to\IC,\;v|_K\in \EFC{s}K\; \forall K\in \Th\}$, for $s\in\IN_0$. 

\subsection{Variational formulation of the DG method}\label{S:DG} 

For any finite-dimensional subspace~$\bVhp\left(\Th\right)$ of the broken Bochner--Sobolev space
$$\bV(\Th) := \prod_{K \in \Th} \ESOBOLEV{1}{\Kt; L^2(\Kx)} \cap L^2\left(\Kt;
\ESOBOLEV{2}{\Kx}\right),$$
the proposed ultra-weak DG variational formulation for the Schr\"odinger equation \eqref{EQN::SCHRODINGER-EQUATION} is:
\begin{equation}
\label{EQN::VARIATIONAL-DG}
\mbox{Seek }\uhp \in \bVhp(\Th) \mbox{ such that: } \cbA{\uhp}{\shp} = 
\ell(\shp) \quad \forall \shp \in \bVhp(\Th),
\end{equation}
where
\begin{align*}
\cbA{\uhp}{\shp}   := &   
\sum_{K \in \Th} \int_K \uhp \conj{\calS \shp}
\dV + i \left(\int_{\Fspa} \uhpT \jump{\cs}_t  \dx + \int_{\FT}  
\uhp \cs \dx \right) \\
& + \frac{1}{2} \int_{\Ftime} \Big(\mvl{\nablax \uhp} \cdot \jump{\cs}_{\bN} + i \alpha 
\jump{\uhp}_{\bN} \cdot \jump{\cs}_{\bN} \\
& - \mvl{\uhp} \jump{\nablax
\cs}_{\bN}  + i \beta 
\jump{\nablax \uhp}_{\bN} \jump{\nablax \cs}_{\bN}\Big) \dS \\
& + \frac{1}{2}  \int_{\FD} 
\left(\normalDer \uhp + i \alpha
\uhp \right)\cs \dS\\
&   + \frac{1}{2}  \int_{\FN} \left(-\uhp \normalDer \Conjugate{\shp} + 
i\beta \left(\normalDer \uhp \right) \left(\normalDer  \Conjugate{\shp}\right) \right) \dS\\
& + \frac{1}{2}  \int_{\FR} \left(\delta \normalDer \uhp 
+(1-\delta)i\vartheta \uhp\right)\left(\cs+\frac{i}\vartheta 
\normalDer \cs\right)\dS \\
& + i \sum_{K \in \Th} \int_K \mu 
\calS \uhp \conj{\calS \shp} \dV, 
\\
 \ell(\shp)  := &  i \int_{\FO}  \psi_0 \cs \dx + \frac{1}{2} \int_{\FD} \gD \left(\normalDer  
\cs + i \alpha\cs\right) \dS \\
& + \frac{1}{2} \int_{\FN} \gN \left(-\Conjugate{\shp} + i \beta \normalDer \cs 
\right) \dS  + \frac{1}{2} \int_{\FR} \gR \left((\delta - 1) \Conjugate{\shp} + 
\frac{i\delta}{\vartheta} \normalDer \cs \right)\dS,
\end{align*}
for some mesh-dependent stabilization functions
\begin{align*}
\alpha \in L^{\infty}(\Ftime \cup \FD),\qquad &\essinf_{\Ftime\cup\FD}\alpha>0,
\\
\beta \in L^{\infty}(\Ftime \cup \FN),\qquad &\essinf_{\Ftime\cup \FN} \beta>0,
\\
\delta  \in L^{\infty}(\FR), \qquad & 0<\delta\leq \frac12,
\\
\mu \in L^\infty(\QT), \qquad & \essinf_{\QT}\mu>0.
\end{align*}
More conditions on these functions, in particular on their dependence on the local mesh size, will be specified in Section~\ref{SECT::ERROR-ESTIMATE}.

The variational formulation \eqref{EQN::VARIATIONAL-DG} can be derived by integrating by parts twice in space and once in time in each element as in \cite{Gomez_Moiola_2022}, and treating the Neumann and the Robin boundary terms similarly to \cite[Rem.~3.7]{Gomez_Moiola_2022}.
However, as the current setting does not require the discrete space~$\bVhp(\Th)$ to satisfy the Trefftz property~($\calS \psi_{|_K} = 0, \ \forall K \in \Th$), there are an additional volume term that is needed to ensure consistency (the first integral over~$K$ in $\cbA\cdot\cdot$), and a local \emph{Galerkin-least squares} correction term (the second integral over $K$ in~$\cbA\cdot\cdot$) that were not present in the previous method.
Such additional terms vanish when~$\bVhp(\Th)$ is a discrete Trefftz space, thus recovering the formulation in~\cite{Gomez_Moiola_2022}. 

\begin{remark}[Implicit time-stepping through time-slabs\label{REM::TIME-SLABS}]
The variational problem~\eqref{EQN::VARIATIONAL-DG} is a global problem involving all the degrees of freedom of the discrete solution for the whole space--time cylinder~$\QT$.
However, as upwind numerical fluxes are taken on the space-like facets, if the space--time prismatic mesh~$\Th$ can be decomposed into time-slabs (i.e., if the mesh elements can be grouped in sets of the form~$\Omega \times [t_{n-1}, t_n]$ for a partition of the time interval of the form $0 = t_0 < t_1 < \ldots < t_N = T$), the global linear system stemming from~\eqref{EQN::VARIATIONAL-DG} can be solved as a sequence of~$N$ smaller systems of the form
\begin{equation*}
    \mathbf{K}_n \Psi_h^{(n)} = b_n \quad 1 \leq n \leq N,
\end{equation*}
where~$b_n = \mathbf{R}_n \Psi_h^{(n - 1)}$ for $n = 2, \ldots, N$. 
This is comparable to an implicit time-stepping, and it naturally allows for local mesh refinement in different regions of the space--time cylinder~$Q_T$.
Moreover, when~$\Th$ is a tensor-product space--time mesh, the potential~$V$ does not vary in time, and the partition of the time interval is uniform, the matrices $\mathbf{K}_n$ and $\mathbf{R}_n$ are the same for every time-slab.
\end{remark}

\begin{remark}[Self-adjointness and volume penalty term] 
The well-posedness of the variational formulation \eqref{EQN::VARIATIONAL-DG} strongly relies on the $L^2(K)$-self-adjointness of the Schr\"odinger operator~$\calS(\cdot)$ on each~$K\in \Th$ (in the sense that $\int_K\calS\psi\,\conj\varphi\dV=\int_K\psi\,\conj{\calS\varphi}\dV$ for all $\psi\in\bV(\Th)$, $\varphi\in \calC^\infty_0(K)$, thanks to the fact that the only odd derivative in $\calS$ is multiplied to the imaginary unit), which makes the local Galerkin-least squares correction term consistent.
On the one hand, such term is essential in the proof of coercivity of the sesquilinear form $\cbA{\cdot\ }{\!\cdot}$ (see Proposition~\ref{PROP::COERCIVITY} below).
On the other hand, numerical experiments suggest that it can be neglected without losing accuracy and stability, see Section~\ref{SECT::STABILIZATION} below.
This is also the case for the quasi-Trefftz DG method for the Helmholtz equation \cite[\S5.1.3]{ImbertGerard_Monk_2017} and for the wave equation \cite[\S5.1]{ImbertGerard_Moiola_Stocker}, where a similar correction term was used. 
Nonetheless, in the design of an ultra-weak DG discretization for a PDE with a non-self-adjoint differential operator~$\mathcal{L}(\cdot)$ (e.g., the heat operator~$\mathcal{L}(\cdot) = \left(\partial_t - \Deltax{} \right)(\cdot)$), the corresponding local least-squares correction term~$\sum_{K \in \Th} \int_K \mu\mathcal L\uhp\conj{\mathcal L\shp}\dV$ would not control the consistency term~$\sum_{K \in \Th}\int_K\uhp\conj{\mathcal L^* \shp}\dV$ arising from the integration by parts.
\end{remark} 

\begin{remark}[Time-dependent potentials]
The variational problem \eqref{EQN::VARIATIONAL-DG} allows for time-dependent potentials~$V$. This is an important feature as, in such a case, the method of separation of variables cannot be used to reduce the time-dependent problem \eqref{EQN::SCHRODINGER-EQUATION} to the time-independent Schr\"odinger equation. 
\end{remark}


\section{Well-posedness, stability and quasi-optimality of the DG method}
\label{SECT::WELL-POSEDNESS}
\noindent The theoretical results in this section are derived for any spatial dimension~$d$, and are independent of the specific choice of the discrete space~$\bVhp(\Th)$.

Recalling that the volume penalty function~$\mu$, the stabilization functions~$\alpha$, $\beta$ and the impedance function~$\vartheta$ are positive, and that~$\delta \in(0, \frac12)$,
we define the following mesh-dependent norms on~$\bV (\Th)$:\footnote{Observe that a factor~$\frac12$ is missing in the first term of the DG norm in~\cite[Eqn.~(3.2)]{Gomez_Moiola_2022}.}
\begin{align}
\nonumber
 \Tnorm{w}{\DG}^2  & : =  \, \sum_{K \in \Th} \Norm{\mu^\frac12
\calS w}{L^2(K)}^2 + \frac{1}{2} \left(\Norm{\jump{w}_t}{L^2(\Fspa)}^2 + 
\Norm{w}{L^2(\FT \cup \FO)}^2\right) \\
& \qquad + \frac{1}{2}\Bigg( \Norm{\alpha^\frac12\jump{w}_{\bN}}{L^2(\Ftime)^d}^2 
+ \Norm{\beta^\frac12 \jump{\nablax w}_{\bN}}{L^2(\Ftime)}^2 + \Norm{\alpha^\frac12w}{L^2(\FD)}^2  \label{EQN::DG-NORM}\\
& \nonumber \qquad + 
 \Norm{\beta^{\frac12} \normalDer w }{L^2(\FN)}^2 \!+ 
\Norm{\big(\vartheta (1 - \delta)\big)^\frac12 w}{L^2(\FR)}^2 + \Norm{\left(\delta 
\vartheta^{-1}\right)^\frac12 \normalDer w }{L^2(\FR)}^2 \Bigg),
\\
\nonumber \Tnorm{w}{\DGp}^2  & : = \, \Tnorm{w}{\DG}^2 + \sum_{K \in \Th} \Norm{\mu^{-\frac12}w}{L^2(K)}^2 + \frac{1}{2} \Norm{w^-}{L^2(\Fspa)}^2 \\
& \label{EQN::DGplus-NORM} \qquad + 
\frac{1}{2} \Bigg(\Norm{\alpha^{-\frac12}\mvl{\nablax w}}{L^2(\Ftime)^d}^2  + \Norm{\alpha^{-\frac12} \normalDer  
w }{L^2(\FD)} \\
& \qquad + 
\Norm{\beta^{-\frac12}\mvl{w}}{L^2(\Ftime)}^2  + \Norm{\beta^{-\frac12} w}{L^2(\FN)}^2 
+ \Norm{\delta^{-\frac12}\vartheta^{\frac12} w }{L^2(\FR)}^2\Bigg).
\nonumber
\end{align}
The sum of the~$L^2(K)$-type terms ensures that~$\Tnorm\cdot\DGp$ is a norm.
That $\Tnorm\cdot\DG$ is a norm on~$\bV(\Th)$ follows from the following reasoning (see also \cite[Lemma~3.1]{Gomez_Moiola_2022}): if~$w\in \bV(\Th)$ and~$\Norm{w}{\DG} = 0$, then~$w$ is the unique variational solution to the Schr\"odinger equation~\eqref{EQN::SCHRODINGER-EQUATION} with homogeneous initial and boundary conditions.
Moreover, by the energy conservation (if~$\FR= \emptyset$) or dissipation (if~$\FR \neq \emptyset$), then~$\Norm{w(\cdot, t)}{L^2(\Omega)}^2 \leq \Norm{w(\cdot, 0)}{L^2(\Omega)}^2 = 0$, for all~$t \in(0, T]$; therefore, $w = 0$.

The DG norms in \eqref{EQN::DG-NORM}--\eqref{EQN::DGplus-NORM} are chosen in order to ensure the following properties of the sesquilinear form~$\cbA{\cdot}{\!\cdot}$ and the antilinear functional~$\ell(\cdot)$, from which the well-posedness and quasi-optimality of the method \eqref{EQN::VARIATIONAL-DG} follow.

\begin{proposition}[Coercivity]
\label{PROP::COERCIVITY} 
For all $w \in \bV(\Th)$ the following identity holds 
\begin{equation*}
\label{EQN::COERCIVITY}
\im\big(\cbA{w}{w}\big) = \Tnorm{w}{\DG}^2.
\end{equation*}
\end{proposition}
\begin{proof}
The result follows from the following identities (see \cite[Prop.~3.2]{Gomez_Moiola_2022} for more details):
\begin{align*} 
\int_\Fspa \left(\REL{v^- \jump{\cv}_t}-  \frac{1}{2} \jump{\abs{v}^2}_t\right) \dx & = 
\frac{1}{2 } \int_\Fspa \abs{\jump{v}_t}^2\dx
& \hspace*{\fill} \forall v \in H^1(\Th),\\
\int_\Ftime \left(\mvl{v} \jump{\btau}_{\bN} + \mvl{\btau} \cdot \jump{v}_{\bN}\right) \dS & = \int_\Ftime \jump{v \btau}_{\bN} \dS 
& \hspace*{\fill} \forall (v, \btau) \in H^1(\Th) \times H^1(\Th)^d,
\end{align*}
\vskip-0.25in
\begin{align*}
\im \left(\sum_{K \in \Th}\int_K w 
\conj{\calS w} \dV\right) & = -\frac{1}{2} \bigg(\int_{\Fspa} \jump{\abs{w}^2}_t \dx +   
\int_{\FT}\abs{w}^2 \dx - \int_\FO  \abs{w}^2 \dx \bigg) \\
& \quad + \frac{1}{2} \im \bigg(\int_{\Ftime} \jump{w \nablax \conj{w}}_{\bN} \dS 
+  \int_{\po \times I }   w  \normalDer \conj{w} \dS\bigg) 
&\hspace*{\fill} \forall w \in \bV(\Th).
\end{align*}
\end{proof}

\begin{proposition}[Continuity]
\label{PROP::CONTINUITY} 
The sesquilinear form~$\cbA{\cdot}{\cdot}$ and the antilinear functional~$\ell(\cdot)$ are continuous in the following sense:
$\forall v,w \in \bV(\Th)$
\begin{subequations}
\begin{align}
\label{EQN::CONTINUITY-SESQUILINEAR-FORM}
\abs{\cbA{v}{w}}  \leq &  \ 2\Tnorm{v}{\DGp} \Tnorm{w}{\DG},
\\
\abs{\ell(v)}  \leq &  \Big(2\Norm{\psi_0}{L^2(\FO)}^2 + 
\Norm{\alpha^{\frac12}\gD}{L^2(\FD)}^2 + \Norm{\beta^{\frac12}\gN}{L^2(\FN)}^2 + \Norm{\vartheta^{-\frac12}\gR}{L^2(\FR)}^2 \Big)^{\frac12} \Tnorm{w}{\DGp}.
\label{EQN::CONTINUITY-LINEAR-FUNCTIONAL}
\end{align}
\end{subequations}
\end{proposition}
\begin{proof}
The terms on~$\Fspa,\FT,\FO,\Ftime$ and~$\FD$ are controlled as in \cite[Prop.~3.3]{Gomez_Moiola_2022}. 
The remaining terms are bounded using Cauchy--Schwarz inequality and the inequality~$\delta \leq 1 - \delta < 1$. 
\end{proof}

\begin{theorem}[Quasi-optimality]\label{THEOREM::WELL-POSEDNESS}
For any finite-dimensional subspace $\bVhp(\Th)$ of\ $\bV(\Th)$, there exists 
a unique solution $\uhp \in \bVhp(\Th)$ satisfying the variational formulation
\eqref{EQN::VARIATIONAL-DG}. Additionally, the following 
quasi-optimality bound holds:
\begin{equation}
\label{EQN::QUASI-OPTIMALITY}
\Tnorm{\psi - \uhp}{\DG} \leq 3 \inf_{\shp \in \bVhp(\Th)}\Tnorm{\psi - 
\shp}{\DGp}.
\end{equation}
Moreover, if~$\gD = 0$ and~$\gN = 0$ (or~$\GD = \emptyset$ and~$\GN = \emptyset$), then 
\begin{equation}
\label{EQN::CONTINUOUS-DEPENDENCE-DATA}
    \Tnorm{\uhp}{\DG} \leq \left(2\Norm{\psi_0}{L^2(\FO)}^2 + \Norm{\vartheta^{-1/2}\gR}{L^2(\FR)}^2 \right)^{1/2}.
\end{equation}
\end{theorem} 
\begin{proof}
Existence and uniqueness of the discrete solution~$\uhp \in \bVhp(\Th)$ of the variational formulation~\eqref{EQN::VARIATIONAL-DG}, and the quasi-optimality bound~\eqref{EQN::QUASI-OPTIMALITY} follow directly from Propositions~\ref{PROP::COERCIVITY}--\ref{PROP::CONTINUITY}, the consistency of the variational formulation \eqref{EQN::VARIATIONAL-DG} and Lax--Milgram theorem. 
The continuous dependence on the data~\eqref{EQN::CONTINUOUS-DEPENDENCE-DATA} follows from Proposition~\ref{PROP::COERCIVITY}, and the fact that if~$\gD = 0$ and~$\gN = 0$ (or $\GD = \emptyset$ and $\GN = \emptyset$), the term~$\Norm{w}{\DGp}$ on the right-hand side of~\eqref{EQN::CONTINUITY-LINEAR-FUNCTIONAL} can be replaced by~$\Norm{w}{\DG}$. 
\end{proof}

Theorem~\ref{THEOREM::WELL-POSEDNESS} implies that it is possible to obtain error estimates in the mesh-dependent norm~$\Tnorm{\cdot }{\DG}$ by studying the best approximation in~$\bVhp(\Th)$ of the exact solution in the~$\Tnorm{\cdot}{\DGp}$ norm. Moreover, according to Proposition \ref{PROP::DGp-BOUND} below, \emph{a priori} error estimates can be deduced from the local approximation properties of the space~$\bVhp(\Th)$ only, as the~$\Tnorm{\cdot}{\DGp}$ norm can be bounded in terms of volume Sobolev seminorms and norms. The proof of error estimates in mesh-independent norms on the full computational domain for ultra-weak DG methods is a delicate issue; see e.g., \cite[Lemma~1]{Hiptmair_Moiola_Perugia_2013} and~\cite[\S 5.4]{Moiola_Perugia_2018} for related results concerning Trefftz methods for the Helmholtz and the wave equations, respectively.

So far, we have not imposed any restriction on the space--time mesh~$\Th$. Henceforth, in our analysis we assume:
\begin{itemize}
\item {\bf Uniform star-shapedness}:  
There exists~$0 < \rho \leq \frac{1}{2}$ such that, each element~$K \in \Th$ is star-shaped with respect to the ball~$B:= B_{\rho \hK}(\bz_K, s_K)$ centered at~$(\bz_K, s_K) \in K$ and
with radius~$\rho \hK$.
\item {\bf Local quasi-uniformity in space}: there exists a number~$\mathsf{lqu}(\Th)>0$ such that~$h_{\Kx^1}\le h_{\Kx^2}\, \mathsf{lqu}(\Th)$ for all~$K^1 = \Kx^1 \times \Kt^1 ,K^2 = \Kx^2 \times \Kt^2\in\Th$ such that~$K^1\cap K^2$ 
has positive~$d$-dimensional measure.
\end{itemize}

The proof of Proposition \ref{PROP::DGp-BOUND} is a direct consequence of a collection of trace inequalities 
(see~\cite[Theorem 1.6.6]{Brenner_Scott_2007} and \cite[Lemma 2]{Moiola_Perugia_2018}), which
in our space--time setting can be written
for any element~$K = \Kx \times  \Kt \in \Th$
as
\begin{align}
\nonumber
& \Norm{\varphi}{L^2(\Kx \times \partial \Kt)}^2
\leq 
C_{\tr}\left( \hKt^{-1} \Norm{\varphi}{L^2(K)}^2 + \hKt
\Norm{\partial_t \varphi}{L^2(K)} ^2\right)  \qquad \qquad \forall\varphi \in \ESOBOLEV{1}{\Kt; L^2(\Kx)}, 
\\
\label{EQN::TRACE-INEQUALITIES}
& \Norm{\varphi}{L^2(\partial \Kx \times \Kt)}^2
\leq C_{\tr} \left(\hKx^{-1} \Norm{\varphi}{L^2(K)}^2 + \hKx
\Norm{\nablax \varphi}{L^2(K)^d}^2 \right) 
 \quad \qquad \forall \varphi \in L^2\left(\Kt; \ESOBOLEV{1}{\Kx}\right),
\\
\nonumber
& \Norm{\nablax \varphi}{L^2(\partial \Kx \times \Kt)^d}^2 
\leq C_{\tr} \left(\hKx^{-1} \Norm{\nablax \varphi}{L^2(K)^d}^2 + \hKx
\Norm{D_{\bx}^2 \varphi}{L^2(K)^{d\times d}}^2 \right)
\quad \forall \varphi \in L^2\left(\Kt; \ESOBOLEV{2}{\Kx}\right),
\end{align}
where~$D_{\bx}^2 \varphi$ is the spatial Hessian of~$\varphi$, and~$C_{\tr}\ge1$
only depends on the star-shapedness parameter~$\rho$.

\begin{proposition}\label{PROP::DGp-BOUND}
Fix~$\delta = \min(\vartheta \hKx, \frac12)$, and assume that~$V \in L^\infty(K), \ \forall K \in \Th$. For all~$\varphi\in\bV(\Th)$,
the following bound holds
\begin{align*}
\Tnorm{\varphi}{\DGp}^2 & \leq  \frac{3}{2} C_{\tr}
  \sum_{K = \Kx \times \Kt \in \Th} \Bigg[
 \hKt^{-1} \Norm{\varphi}{L^2(K)}^2 + \hKt
\Norm{\partial_t \varphi}{L^2(K)}^2 + \mathrm{a}_K^2 \hKx^{-1}\Norm{\varphi}{L^2(K)}^2 \\
& \quad\qquad + \big( \mathrm{a}^2_K \hKx +\mathrm{b}_K^2 \hKx^{-1} \big)
\Norm{ \nablax \varphi}{L^2(K)^d}^2
+ \mathrm{b}_K^2 \hKx \Norm{D_{\bx}^2 \varphi}{L^2(K)^{d\times d}}^2 +  \Norm{\mu^{\frac12} \partial_t \varphi}{L^2(K)}^2\\
& \qquad\quad  + \Norm{\mu^{\frac12}\Deltax \varphi}{L^2(K)}^2 + \Norm{V}{L^\infty(K)}^2 \Norm{\mu^{\frac12} \varphi}{L^2(K)}^2 + \Norm{\mu^{-\frac12}{\varphi}}{L^2(K)}^2\Bigg],
\end{align*}
where
\begin{align*}
\mathrm{a}_K^2 := &\max\Bigg\{
\underset{\partial K \cap \left(\Ftime \cup \FD\right)}{\esssup} \alpha,\;\;
\bigg(\underset{\partial K \cap (\Ftime \cup \FN)}{\essinf}\beta\bigg)^{-1},\;\;
\underset{\partial K \cap \FR}{\esssup\;}\vartheta
\Bigg\},
\\
\mathrm{b}_K^2 := &\max\Bigg\{
\bigg(\underset{\partial K \cap \left(\Ftime \cup \FD\right)}{\essinf} \alpha\bigg)^{-1},\;\;
\underset{\partial K \cap (\Ftime \cup \FN)}{\esssup}\beta,\;\;
\hKx
\Bigg\}.
\end{align*}
\end{proposition}

The factor~$\frac{3}{2} C_{\tr}$ appearing in the bound of Proposition \ref{PROP::DGp-BOUND} is due to the 
integral terms with arguments 
$\frac{1}{2}\alpha\abs{\jmp{w}_\bN}{}^2$, $\frac{1}{2}\beta^{-1}\abs{\mvl{w}}^2$ on $\Ftime$ in the definition 
\eqref{EQN::DG-NORM} of the $\Tnorm{\cdot}\DG$ norm.
The volume term $\Norm{\mu^{\frac12}
\calS w
}{L^2(K)}^2$ is controlled by the inequality $\abs{\by}_1 \leq \sqrt{n}\abs{\by}_2$, $\forall \by \in \IC^n$.
\begin{remark}[Inhomogeneous Schr\"odinger equation]
The space--time ultra-weak DG variational formulation in~\eqref{EQN::VARIATIONAL-DG} can be easily extended to approximate the solution to inhomogeneous Schr\"odinger-type problems with a sufficiently smooth term~$f : \QT \rightarrow \IC$ at the right-hand side of the first equation in~\eqref{EQN::SCHRODINGER-EQUATION}; see~\cite[Ch.~3, \S~10]  {Lions_Magenes_1972} for the well-posedness of such problems. In order to preserve the  consistency of the method, it is necessary to add the following term to the antilinear functional~$\ell(\cdot)$:
\begin{equation*}
\sum_{K \in \Th} \int_K f\left(\cs + i \mu \Conjugate{\calS \shp}\right) \dV.
\end{equation*}
The existence and uniqueness of the discrete solution for any choice of the discrete space~$\bVhp(\Th)$, as well as the quasi-optimality estimate~\eqref{EQN::QUASI-OPTIMALITY}, follow from the coercivity and continuity of the sesquilinear form~$\cbA{\cdot}{\cdot}$ on the continuous space~$\bV(\Th)$ in Propositions~\ref{PROP::COERCIVITY} and~\ref{PROP::CONTINUITY}, together with the consistency of the method. Thus, optimal convergence rates can be proven for the full polynomial space as in Section~\ref{SECT::ERROR-ESTIMATE-FULL-POLYNOMIAL}, since this space provides a good enough approximation of any sufficiently smooth solution. On the other hand, the quasi-Trefftz space introduced in Section~\ref{SECT::ERROR-ESTIMATE-QUASI-TREFFTZ} would require some adjustments in order to approximate the solution of an inhomogeneous problem.
\end{remark}

\begin{remark}[Energy dissipation]\label{Rem:Energy}
It is well known that the Schr\"odinger equation~\eqref{EQN::SCHRODINGER-EQUATION} with
homogeneous Dirichlet and/or Neumann boundary conditions and~$\GR = \emptyset$
preserves the energy (or probability) functional~$\calE(t; 
\psi):=\frac12\int_{\Omega} |\psi(\bx, t)|^2 \dx$, i.e.\ $\dfhdt{}\calE(t; \psi) =0$.

The proposed DG method is dissipative, but the energy loss can be quantified in terms of the local least-squares error, the initial condition error, the jumps of the solution on the mesh skeleton, and the error on~$\FD \cup \FN$ due to the weak imposition of the boundary conditions. More precisely, for~$\gD=0$, $\gN = 0$ and~$\FR = \emptyset$, the discrete solution to~\eqref{EQN::VARIATIONAL-DG} satisfies
\begin{align*}
& \calE(0;\psi_0) - \calE(T; \uhp) = \
\calE_{loss} : =  \ \delta_{\calE} + \frac{1}{2}\Norm{\psi_0 - 
\uhp}{\FO}^2, 
\end{align*}
where
\begin{align*}
\delta_{\calE} & : = \,\sum_{K \in \Th} \Norm{\mu^{\frac12}
\calS \uhp}{L^2(K)}^2 + \frac{1}{2} \Norm{\jump{\uhp}_t}{L^2(\Fspa)}^2 + 
\frac{1}{2} \Big(\Norm{\alpha^{\frac12} \uhp}{L^2(\FD)}^2  \\
& \quad  +  \Norm{\beta^{\frac12} \normalDer \uhp }{L^2(\FN)}^2  + \Norm{\alpha^{\frac12} \jump{\uhp}_{\bN}}{L^2(\Ftime)^d}^2 + \Norm{\beta^{\frac12} 
\jump{\nablax 
\uhp}_{\bN}}{L^2(\Ftime)}^2 \Big).
\end{align*}
This follows from the definition of the~$\Tnorm{\cdot}\DG$ norm of the solution~$\uhp$, the coercivity of the sesquilinear form~$\cbA{\cdot\,}{\!\cdot}$, the 
definition of the antilinear functional~$\ell(\cdot)$ and simple algebraic manipulations; see~\cite[Rem.~3.6]{Gomez_Moiola_2022}.
\end{remark}

\section{Discrete spaces and error estimates}\label{SECT::ERROR-ESTIMATE}
\noindent In this section we prove \emph{a priori}~$h$-convergence estimates on the~$\Tnorm{\cdot}{\DGp}$ norm of the error for some discrete polynomial spaces.
In particular, for each element~$K \in \Th$, we consider two different polynomial spaces: the space $\IP^p(K)$ of polynomials of degree $p$ on $K$, and a quasi-Trefftz subspace~$\QTrefftz{p}{K} \subset \IP^p(K)$ with much smaller dimension, i.e.,~$\dim(\QTrefftz{p}{K}) \ll \dim(\IP^p(K))$ (see Proposition~\ref{PROP::BASIS-DIMENSION-QUASI-TREFFTZ} below). A polynomial Trefftz space for the case of zero potential~$V$ has been studied in~\cite{Gomez_Moiola_Perugia_Stocker_2023}.
We denote the local dimensions~$n_{d+1,p} := \dim(\QTrefftz{p}{K})$ and~$r_{d+1,p} := \dim(\IP^p(K))$ in dependence of the space dimension~$d$ of the problem and the polynomial degree~$p$, but independent of the element~$K$.
For simplicity, we only describe the case where the same polynomial degree is chosen in every element; the general case can easily be studied.

\subsection{Multi-index notation and preliminary results\label{SUBSECTION::PRELIMINARS}}
We use the standard multi-index notation for partial derivatives and monomials, adapted to the space--time setting: for~$\bj = 
(\bjx, j_t) = \left(j_{x_1}, \ldots, j_{x_d}, j_t\right)\in \IN_0^{d+1},$
\begin{align*}
\bj!&:= j_{x_1}!\cdots j_{x_d}!j_t!,
&\abs{\bj} &:= \abs{\bjx} + j_t := j_{x_1} + \cdots + j_{x_d} + j_t,\\
\DerA{\bj}{f} &:= \PartDer{j_{x_1}}{x_1}\cdots 
\PartDer{j_{x_d}}{x_d}\PartDer{j_t}{t} f,
&\bx^{\bjx}t^{j_t} &:= x_1^{j_{x_1}}\cdots x_d^{j_{x_d}}t^{j_t}.
\end{align*}

We also recall the definition and approximation properties of multivariate Taylor polynomials, which constitute the basis of our error analysis. 
On an open and bounded set~$\Upsilon \subset\IR^{d+1}$,
the Taylor polynomial of order~$m\in\IN$ (and degree~$m - 1$), centered at~$(\bz,s)\in \Upsilon$,
of a function~$\varphi \in \EFC{m - 1}{\Upsilon}$ is defined as
\begin{equation*}
\Taylor{(\bz,s)}{m}{\varphi}(\bx, t) := \sum_{\abs{\bj} < m} 
\frac{1}{\bj!} \DerA{\bj}{\varphi}(\bz, s)(\bx - \bz)^{\bjx} (t - s)^{j_t}. 
\end{equation*}
If~$\varphi \in \EFC{m}{\Upsilon}$ and the segment 
$[(\bz,s),(\bx,t)]\subset \Upsilon$, 
the Lagrange's form of the Taylor remainder (see~\cite[Corollary 3.19]{Callahan_2010}) is bounded as follows:
\begin{align*}
\label{EQN::TAYLOR-REMAINDER}
\abs{\varphi(\bx,t) - \Taylor{(\bz,s)}{m}{\varphi}(\bx,t)} 
& \leq  \abs{\varphi}_{{\EFC{m}{\Upsilon}}} \sum_{\abs{\bj} = m} \frac{1}{\bj!} {\left|\left(\bx - 
\bz\right)^{\bjx} (t - s)^{j_t}\right|}
 \leq \frac{(d + 1)^{\frac m2}}{m!}
h_\Upsilon^{m} 
\abs{\varphi}_{\EFC{m}{\Upsilon}},
\end{align*}
where~$h_\Upsilon$ is the diameter of~$\Upsilon$.
In particular, if $\Upsilon$ is star-shaped with respect to~$(\bz, s)$, 
then the following estimate is obtained
\begin{equation*}
\label{EQN::ESTIMATE-TAYLOR}
  \Norm{\varphi(\bx,t) - \Taylor{(\bz,s)}{m}{\varphi}(\bx,t)}{L^2(\Upsilon)} \leq \frac{(d + 1)^{\frac{m}{2}} \abs{\Upsilon}^{\frac12}}{m!}
h_\Upsilon^{m} 
\abs{\varphi}_{\EFC{m}{\Upsilon}},
\end{equation*}
which, together with the well-known identity  (see~\cite[Prop.~(4.1.17)]{Brenner_Scott_2007})
$D^\bj\Taylor{(\bz,s)}{m}{\varphi} = 
\Taylor{(\bz,s)}{m-|\bj|}{D^\bj\varphi}, |\bj|<m$,
gives the
estimate
\begin{equation}
\label{EQN::ESTIMATE-TAYLOR-DERIVATIVES}
\abs{\varphi - \Taylor{(\bz, s)}{m}{\varphi}}_{\ESOBOLEV{r}{\Upsilon}}
\leq \binom{d + r}{d}^\frac12 \frac{(d+1)^{\frac{m - r}2} \abs{\Upsilon}^{\frac12}}{(m - r)!} h_\Upsilon^{m - r} \abs{\varphi}_{\EFC{m}{\Upsilon}{}} \ \  \ r < m,\ \forall \varphi \in \EFC{m}{\Upsilon}{}.
\end{equation} 
The Bramble--Hilbert lemma provides an estimate for the error of the 
averaged Taylor polynomial, see~\cite{Duran_1983} 
and~\cite[Thm.~4.3.8]{Brenner_Scott_2007}.

\begin{lemma}[Bramble--Hilbert]
\label{LEMMA::Bramble--HILBERT}
Let~$\Upsilon \subset \IR^{d + 1}$, $1 \leq 
d \in \IN$, be an open and bounded set with diameter $h_\Upsilon$, star-shaped with 
respect to the ball $B:= B_{\rho h_\Upsilon}(\bz, s)$ centered at $(\bz, s) \in \Upsilon$ and 
with radius $\rho h_\Upsilon$, for some $0 < \rho \leq \frac{1}{2}$. If~$\varphi \in \ESOBOLEV{m}{\Upsilon}$, the \emph{averaged Taylor polynomial} of order~$m$ (and degree $m
- 1$) defined as
\begin{equation*}
\label{EQN::AVERAGED-TAYLOR-POLYNOMIAL}
\AvTaylor{m}{\varphi}(\bx, t) := 
\frac{1}{\abs{B}} \int_{B} \Taylor{(\bz,s)}{m}{\varphi}(\bx, t) \dV(\bz,s),
\end{equation*}
satisfies the following error bound for all~$s < m$
\begin{equation*}
\label{EQN::Bramble--HILBERT-LEMMA}
\abs{\varphi - \AvTaylor{m}{\varphi}}_{H^s(\Upsilon)} \leq 
C_{d,m, \rho} \:
h_\Upsilon^{m-s} \abs{\varphi}_{\ESOBOLEV{m}{\Upsilon}} \leq 2 \binom{d + s}d \frac{(d + 1)^{m - s}}{(m - s-1)!}\frac{h_{\Upsilon}^{m - s}}{\rho^{\frac{d+1}{2}}} \abs{\varphi}_{\ESOBOLEV{m}{\Upsilon}}.
\end{equation*}
A sharp bound on~$C_{d,m,\rho}>0$ is given in \cite[p.~986]{Duran_1983} in dependence of $d$, $s$, $m$ and $\rho$, and the second bound is proven in \cite[Lemma 1]{Moiola_Perugia_2018}.
\end{lemma}
\subsection{Full polynomial space}
\label{SECT::ERROR-ESTIMATE-FULL-POLYNOMIAL}
\noindent In next theorem, we derive \emph{a priori} error estimates for the DG formulation \eqref{EQN::VARIATIONAL-DG} for the space of elementwise polynomials
\begin{equation}
\label{EQN::POLYNOMIAL-SPACE}
\bVhp(\Th) = \prod_{K \in \Th} \IP^p(K).
\end{equation}
\begin{theorem}
\label{THM::ERROR-ESTIMATE-FULL-POLYNOMIALS}
Let $p\in\IN$, fix $\delta$ as in Proposition~\ref{PROP::DGp-BOUND} and assume that $V \in L^{\infty}(\QT)$. Let $\psi \in \bV(\Th) 
\cap \ESOBOLEV{p + 1}{\Th}$ be the exact solution of \eqref{EQN::SCHRODINGER-EQUATION} and $\uhp\in \bVhp(\Th)$ be the solution to the variational formulation \eqref{EQN::VARIATIONAL-DG} with $\bVhp(\Th)$ given by \eqref{EQN::POLYNOMIAL-SPACE}.
Set the volume penalty function and the stabilization functions as 
\begin{equation*}
\min\left\{\hKt^2, \hKx^2\right\} \leq \mu|_K \leq \max\left\{\hKt^2, \hKx^2\right\},
\end{equation*}
\begin{equation*}
\alpha|_{F} = \frac1{h_{F_\bx}}
\;\; \forall F \subset \Ftime\cup\FD, 
\quad \qquad  \beta|_{F} = h_{F_\bx} \;\; \forall F \subset \Ftime \cup \FN,
\end{equation*}
where
\begin{equation*}
\begin{cases}
h_{F_\bx}=h_{\Kx} & \text{if } F\subset \deK\cap \left(\FD \cup \FN\right), \\
\min\{h_{\Kx^1},h_{\Kx^2}\}\le h_{F_\bx}\le
\max\{h_{\Kx^1},h_{\Kx^2}\} & \text{if } F=K^1 \cap K^2 \subset \Ftime,\\
\end{cases}
\end{equation*} 
then the following estimate holds
\begin{align*}
\nonumber
\Tnorm{\psi - \uhp}{\DG}  &  \leq 3\sqrt{6 C_{\tr}} \rho^{-\frac{p + 1}{2}}\frac{(d+1)^{p + 1}}{p!} 
\sum_{K=\Kx\times \Kt\in \Th} \Bigg[ \hKt^{-\frac12} \hK^{p + 1}  
 \\
\nonumber
& \quad +  p \hKt^{\frac12} \hK^{p} + \mathsf{lqu}(\Th) \left(
\hKx^{-1}\hK^{p+1} + 2 p \hK^p +  \frac{(p - 1)p}{2} \left(\frac{d + 2}{d + 1}\right) \hKx \hK^{p - 1}\right) 
\\ 
& \quad 
+
p \max\{\hKx, \hKt\} \hK^{p} + 
 \frac{(p - 1)p}{2}  \left(\frac{d + 2}{d + 1}\right)\max\{\hKx, \hKt\} \hK^{p - 1} \\
& \quad + \Norm{V}{L^{\infty}(K)} \max\{\hKx, \hKt\} \hK^{p + 1}    + 
\min\{\hKx^{-1}, \hKt^{-1}\} \hK^{p + 1}  \Bigg] \abs{\psi}_{\ESOBOLEV{p+1}{K}}.
\end{align*}
Moreover, if~$\hKx \simeq \hKt$ for all~$K \in \Th$, there exists a positive constant $C$ independent of the element sizes~$\hKx,\hKt$, 
but depending on the degree~$p$, the~$L^\infty(\QT)$ norm of~$V$, the trace inequality constant~$C_{\tr}$ in~\eqref{EQN::TRACE-INEQUALITIES}, the local quasi-uniformity parameter~$\mathsf{lqu}(\Th)$ and the star-shapedness parameter~$\rho$  such that
\begin{equation*}
\Tnorm{\psi - \uhp}{\DG} \leq C \sum_{K \in \Th} \hK^p \abs{\psi}_{\ESOBOLEV{p+1}{K}}.
\end{equation*}
\end{theorem}
\begin{proof}
The proof follows from the choice of the volume penalty function~$\mu$ and the stabilization functions~$\alpha, \ \beta$, the quasi-optimality bound~\eqref{EQN::QUASI-OPTIMALITY}, Proposition~\ref{PROP::DGp-BOUND}, the inequality~$\sqrt{\abs{\bv}_1} \leq \sum_{i = 1}^N \sqrt{\abs{v_i}}\ \forall \bv \in \IR^N$, the fact that~$\AvTaylor{p+1}{\psi_{|_K}} \in \bVhp(K)$ for all elements~$K\in \Th$, and the Bramble-Hilbert lemma \ref{LEMMA::Bramble--HILBERT}. 
\end{proof}

\subsection{Quasi-Trefftz spaces} \label{SECT::ERROR-ESTIMATE-QUASI-TREFFTZ}
\noindent We now introduce a polynomial quasi-Trefftz space. Let~$p \in \IN$ and assume that~$V\in \EFC{p-2}{K}$. For each $K\in \Th $ we define the following local polynomial quasi-Trefftz space:
\begin{equation}
\label{EQN::DEFINITION-QUASI-TREFFTZ}
    \QTrefftz{p}{K} := \left\{q_p \in \IP^{p}(K) : \DerA{\bj}{ \calS q_p}(\bx_K, t_K) = 0, \ 
        \abs{\bj} \leq p - 2
    \right\},
\end{equation}
for some point $(\bx_K, t_K)$ in~$K$. We consider the following global discrete space
\begin{equation}
\label{EQN::GLOBAL-QUASI-TREFFTZ}
    \bVhp(\Th) = \prod_{K \in \Th} \QTrefftz{p}{K}.
\end{equation}
For all~$\bj \in \IN^{d + 1}$, if~$V \in \EFC{\abs{\bj}}{K}$ and~$f \in \EFC{\abs{\bj} + 2}{K}$, then by the multi-index Leibniz product rule for multivariate functions we have
\begin{equation}
\begin{split}
\label{EQN::IDENTITY-DERIVATIVES-SCHRODINGER}
\DerA{\bj}\calS f(\bx_K, t_K) = &  i \DerA{\bjx, j_t + 1}{} f(\bx_K, t_K) + \frac{1}{2} \sum_{\ell = 1}^d \DerA{\bjx + 2\boldsymbol{e}_{\ell}, j_t}{} f(\bx_K, t_K) \\
& - \sum_{\bz \leq \bj} \binom\bj\bz \DerA{\bj -  \bz} V(\bx_K, t_K) \DerA{\bz} f(\bx_K, t_K), 
\end{split}
\end{equation}
where $\left\{\boldsymbol{e}_\ell\right\}_{\ell = 1}^d \subset \IR^d \text{ is the canonical basis,}$
\begin{equation*}
\begin{split}
\binom\bj\bz
= \frac{\bj!}{\bz! (\bj - \bz)!}, \quad \text{ and } \quad \bj \leq \bz \Leftrightarrow j_{x_i} \leq z_{x_i} \ (1 \leq i \leq d) \text{ and } j_t \leq z_t.
\end{split}
\end{equation*}

The next proposition is the key ingredient to prove optimal convergence rates in Theorem~\ref{THM::ERROR-ESTIMATE-QUASI-TREFFTZ} for the DG method \eqref{EQN::VARIATIONAL-DG} when~$\bVhp(\Th)$ is chosen as the quasi-Trefftz polynomial space defined in~\eqref{EQN::DEFINITION-QUASI-TREFFTZ}. 
\begin{proposition}
\label{PROP::TAYLOR-QUASI-TREFFTZ}
Let $p \in \IN$ and $K \in \Th$. Assume that $V \in \EFC{\max\{p - 2,0\}}{K}$ and $\psi \in \EFC{p}{K}$ satisfies $\calS \psi = 0$ in $K$, then the Taylor polynomial~$\Taylor{(\bx_K, t_K)}{p+1}{\psi} \in \QTrefftz{p}{K}$.
\end{proposition}
\begin{proof}
By the definition of the Taylor polynomial,  $\Taylor{(\bx_K, t_K)}{p+1}{\psi} \in \IP^{p}(K)$.
Therefore, it only remains to show that~$\DerA{\bj}{ \calS \Taylor{(\bx_K, t_K)}{p+1}{\psi}}(\bx_K, t_K) = 0$ for all $\abs{\bj} \leq p - 2$.
Taking~$f = \Taylor{(\bx_K, t_K)}{p+1}{\psi}$ in~\eqref{EQN::IDENTITY-DERIVATIVES-SCHRODINGER}, all the derivatives of~$\Taylor{(\bx_K, t_K)}{p+1}{\psi}$ at~$(\bx_K, t_K)$ that appear in \eqref{EQN::IDENTITY-DERIVATIVES-SCHRODINGER} are at most of total order~$\abs{\bj} + 2 \leq p$, so 
they coincide with the corresponding derivatives of $\psi$.
Furthermore, since~$\calS \psi = 0$, then
$$\DerA{\bj}{ \calS \Taylor{(\bx_K, t_K)}{p+1}{\psi}}(\bx_K, t_K) = \DerA{\bj}{ \calS \psi}(\bx_K, t_K) = 0,$$
which completes the proof.
\end{proof}

Proposition~\ref{PROP::TAYLOR-QUASI-TREFFTZ} allows for the use of the Taylor error bound \eqref{EQN::ESTIMATE-TAYLOR-DERIVATIVES} in the analysis of the quasi-Trefftz DG scheme.

\begin{theorem}
\label{THM::ERROR-ESTIMATE-QUASI-TREFFTZ}
Let~$p\in\IN$, fix~$\delta$ as in Proposition~\ref{PROP::DGp-BOUND} and assume that~$V \in L^{\infty}(\QT)\cap\EFC{\max\{p - 2,  0\}}{\Th}{}$. Let~$\psi \in \bV(\Th) 
\cap \EFC{p + 1}{\Th}$ be the exact solution of~\eqref{EQN::SCHRODINGER-EQUATION} and~$\uhp\in \bVhp(\Th)$ be the solution to the variational formulation \eqref{EQN::VARIATIONAL-DG} with~$\bVhp(\Th)$ given by~\eqref{EQN::GLOBAL-QUASI-TREFFTZ}.
Set the volume penalty function $\mu$ and the stabilization functions~$\alpha,\beta$ as in Theorem \ref{THM::ERROR-ESTIMATE-FULL-POLYNOMIALS}.
Then, the following estimate holds
\begin{align*} 
\nonumber
\Tnorm{\psi - \uhp}{\DG} &  \leq  \frac{3}{2} \sqrt{6C_{\tr}} \abs{\QT}^{\frac12} \frac{(d+1)^{\frac{p + 1}{2}}}{(p + 1)!} 
\sum_{K=\Kx\times \Kt\in \Th} \Bigg[\hKt^{-\frac12} \hK^{p + 1} + (p + 1) \hKt^{\frac12} \hK^{p}\\
\nonumber
& \quad + \mathsf{lqu}(\Th) \left(
\hKx^{-1}\hK^{p+1} + 2(p + 1) \hK^p  +  p(p+1) \left(\frac{d+2}{2(d+1)}\right)^{\frac12} \hKx \hK^{p - 1}\right) \\ 
  \nonumber
&\quad  + 
(p + 1) \max\{\hKx, \hKt\} \hK^{p}
  +  p(p + 1) \left(\frac{d+2}{2(d+1)}\right)^{\frac12} \max\{\hKx, \hKt\} \hK^{p - 1} \\
& \quad +  \Norm{V}{\EFC{0}{K}} \max\{\hKx, \hKt\} \hK^{p + 1}  + 
\min\{\hKx^{-1}, \hKt^{-1}\} \hK^{p + 1} \Bigg] \abs{\psi}_{\EFC{p+1}{K}}.
\end{align*}
Moreover, if~$\hKx \simeq \hKt$ for all~$K \in \Th$, there exists a positive constant $C$ independent of the mesh size~$h$, but depending on the degree~$p$,  the~$L^\infty(\QT)$ norm of~$V$, the trace inequality constant~$C_{\tr}$ in~\eqref{EQN::TRACE-INEQUALITIES}, the local quasi-uniformity parameter~$\mathsf{lqu}(\Th)$ and the measure of the space--time domain $\QT$ such that
\begin{equation*}
\Tnorm{\psi - \uhp}{\DG} \leq C \sum_{K \in \Th} \hK^{p} \abs{\psi}_{\EFC{p+1}{K}}.
\end{equation*}
\end{theorem}
\begin{proof}
The proof follows from the choice of the volume penalty function~$\mu$ and the stabilization functions~$\alpha, \ \beta$, the quasi-optimality bound~\eqref{EQN::QUASI-OPTIMALITY}, bound~\eqref{PROP::DGp-BOUND}, the inequality~$\sqrt{\abs{\bv}_1} \leq \sum_{i = 1}^N \sqrt{\abs{v_i}}\ \forall \bv \in \IR^N$, Proposition~\ref{PROP::TAYLOR-QUASI-TREFFTZ}, and the estimate~\eqref{EQN::ESTIMATE-TAYLOR-DERIVATIVES}. 
\end{proof}

The \emph{a priori} error estimate in Theorem~\ref{THM::ERROR-ESTIMATE-QUASI-TREFFTZ} requires stronger regularity assumptions on~$\psi$ than Theorem~\ref{THM::ERROR-ESTIMATE-FULL-POLYNOMIALS} (namely $\psi\in \EFC{p+1}\Th$ instead of $\psi\in H^{p+1}(\Th)$)
due to the fact that~$\QTrefftz{p}{K}$ is tailored to contain the Taylor polynomial~$\Taylor{(\bx_K, t_K)}{p + 1}{\psi}$, but in general it does not contain the averaged Taylor polynomial~$\AvTaylor{p + 1}{\psi}$.

\begin{remark}[Non-polynomial spaces]
Optimal $h$-convergence estimates can also be derived for non-polynomial spaces, by requiring the local space~$\bVhp(K)$ to contain an element whose Taylor polynomial coincides with that of the exact solution. This is the approach in~\cite{Gomez_Moiola_2022} for the Trefftz space of complex exponential wave functions for the Schr\"odinger equation with piecewise-constant potential.
\end{remark}

\subsubsection{Basis functions and dimension \label{SECT::BASIS}}
\noindent So far, we have not specified the dimension and a basis for the space~$\QTrefftz{p}{K}$, which is the aim of this section.

Recalling that~$r_{d, p} = \dim\big(\IP^p(\IR^d)\big)=\binom{p+d}d$, let~$\{\widehat{m}_\alpha\}_{\alpha = 1}^{r_{d, p}}$ and~$\{\widetilde{m}_\beta\}_{\beta = 1}^{r_{d, p- 1}}$ be bases of~$\IP_p(\IR^{d})$ and~$\IP_{p-1}(\IR^{d})$, respectively. We define 
$$
n_{d+1, p} := r_{d,p}+r_{d,p-1}
= \binom{p + d}d + \binom{p + d - 1}d = \frac{(p + d - 1)! (2p + d)}{d! p!},
$$
and the following~$n_{d+1, p}$ elements of~$\QTrefftz{p}{K}$
\begin{equation}
\left\{b_J \in \QTrefftz{p}{K} : 
\begin{cases}
    b_J\left(\bx_K^{(1)}, \cdot\right) = \widehat{m}_J \text{ and } \partial_{x_1} b_J\left(\bx_K^{(1)}, \cdot\right) = 0 & \!\!\!\!\text{ if }J \leq r_{d, p} \\
    b_J\left(\bx_K^{(1)}, \cdot\right) = 0 \text{ and } \partial_{x_1} b_J\left(\bx_K^{(1)}, \cdot\right) = \widetilde{m}_{J-r_{d, p}} & \!\!\!\! \text{ if } r_{d, p} < J \leq n_{d+1,p}\\
\end{cases}
\right\},
\label{DEF::QUASI-TREFFTZ-BASIS}
\end{equation}
where~$g\left(\bx_K^{(1)}, \cdot\right)$ denotes the restriction of~$g: K \rightarrow \IC$ to~$x_1 = \bx_K^{(1)}$, where~$\bx_K^{(1)}$ is the first component of~$\bx_K \in \IR^d$.

Any element~$q_p \in \QTrefftz{p}{K}$ can be expressed in the scaled monomial basis as
\begin{equation*}
\label{EQN::MONOMIAL-EXPRESSION-BJ}
    q_p(\bx, t) = \sum_{\abs{\bj} \leq p} C_\bj \left(\frac{\bx - \bx_K}{h_K}\right)^{\bjx} \left(\frac{t - t_K}{h_K}\right)^{j_t},
\end{equation*}
for some complex coefficients~$\left\{C_\bj\right\}_{\abs{\bj} \leq p}$.
By the conditions~$\DerA{\bj} \calS q_p (\bx_K, t_K) = 0$ for all~$\abs{\bj} \leq p -2$, in the definition of~$\QTrefftz{p}{K}$, we have the following relations between the coefficients
\begin{equation*}
\begin{split}
    \frac{i}{h_K} (j_t + 1) C_{\bjx, j_t + 1} & +  \frac{1}{2h_K^2} \sum_{\ell = 1}^d (\bjxl + 1)(\bjxl + 2) C_{\bjx + 2\boldsymbol{e}_\ell, j_t}^J  - \sum_{\bz \leq \bj} \frac{h_K^{\abs{\bj} - \abs{\bz}}}{(\bj - \bz)!} \DerA{\bj - \bz} V(\bx_K, t_K) C_{\bz}^J = 0,
\end{split}
\end{equation*}
which can be rewritten as
\begin{align}
\label{EQN::TAYLOR-COEFFICIENTS-RELATION}
    C_{\bjx + 2\boldsymbol{e}_1, j_t} = &  
    \frac{1}{(\bjxo + 1)(\bjxo + 2)}\Bigg( - 2ih_K (j_t + 1) C_{\bjx, j_t + 1}^J  \\
\nonumber
    & -  \sum_{\ell = 2}^d (\bjxl + 1)(\bjxl + 2) C_{\bjx + 2\boldsymbol{e}_\ell, j_t}^J + 2\sum_{\bz \leq \bj} \frac{h_K^{\abs{\bj} - \abs{\bz} + 2}}{(\bj - \bz)!} \DerA{\bj - \bz} V(\bx_K, t_K) C_{\bz}^J\Bigg).
\end{align}
The conditions imposed in~\eqref{DEF::QUASI-TREFFTZ-BASIS} on the restriction of~$b_J$ to~$x_1 = \bx_K^{(1)}$ fix the coefficients of their expansion for all~$\bj$ with~$j_{x_1} \in \{0, 1\}$.
In Figures \ref{FIG::INDEPENDENT-RELATIONS} and \ref{FIG::Relations2+1D}, we illustrate how the
coefficients that are not immediately determined by the conditions in~\eqref{DEF::QUASI-TREFFTZ-BASIS}
(i.e., those for~$\bj_{x_1} \geq 2$) are uniquely defined and can be computed for the~$(1+1)$- and~$(2+1)$-dimensional cases using the recurrence relation \eqref{EQN::TAYLOR-COEFFICIENTS-RELATION}.

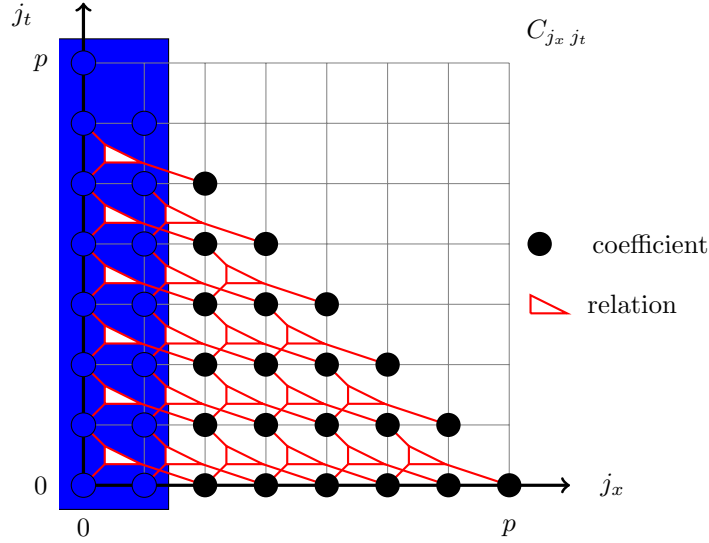
\begin{figure}[htb]
\centering
\begin{tikzpicture}[scale=.8]
\draw[fill = blue, opacity = 0.2] (-.4,-.4)--(1.4,-.4)--(1.4,7.4)--(-.4,7.4);
\draw[step=1cm,gray,very thin] (0,0) grid (7,7);
\draw[very thick,->](0,0)--(8,0);\draw[very thick,->](0,0)--(0,8);
\draw(8.7,0)node{$j_x$};\draw(-1,7.8)node{$j_t$};
\draw(-.7,7)node{$p$};\draw(-.7,0)node{$0$};
\draw(7,-.7)node{$p$};\draw(0,-.7)node{$0$};
\foreach \y in {1,...,6}\foreach \x in {\y,...,6} 
    { \draw[thick,red](6-\x,\y)--(6.5-\x,\y-.5)--(8-\x,\y-1);
    \draw[thick, red](6-\x,\y-1)--(6.5-\x,\y-.5);
    \draw[red,fill=white,thick](6.5-\x-.15,\y-.5-.15)--(6.5-\x+.45,\y-.5-.15)--(6.5-\x-.15,\y-.5+.15)--(6.5-\x-.15,\y-.5-.15);}
\draw(7.8,7.5)node{$C_{j_x\,j_t}$};
\foreach \x in {2,...,7}\foreach \y in {\x,...,7} { \fill (\x,7-\y)circle(.2); }
\foreach \x in {0,1}\foreach \y in {\x,...,7} { \draw[fill = blue] (\x,7-\y)circle(.2); }
\fill (7.5,4) circle(.2);
\draw(9.1,4)node{\quad coefficient};
\draw[red,fill=white,thick](7.5-.15,3-.15)--(7.5+.45,3-.15)--(7.5-.15,3+.15)--(7.5-.15,3-.15);
\draw(8.8,3)node{\quad relation};
\end{tikzpicture}
\caption[]{A representation of the relations defining the coefficients of $b_J$ for the (1+1)-dimensional case.
The colored dots in the $(j_x,j_t)$ plane represent the coefficients $C_{j_x\,j_t}$.
Each shape 
\begin{tikzpicture}[scale=.35]
\draw[red] (0,0)--(.5,.5)--(2,0);
\draw[red] (0,1)--(.5,.5);
\fill (0,0)circle(.2); \fill (2,0)circle(.2); \fill (0,1)circle(.2); 
\draw[red,fill=white](.5-.15,.5-.15)--(.5+.45,.5-.15)--(.5-.15,.5+.15)--(.5-.15,.5-.15);
\end{tikzpicture}
connects three dots located at the points $(j_x,j_t+1)$, $(j_x,j_t)$ and $(j_x+2,j_t)$: this shape represents one of the equations \eqref{EQN::TAYLOR-COEFFICIENTS-RELATION} which, given $C_{j_x(j_t+1)}$ and $C_{j_x\,j_t}$, allows to compute $C_{(j_x+2)j_t}$.
If the $2p+1$ values with $j_x\in\{0,1\}$ (corresponding to the blue nodes in the 
shaded region) are given, then these relations uniquely determine all the other 
coefficients, which can be computed sequentially using the relations \eqref{EQN::TAYLOR-COEFFICIENTS-RELATION} by proceeding left to right in  
the diagram. 
In the figure $p=7$, the number of nodes is $r_{2,p}=36$, the number of nodes in the 
shaded region is $n_{2,p}=15$, the number of relations is $r_{2,p}-n_{2,p}=21$.
\label{FIG::INDEPENDENT-RELATIONS}}
\end{figure}

\begin{figure}[htb]
\centering
\begin{tikzpicture}
[rotate around x=-90,rotate around y=0,rotate around z=-125,grid/.style={very thin,gray},scale=.9]

\def\maxX{5}
\foreach \x in {0,1,...,\maxX}
\foreach \y in {0,1,...,\maxX}
\foreach \z in {0,1,...,\maxX}
    {
    \draw[grid,lightgray] (\x,0,0) -- (\x,\maxX,0);
    \draw[grid,lightgray] (0,\y,0) -- (\maxX,\y,0);
    \draw[grid,lightgray] (0,\y,0) -- (0,\y,\maxX);
    \draw[grid,lightgray] (0,0,\z) -- (0,\maxX,\z);
    \draw[grid,lightgray] (\x,0,0) -- (\x,0,\maxX);
    \draw[grid,lightgray] (0,0,\z) -- (\maxX,0,\z);
    }

\foreach \x in {0,...,\maxX} 
    {
    \pgfmathsetmacro\xx{\maxX-\x}  
    \foreach \y in {0,...,\xx}
        {
        \pgfmathsetmacro\xxy{\maxX-\x-\y}  
        \foreach \z in {0,...,\xxy}
            {  
            \ifnum \x>0 \relax  \ifnum \y>0 \relax  \ifnum \z>0 \relax
                \draw [lightgray,dashed] (\x,\y,0)--(\x,\y,\z);
                \draw [lightgray,dashed] (\x,0,\z)--(\x,\y,\z);
                \draw [lightgray,dashed] (0,\y,\z)--(\x,\y,\z);
            \fi \fi \fi
            \ifnum \x<2\relax
                \draw[blue] plot [mark=*, mark size=2] coordinates{(\x,\y,\z)}; 
            \else
                \draw plot [mark=*, mark size=2] coordinates{(\x,\y,\z)}; 
            \fi
            }
        }
    }
    
\pgfmathsetmacro\maxXX{\maxX-2} 
\foreach \x in {0,...,\maxXX}
    {
    \pgfmathsetmacro\xx{\maxXX-\x}
    \foreach \y in {0,...,\xx}
        {
        \pgfmathsetmacro\xxy{\maxXX-\x-\y}  
        \foreach \z in {0,...,\xxy}
            {        
            \draw [very thick](\x,\y,\z)--(\x+.5,\y+.5,\z+.5)--(\x,\y,\z+1);
            \draw [very thick] (\x+2,\y,\z)--(\x+.5,\y+.5,\z+.5)--(\x,\y+2,\z);
            \draw [fill=white, thick] plot [mark=*, mark size=3] coordinates{(\x+.5,\y+.5,\z+.5)};      
            }
        }
    }
\draw(\maxX+.4,0,0)node{$j_x$};
\draw(0,\maxX+.4,0)node{$j_y$};
\draw(0,.5,\maxX+0.2)node{$j_t$};

\draw [red,very thick](0,1,2)--(.5,1.5,2.5)--(0,1,2+1);
\draw [red,very thick](2,1,2)--(0+.5,1.5,2.5)--(0,1+2,2);
\draw [red,fill=white,thick] plot [mark=*, mark size=3] coordinates{(.5,1.5,2.5)}; 
\end{tikzpicture}
\caption{A representation of the relations defining the coefficients of~$b_J$ for the (2+1)-dimensional case.
The colored dots in position~$\bj=(j_x,j_y,j_t)$, $|\bj|\le p$, correspond to the coefficients~$C_{j_x\,j_y\,j_t}$ (here~$p=5$ and~$r_p=56$).
Each white circle is connected by the segments to four nodes and represents one of the equations in~\eqref{EQN::TAYLOR-COEFFICIENTS-RELATION}: given~$C_{j_x\,j_y\,j_t}$, $C_{j_x\,j_y(j_t+1)}$ and~$C_{j_x(j_y+2)j_t}$, it allows to compute~$C_{(j_x+2)j_y\,j_t}$ (the leftmost of the four nodes connected to a given white circle) using~\eqref{EQN::TAYLOR-COEFFICIENTS-RELATION}.
The red dot exemplifies one of these relations, for~$\bj=(0,1,2)$.
Given the~$(p+1)^2$ coefficients with~$j_x\in\{0,1\}$ (the blue dots), all other coefficients are uniquely determined.
\label{FIG::Relations2+1D}}
\end{figure}
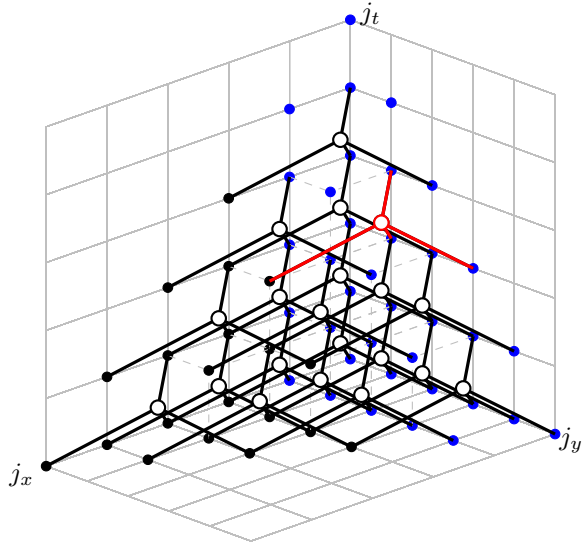

\begin{proposition} \label{PROP::BASIS-DIMENSION-QUASI-TREFFTZ} The set of functions~$\left\{b_J\right\}_{J = 1}^{n_{d + 1,p}}$ defined in~\eqref{DEF::QUASI-TREFFTZ-BASIS} are a basis for the space~$\QTrefftz{p}{K}$. Therefore, 
\begin{align*}
\dim\big(\QTrefftz{p}{K}\big) \!
= \! n_{d+1, p} 
\! & = \! \frac{(p + d - 1)! (2p + d)}{d! p!} \\
& 
\! = \! \mathcal{O}_{p\to\infty}(p^d) \! \ll \! \dim(\IP^p(K)) 
\! =\! \binom{d + 1+ p}{d + 1} \! = \! \mathcal{O}_{p\to\infty}(p^{d+1}).\!
\end{align*}

\end{proposition}
\begin{proof}
We first observe that the set of polynomials~$\left\{b_J\right\}_{J = 1}^{n_{d+1,p}}$ is linearly independent due to their restrictions to~$x_1 = \bx_{K}^{(1)}$.
On the other hand, the relations~\eqref{EQN::TAYLOR-COEFFICIENTS-RELATION}, imply that~$q_p$ is uniquely determined by its restriction~$q_p(\bx_K^{(1)}, \cdot)$ and the restriction of its derivative~$\partial_{x_1} q_p(\bx_K^{(1)}, \cdot)$.
In addition, there exist some complex coefficients~$\left\{\lambda_s\right\}_{s = 1}^{n_{d + 1, p}}$ such that
\begin{align*}
    q_p\left(\bx_K^{(1)}, \cdot\right)  
    & = \sum_{s = 1}^{r_{d, p}} \lambda_s \widehat{m}_s(\cdot) 
    = \sum_{s = 1}^{ r_{d, p}} \lambda_s b_s\left(\bx_K^{(1)}, \cdot\right),  \\
    \partial_{x_1} q_p\left(\bx_K^{(1)}, \cdot\right)  
    & = \!\!\sum_{s = r_{d, p} + 1}^{n_{d + 1, p}} \lambda_{s} \widetilde{m}_{s - r_{d, p}}(\cdot) 
    = \!\!\sum_{s = r_{d,p}+1}^{n_{d + 1, p}} \lambda_{s} \partial_{x_1} b_s(\bx_K^{(1)}, \cdot),
\end{align*}
whence
$q_p = \sum_{s = 1}^{n_{d + 1, p} } \lambda_s b_s$,
which completes the proof.
\end{proof}

\begin{remark}[Quasi-Trefftz basis construction: difference between Schr\"odinger and wave equations]
The definition of the basis functions~$b_J$ in~\eqref{DEF::QUASI-TREFFTZ-BASIS} can be modified by fixing the restriction of $b_J$ and its partial derivative~$\partial_{x_\ell} b_J$ to~$x_\ell = \bx_K^{(\ell)}$ for any~$1 \leq \ell \leq d$. However, it is not possible to assign the values for a given time $t=t_K$, as the order of the time derivative appearing in the Schr\"odinger equation is lower than the order of the space derivatives.
How this affects the basis construction is visible from Figure \ref{FIG::INDEPENDENT-RELATIONS}: 
the coefficients (the colored dots) can be computed sequentially when all the other coefficients of a relation (the Y-shaped stencil) are known, so it is possible to reach all dots moving left to right, but not moving bottom to top.
Imposing the values at a given time is possible for the wave equation, as it is done in \cite[\S4.4]{ImbertGerard_Moiola_Stocker}, precisely because in that case time and space derivatives have the same order.
\end{remark}

\begin{remark}[Constant-potential case]
 The space~$\QTrefftz{p}{K}$ does not reduce to a Trefftz space for the case of constant potential~$V$. Nonetheless, the pure Trefftz space~$\bVp{p}(K)$ defined as
\begin{equation*}
        \bVp{p}(K) = \left\{q_p \in \IP^p(K) : \calS q_p = 0\right\},
\end{equation*}
does not possess strong enough approximation properties to guarantee optimal $h$-convergence. In particular, it does not contain the Taylor polynomial of all local solutions to the Schr\"odinger equation;  for~$d = 1$, $p = 1$ and~$V = 0$, $\bVp{p}(K) = span\left\{1, x\right\}$; however, $\psi(x, t) = \exp\left( x + \frac{i}{2}t\right)$ satisfies~$\calS \psi = 0$, and $\Taylor{(0, 0)}{p+1}{\psi} = 1 +  x + \frac{i}{2}t \not\in \bVp{p}(K)$.
\end{remark}

\begin{remark}[Trefftz dimension]
As seen in Proposition~\ref{PROP::BASIS-DIMENSION-QUASI-TREFFTZ}, the quasi-Trefftz polynomial space has considerably lower dimension than the full polynomial space of the same degree.
This ``dimension reduction'' is common to all Trefftz and quasi-Trefftz schemes.
In particular, the dimension~$n_{d + 1,p}$ of~$\QTrefftz{p}{K}$ is equal to the dimension of the space of harmonic polynomials of degree~$\le p$ in~$\mathbb R^{d+1}$, the Trefftz space of complex exponential wave functions for the Schr\"odinger equation with piecewise-constant potential in~\cite{Gomez_Moiola_2022}, 
the Trefftz and quasi-Trefftz polynomial space for the wave equation in~\cite[Eq.~(42)--(43)]{Moiola_Perugia_2018} and \cite{ImbertGerard_Moiola_Stocker}.
\end{remark}


\section{Numerical experiments} 
\label{SECT::NUMERICAL-RESULTS}
\noindent In this section we validate the theoretical results regarding the~$h$-convergence of the proposed method, and numerically assess some additional features such as~$p$-convergence and conditioning. Although we do not report the results here, optimal convegence rates of order~$\ORDER{h^{p+1}}$ are observed for the error in the~$L^2(\QT)$-norm.

We list some aspects regarding our numerical experiments
\begin{itemize}
\item We use Cartesian-product space--time meshes with uniform partitions along each direction, which are a particular case of the situation described in Remark~\ref{REM::TIME-SLABS}.
\item We choose~$(\bx_K, t_K)$ in the definition of the quasi-Trefftz space~$\QTrefftz{p}{K}$ in \eqref{EQN::DEFINITION-QUASI-TREFFTZ} as the center of the element~$K$.
\item In all the experiments we consider Dirichlet boundary conditions.
\item The linear systems are solved using Matlab's backslash command. 
\item The quasi-Trefftz basis functions $\{b_J\}_{J = 1}^{n_{d + 1,p}}$ are constructed by choosing $\widehat m_J$ and $\widetilde m_J$ in \eqref{DEF::QUASI-TREFFTZ-BASIS} as scaled monomials and by computing the remaining coefficients $C_{\mathbf j}$ with the relations \eqref{EQN::TAYLOR-COEFFICIENTS-RELATION}.
\item In the $h$-convergence plots, the numbers in the yellow rectangles are the empirical algebraic convergence rates for the quasi-Trefftz version (continuous lines). The dashed lines correspond to the errors obtained for the full polynomial space.
\end{itemize}

\subsection{\texorpdfstring{$(1+1)$}{(1+1)}-dimensional test cases}
\noindent We first focus on the~$(1+1)$-dimensional case, for which families of explicit solutions are available for some well-known potentials $V$.

\subsubsection{$h$-convergence} 
\noindent In order to validate the error estimates in Theorems~\ref{THM::ERROR-ESTIMATE-FULL-POLYNOMIALS} and~\ref{THM::ERROR-ESTIMATE-QUASI-TREFFTZ}, we consider a series of problems with different potentials $V$.
No significant difference in terms of accuracy between the quasi-Trefftz and the full polynomial versions of the method with the same polynomial degree~$p$ (corresponding to different numbers of DOFs~$n_{d + 1,p}$ and $r_{d+1,p}$, respectively) is observed in all the experiments. 
\paragraph{Harmonic oscillator potential ($V(x) = \frac{\omega^2 x^2}{2}$)} 
    
For this potential, the Schr\"odinger equation~\eqref{EQN::SCHRODINGER-EQUATION} models the situation of a quantum harmonic oscillator for an angular frequency~$\omega > 0$. On~$\QT = (-3, 3) \times (0, 1)$, we consider the following well-known family of solutions (see e.g.,~\cite[Sect. 2.3]{Griffiths_1995}) 
\begin{equation}
        \label{EQN::EXACT-SOLUTION-HARMONIC}
        \psi_n(x, t) = \frac{1}{\sqrt{2^n n!}} \left(\frac{\omega}{\pi}\right)^{1/4}  \mathcal{H}_n\left(\sqrt{\omega} x\right) \exp\left(-\frac{1}{2}\left(\omega x^2 + (2n + 1)i\omega t\right)\right)
        \quad n\in\mathbb N,
\end{equation}
where~$\mathcal{H}_n(\cdot)$ denotes the~$n$-th physicist's Hermite polynomials as defined in~\cite[Table~18.3.1, denoted by~$H_n(\cdot)$]{DLFM_2010}.  
    
In Figure~\ref{FIG::HARMONIC-OSCILLATOR-ERROR}, we present the errors obtained for~$\omega = 10$, $n = 2$ and a sequence of Cartesian meshes with uniform partitions and~$h_x = h_t = 0.05 \times 2^{-i}$, $i = 0, \ldots 4$. Rates of convergence of order~$\ORDER{h^p}$ in the DG norm are observed, as predicted by the error estimate in Theorem~\ref{THM::ERROR-ESTIMATE-QUASI-TREFFTZ}.  
A convergence of at least order~$\ORDER{h^{p+1}}$ is observed for the~$L^2$-error at the final time, which is faster (by a factor~$h$) than the order that can be deduced from the estimates in Theorems~\ref{THM::ERROR-ESTIMATE-FULL-POLYNOMIALS} and~\ref{THM::ERROR-ESTIMATE-QUASI-TREFFTZ}.  We have also included the plots for the error decay with respect to the total number of degrees of freedom, where the same~$h$-convergence rates are observed for both versions of the method (see also the~$p$-convergence plot in
Figure~\ref{FIG::P-ERROR-HO} for a clearer understanding of the dependence of the error
on~$p$).
        \begin{figure}[htb] 
    \centering 
    \subfloat[Energy error evolution]{\includegraphics[width = .45\textwidth]{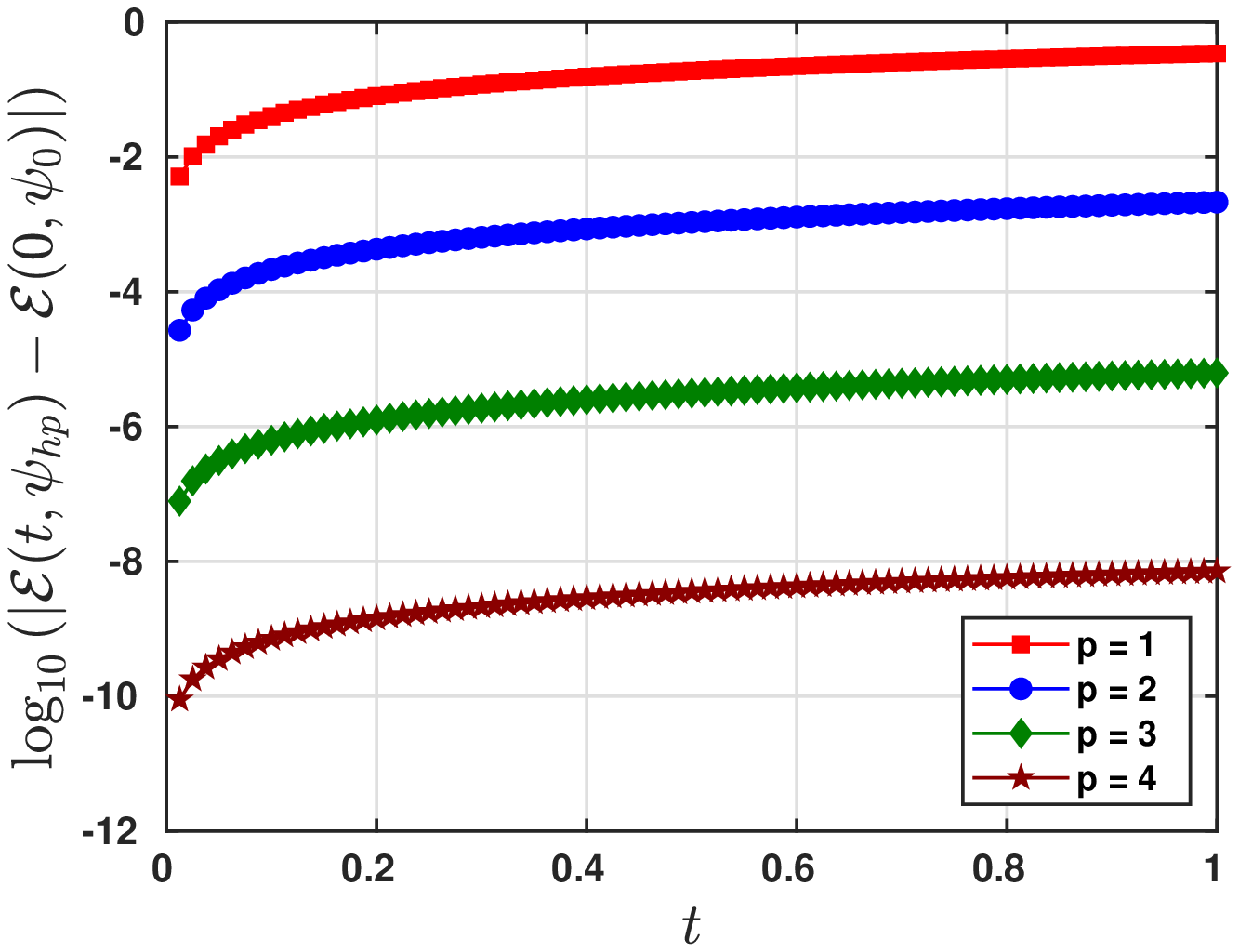}}
    \hspace{0.1in}
    \subfloat[Energy loss at $T = 1$]{\includegraphics[width = .45\textwidth]{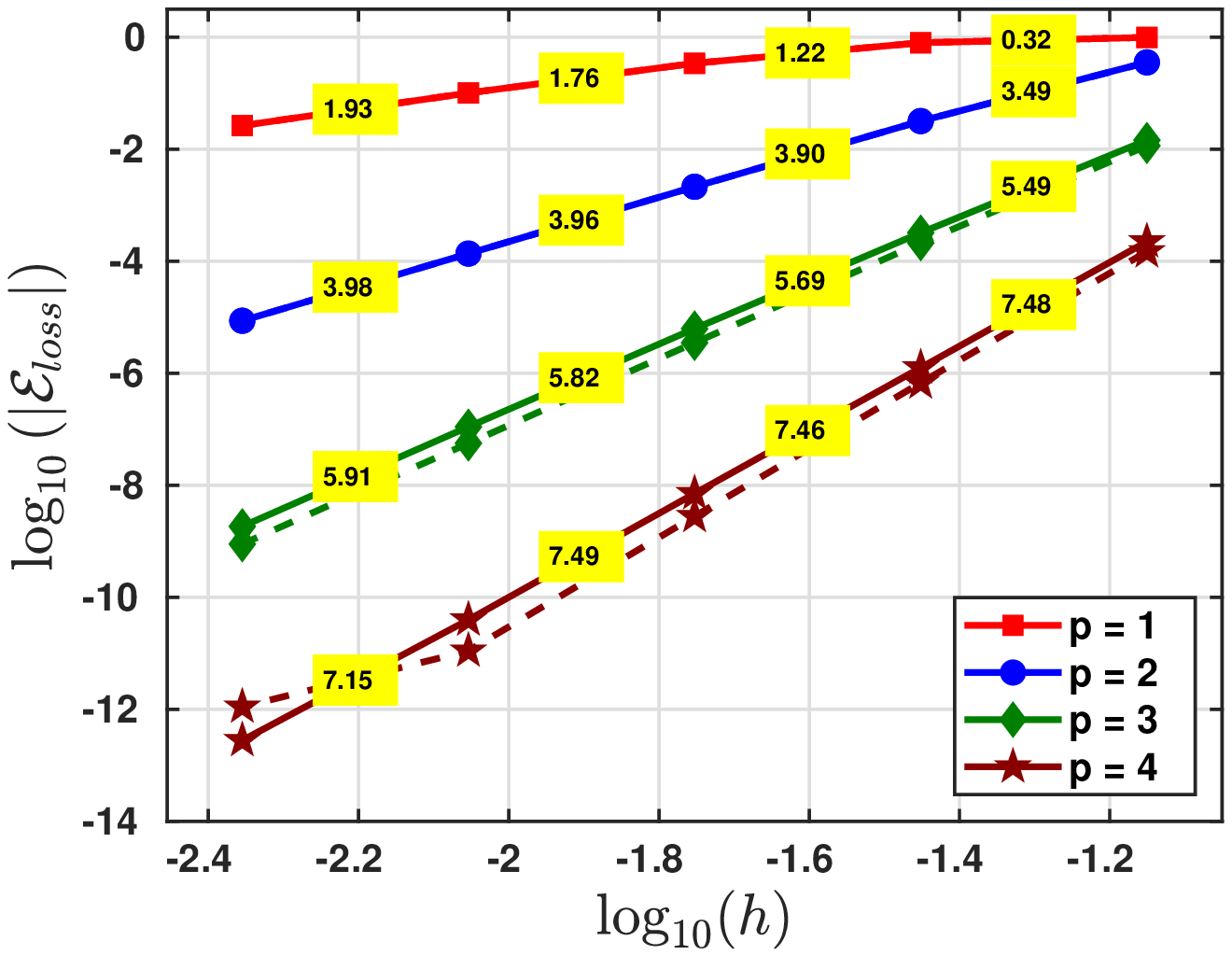}}
    \caption{Time-evolution of the energy error for the 
 quantum harmonic oscillator problem with potential~$\left(V(x) = 50 x^2\right)$ and exact solution~$\psi_2$ in \eqref{EQN::EXACT-SOLUTION-HARMONIC}. \label{FIG::HARMONIC-OSCILLATOR-ENERGY}}
    \end{figure}

Due to the fast decay of the exact solution close to the boundary (see Figure~\ref{FIG::PLOT} (panel a), the energy is expected to be preserved. In Figure~\ref{FIG::HARMONIC-OSCILLATOR-ENERGY}, we show the evolution of the energy error, and the convergence of the energy loss~$\mathcal{E}_{loss}$ to zero for the quasi-Trefftz version.
In the latter, rates of order~$\ORDER{h^{2p}}$ are observed, which follows from Remark~\ref{Rem:Energy} and the error estimates in Theorems~\ref{THM::ERROR-ESTIMATE-FULL-POLYNOMIALS} and~\ref{THM::ERROR-ESTIMATE-QUASI-TREFFTZ}. 
    
    \begin{figure}[htb]
    \centering
    \subfloat[DG error]{\includegraphics[width = .45\textwidth]{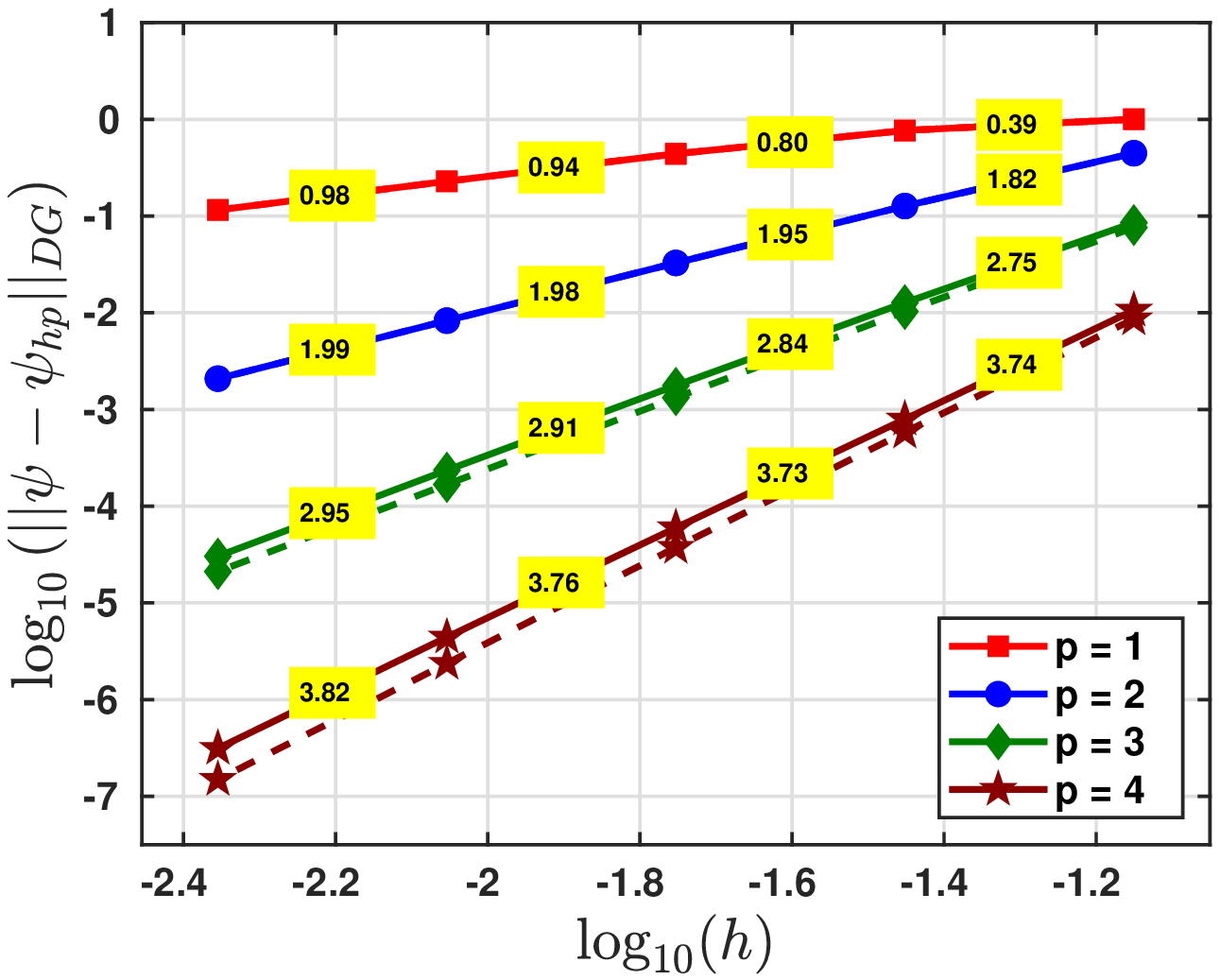}}
    \hspace{0.1in}
    \subfloat[$L^2$-error at~$T = 1$]{\includegraphics[width = .45\textwidth]{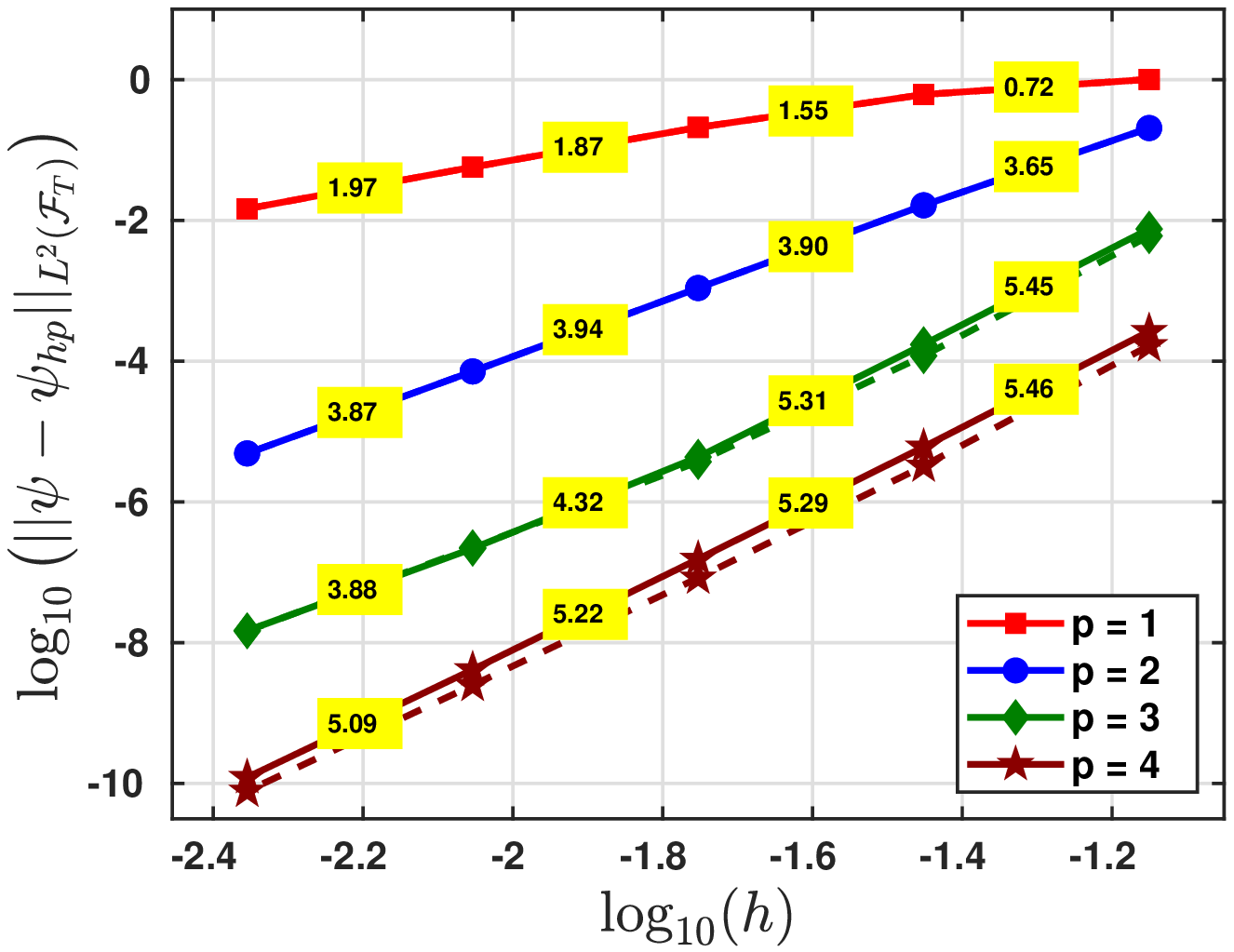}}\\
	\subfloat[DG error]{\includegraphics[width = .45\textwidth]{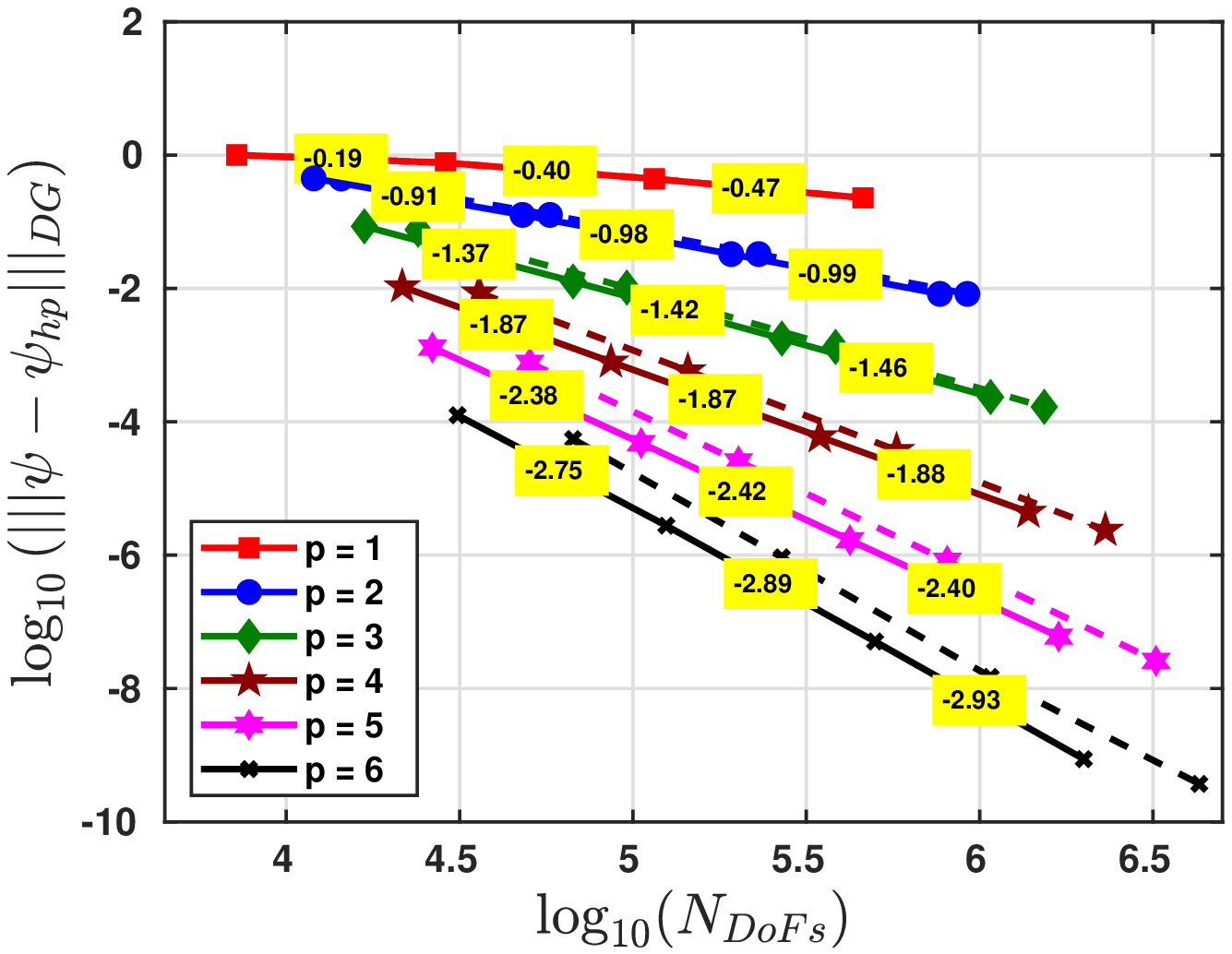}}
    \hspace{0.1in}
    \subfloat[$L^2$-error at~$T = 1$]{\includegraphics[width = .45\textwidth]{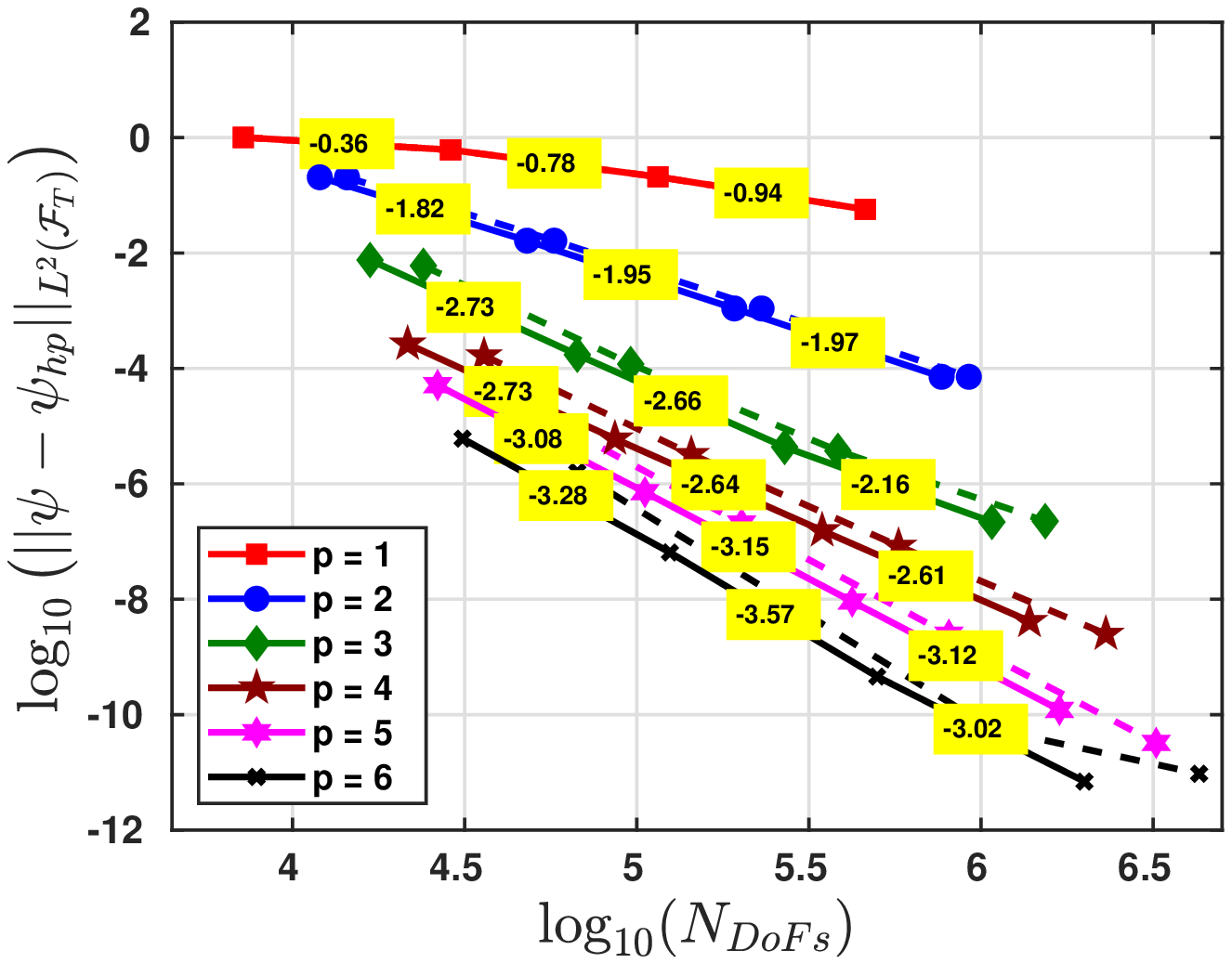}}  
    \caption{$h$-convergence for the~$(1+1)$ quantum harmonic oscillator problem with potential~$\left(V(x) = 50 x^2\right)$  and exact solution~$\psi_2$ in~\eqref{EQN::EXACT-SOLUTION-HARMONIC}. Convergence with respect to the mesh size~$h$ (top panels) and the total number of degrees of freedom (bottom panels). \label{FIG::HARMONIC-OSCILLATOR-ERROR}} 
    \end{figure}
\paragraph{Reflectionless potential ($V(x) = -a^2\mathrm{sech}^2(ax)$)} 
    
This potential was studied in \cite{Crandall_Litt_1983} as an example of a reflectionless potential. On the space--time domain $\QT = (-5, 5) \times (0, 1)$, we consider the Schr\"odinger equation with exact solution (see \cite[Problem 2.48]{Griffiths_1995})
    \begin{equation}
        \label{EQN::EXACT-SOLUTION-SECH}
        \psi(x, t) = \left(\frac{\sqrt{2}i -  a \tanh(ax)}{\sqrt{2}i + a}\right) \exp\left(i\left(\sqrt{2}x - t\right)\right).
    \end{equation}
    
In Figure~\ref{FIG::SECH-ERROR}, we show the errors obtained for a sequence of meshes with~$h_x = 2h_t = 0.2 \times 2^{-i},\ i = 0, \ldots, 4,$ and~$a = 1$. As in the previous experiment, rates of convergence of order~$\ORDER{h^p}$ and~$\ORDER{h^{p + 1}}$ are observed in the DG norm and the~$L^2$ norm at the final time, respectively. The real part of the exact solution is depicted in Figure~\ref{FIG::PLOT} (panel b).
    
\begin{figure}[!ht]
    \centering
    \subfloat[DG error]{\includegraphics[width = .45\textwidth]{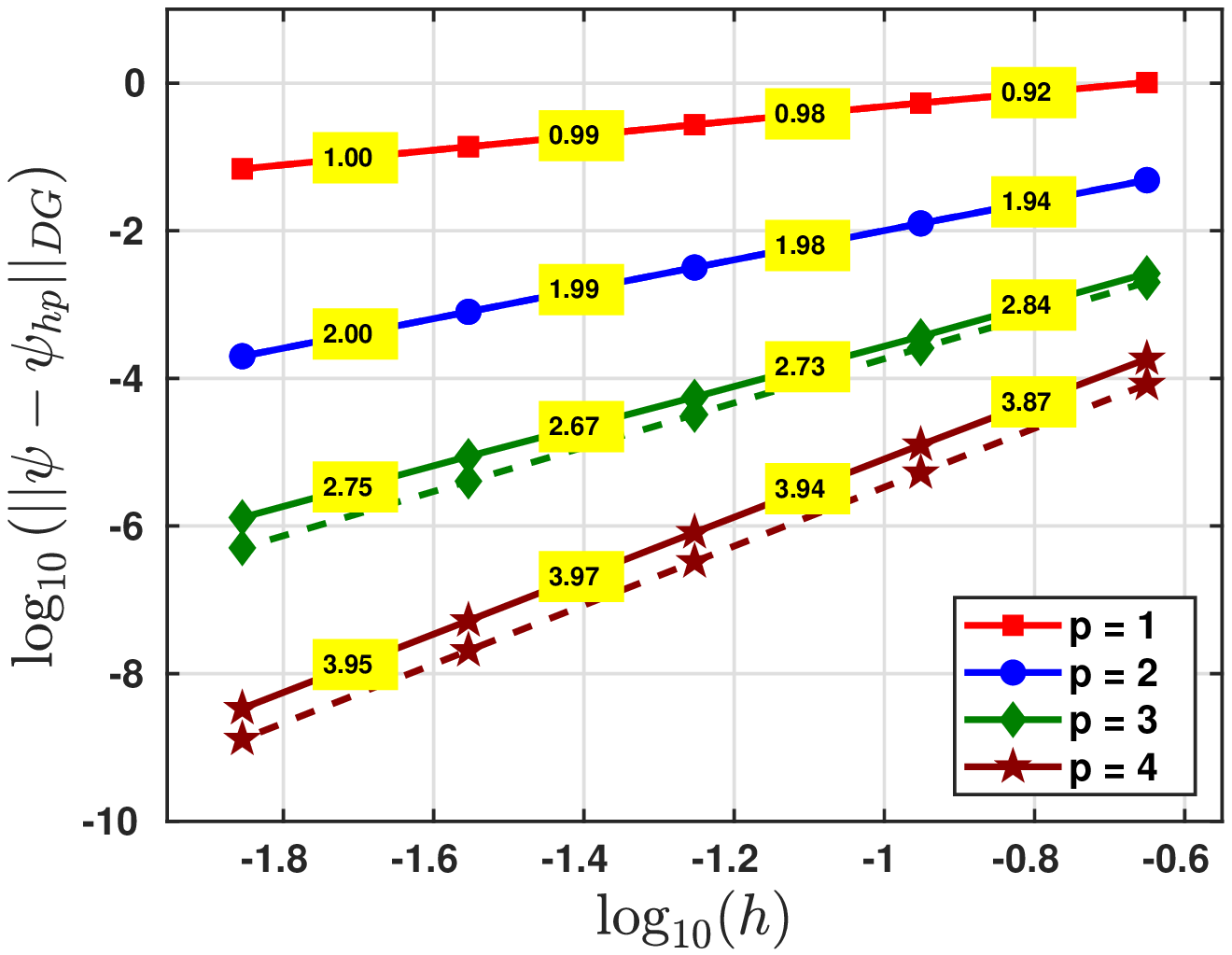}}
    \hspace{0.1in}
    \subfloat[$L^2$ error at $T = 1$]{\includegraphics[width = .45\textwidth]{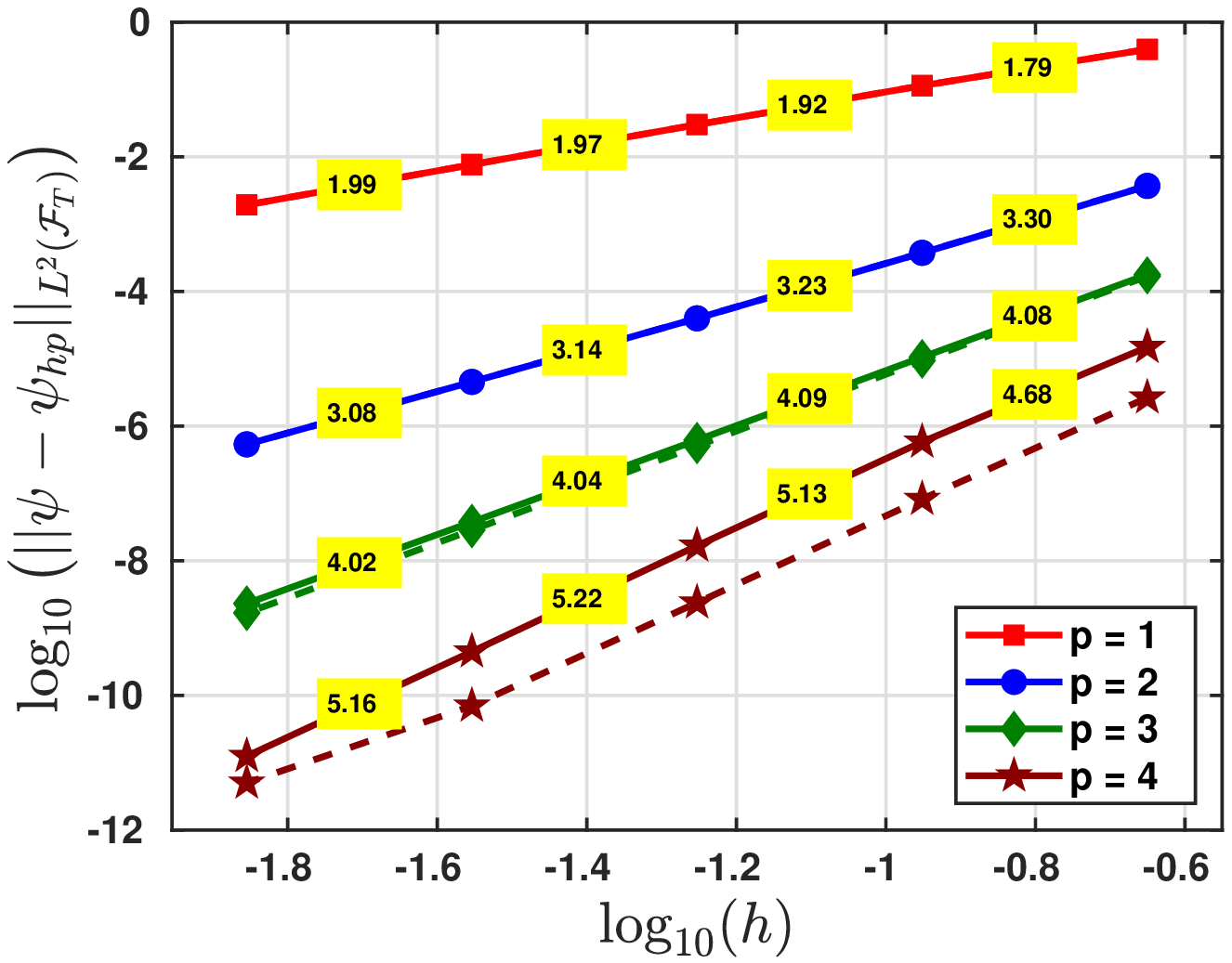}}
    \caption{$h$-convergence for the $(1+1)$ problem with potential $V(x) = -\text{sech}^2(x)$ and exact solution \eqref{EQN::EXACT-SOLUTION-SECH}. \label{FIG::SECH-ERROR}}
    \end{figure}

\paragraph{Morse potential ($V(x) = D (1 - \mathrm{e}^{-\alpha x})^2$)} 
    
This potential was introduced by Morse in~\cite{Morse_1929} to obtain a quantum-mechanical energy level spectrum of a vibrating, non-rotating diatomic molecule. There, the following family of solutions was presented (see also~\cite{Dahl_Springborg_1988})
\begin{equation}
\begin{split}
    \label{EQN::EXACT-SOLUTION-MORSE}
    \psi_{\lambda, n}(x, t) = & N(\lambda, n) \xi(x)^{\lambda - n - 1/2} \mathbb{L}_n^{(2\lambda - 2n - 1)}(\xi(x))\\
    & \times \exp\left(-\frac{\xi(x)}{2} - it\left\lfloor(n + 1/2) - \frac{1}{2 \lambda} (n + 1/2)^2 \right\rfloor \omega_o\right),
\end{split}    
\end{equation}
where~$\lfloor \cdot \rfloor$ is the floor function, $n = 0, \ldots, \lfloor\lambda - 1/2 \rfloor$, $\mathbb{L}_n^{(\alpha)}$ denote the general associated Laguerre polynomials as defined in~\cite[Table~18.3.1]{DLFM_2010}
and 
\begin{equation*}
    N(\lambda, n) = \left\lfloor\frac{(2\lambda - 2n - 1) \Gamma(n + 1)}{\Gamma(2\lambda - n)}\right\rfloor^{\frac12}, \ \lambda = \frac{\sqrt{2D}}{\alpha}, \ \xi(x) = 2\lambda \exp(-\alpha x), \ \omega_o = \sqrt{2 D} \alpha.
\end{equation*}

In Figure~\ref{FIG::MORSE-ERROR}, we show the errors obtained for the Morse potential problem with~$D = 8$, $\alpha = 4$ and exact solution~$\psi_{1, 1}$ on the space--time domain~$\QT = (-0.5, 1.5) \times (0, 1)$ for a sequence of meshes with~$h_x = h_t = 0.1 \times 2^{-i},\ i = 0, \ldots, 4$. The observed rates of convergence are in agreement with those obtained in the previous experiments. The real part of the exact solution is depicted in Figure~\ref{FIG::PLOT} (panel c).

\begin{figure}[!ht]
    \centering
    \subfloat[DG error]{\includegraphics[width = .45\textwidth]{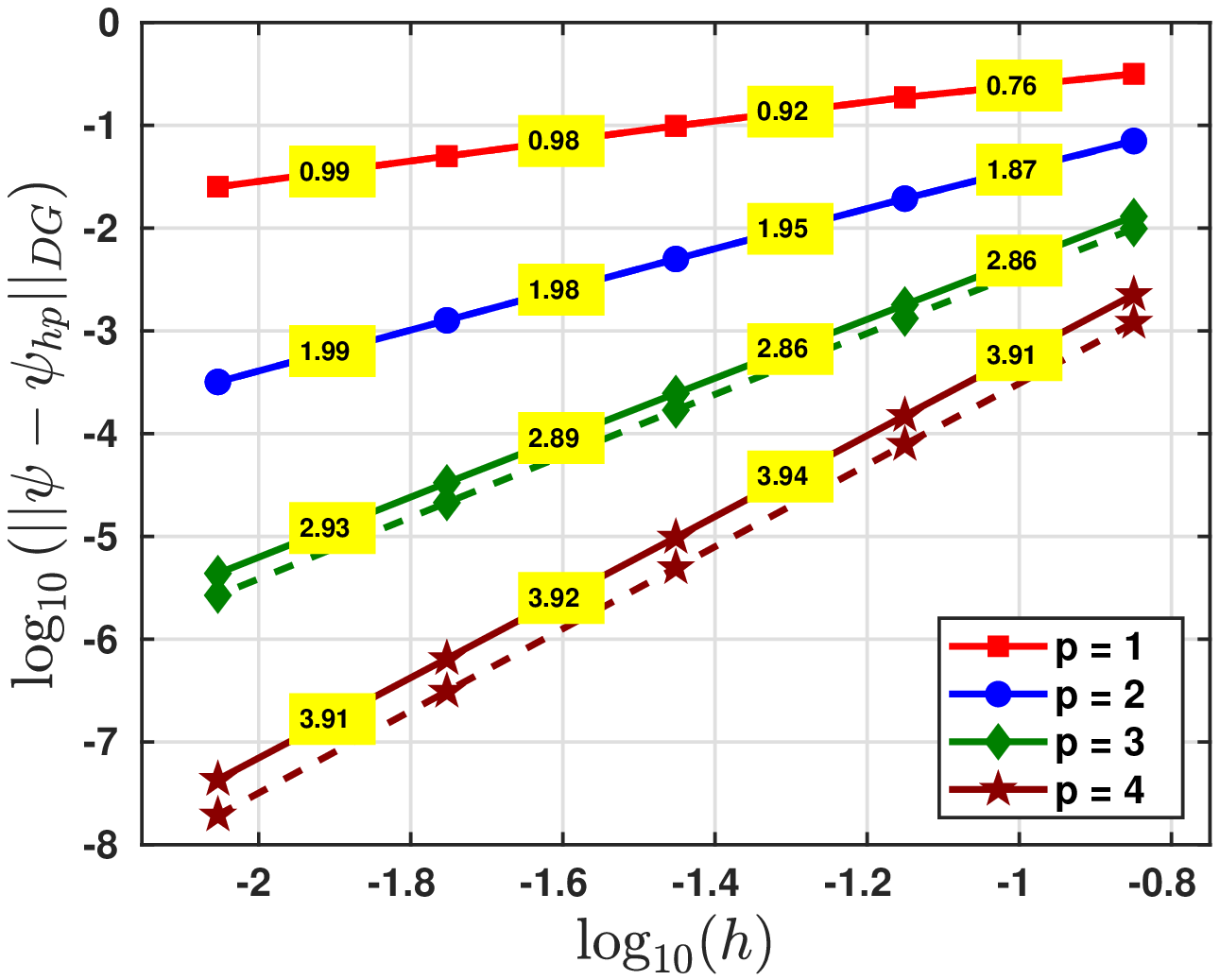}}
    \hspace{0.1in}
    \subfloat[$L^2$ error at $T = 1$]{\includegraphics[width = .45\textwidth]{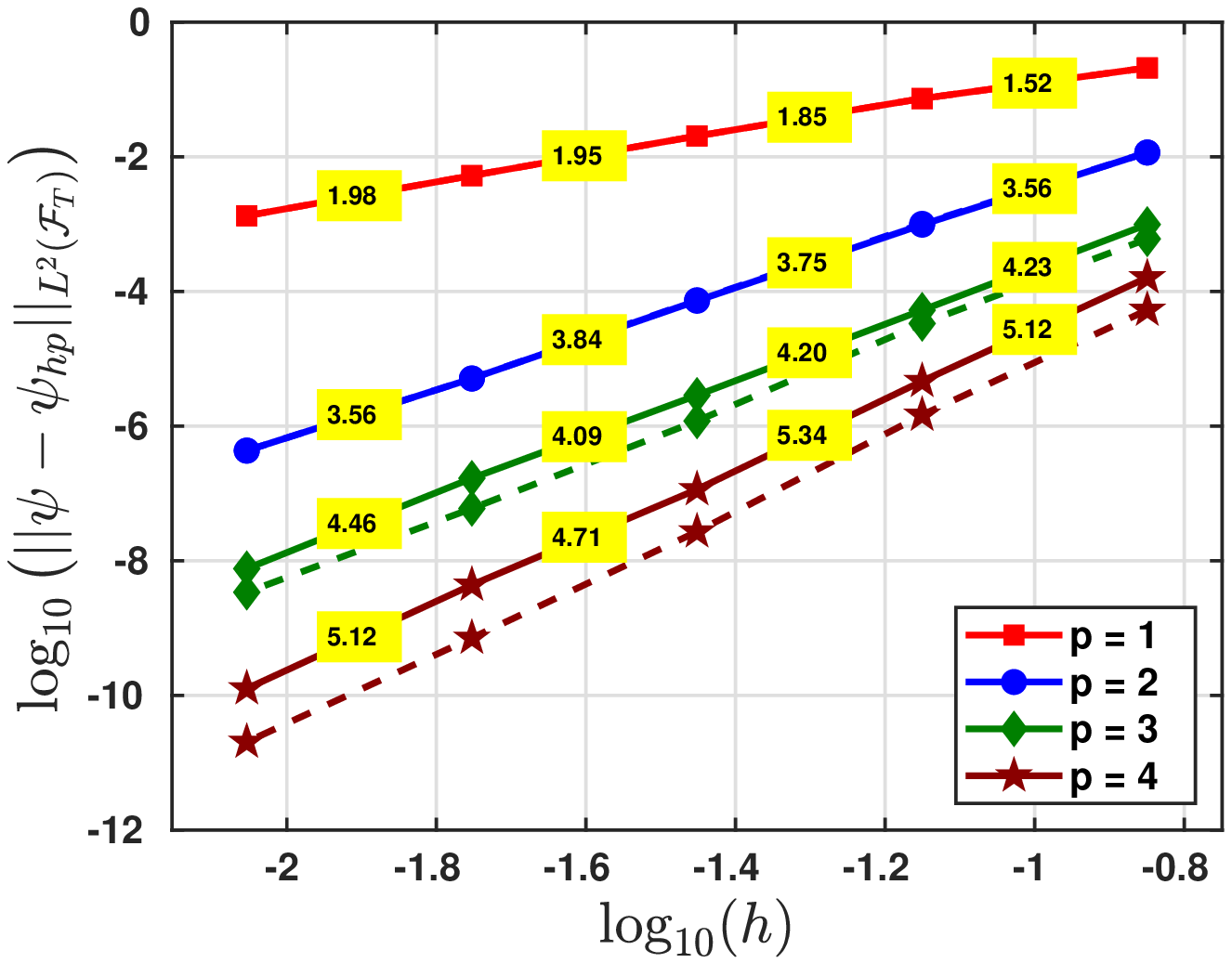}}
    \caption{$h$-convergence for the~$(1+1)$-dimensional problem with Morse potential~$V(x) = D(1 - \exp(-\alpha x))^2$ for~$D = 8$ and~$\alpha = 4$ with exact solution~\eqref{EQN::EXACT-SOLUTION-MORSE}. \label{FIG::MORSE-ERROR}}
\end{figure}

\paragraph{Square-well potential} 
We now consider a problem taken from~\cite{Gomez_Moiola_2022}, whose exact solution is not globally smooth. 
On the space--time domain $\QT = (-\sqrt{2}, \sqrt{2}) \times (0, 1)$, we consider the Schr\"odinger equation with 
homogeneous Dirichlet boundary conditions and the following square-well potential
\begin{equation}
\label{EQN::SQUARE-WELL}
V(x) = \left\{  
\begin{tabular}{ll}
	$0$ & $x\in (-1, 1),$\\
	$V_*$ & $x \in (-\sqrt{2}, \sqrt{2})\ \setminus\ (-1, 1),$
\end{tabular}
\right.	
\end{equation} 
for some fixed $V_* > 0$.
The initial condition is taken as an eigenfunction (bound state) of 
$-\frac12 \partial_x^2+V$ on $(-\sqrt{2}, \sqrt{2})$:
\begin{equation*}
\psi_0(x) = \left\{
\begin{tabular}{ll}
$\cos\left(k_*\sqrt{2} x\right)$ & $x \in (-1, 1),$ \\
$\frac{\cos(k_*)}{\sinh(\sqrt{V_* - k_*^2})} \sinh\big(\sqrt{V_* - k_*^2}(2 - \sqrt{2} |x|)\big)$ & $x \in (-\sqrt{2}, \sqrt{2}) \ \setminus\ (-1, 1)$,
\end{tabular}
\right.
\end{equation*}
where $k_*$ is a real root of the function $f( k) := \sqrt{V_* - k^2} - 
k\tan( k) \tanh(\sqrt{V_* - k^2} )$.
The solution of the corresponding initial boundary value problem 
\eqref{EQN::SCHRODINGER-EQUATION} is~$\psi(x, t) = \psi_0(x)\exp(-ik^2 t)$ and belongs to the space~$H^{p+1}(\Th)\cap \EFC{\infty}{I;\EFC{1}{\Omega}}\backslash \EFC{\infty}{I;\EFC{2}{\Omega}}$ for all~$p \in \IN$, provided that~$\Th$ is aligned with the discontinuities of the
potential~$V$; therefore, Theorems~\ref{THM::ERROR-ESTIMATE-FULL-POLYNOMIALS} and~\ref{THM::ERROR-ESTIMATE-QUASI-TREFFTZ} apply.
Among the finite set of values~$k_*$ for a given~$V_*$, in this experiment we take the largest one, corresponding to faster oscillations in space and time. 

In Figure \ref{FIG::SQUARE-WELL}, we show the errors obtained for $V_* = 20\ (k_* \approx 3.73188)$ and a sequence of meshes with $h_t = \sqrt{2} h_x = 0.1 \times 2^{-i}, \ i = 0, \ldots, 4$. Optimal convergence in both norms is observed for the errors of the quasi-Trefftz version of the method. 

\begin{figure}[!ht]
    \centering
    \subfloat[DG error]{
        \includegraphics[width = .45\textwidth]{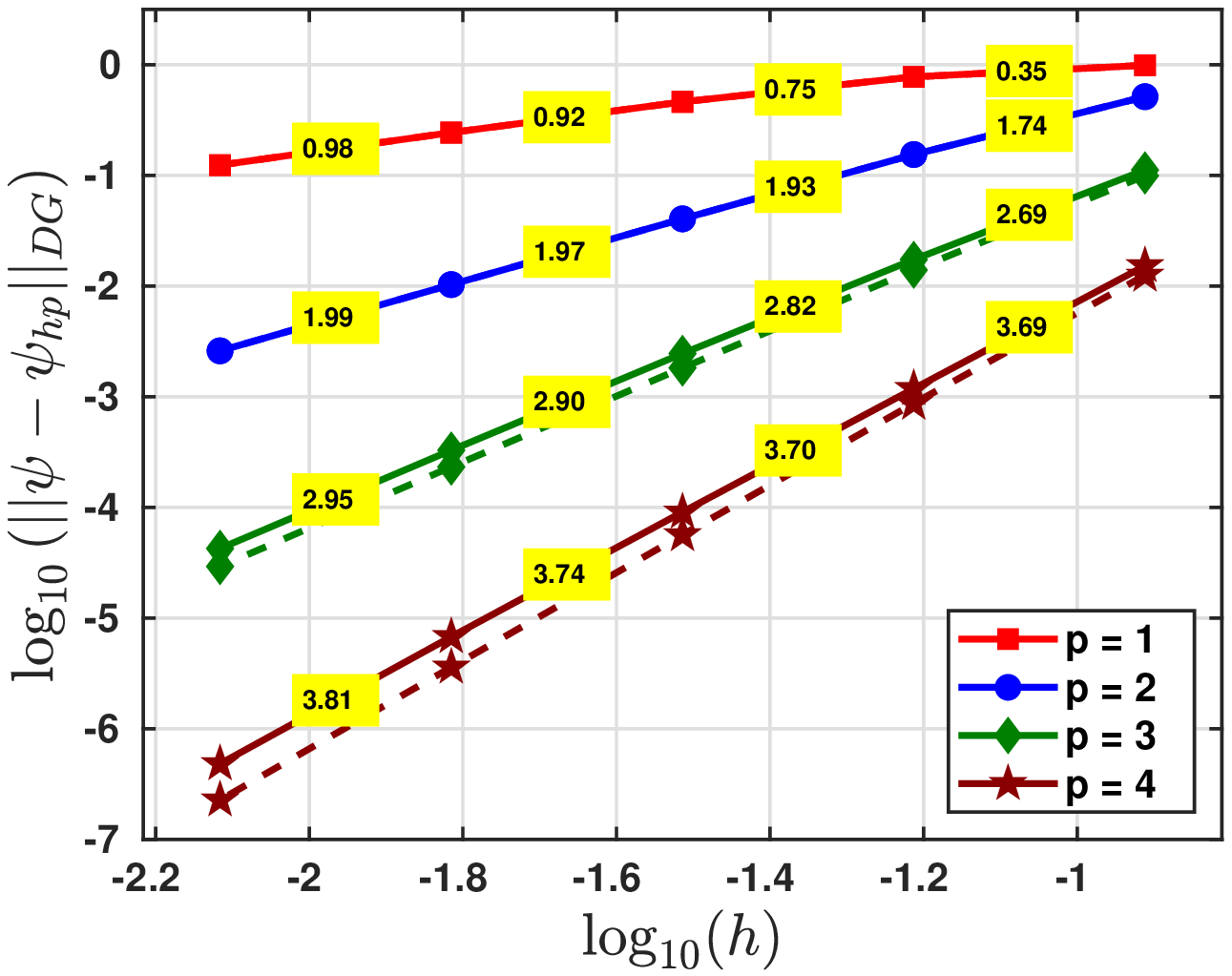}
    }
    \hspace{0.1in}
    \subfloat[$L^2$-error at $T = 1$]{
        \includegraphics[width = .45\textwidth]{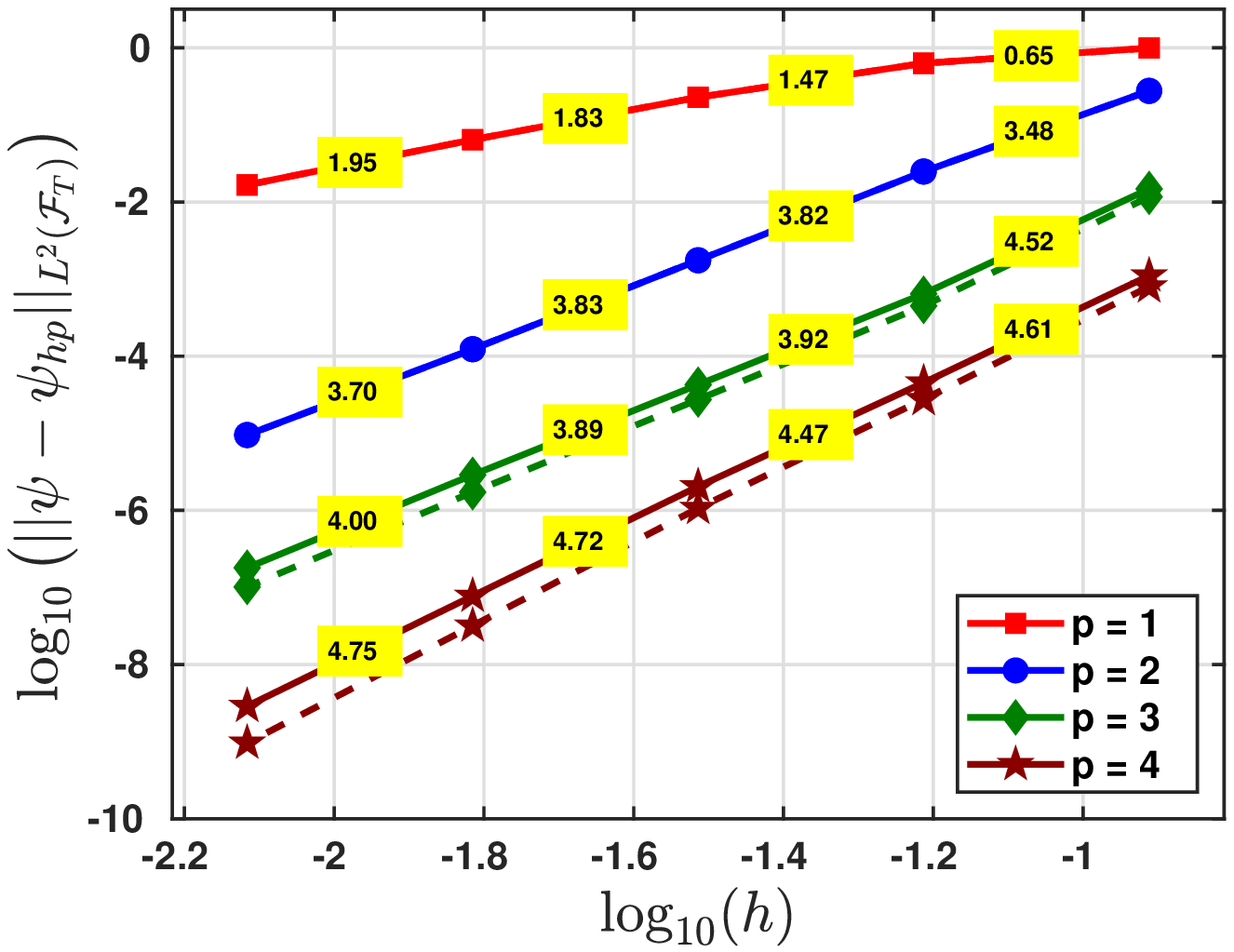}
    }
    \caption{$h$-convergence for the~$(1+1)$-dimensional problem with the square-well potential~$V(x)$ in \eqref{EQN::SQUARE-WELL}. \label{FIG::SQUARE-WELL}}
\end{figure}
    
\begin{figure}[!ht]
    \centering
    \subfloat[$V(x) = 50 x^2$]{
        \includegraphics[width = .45\textwidth]{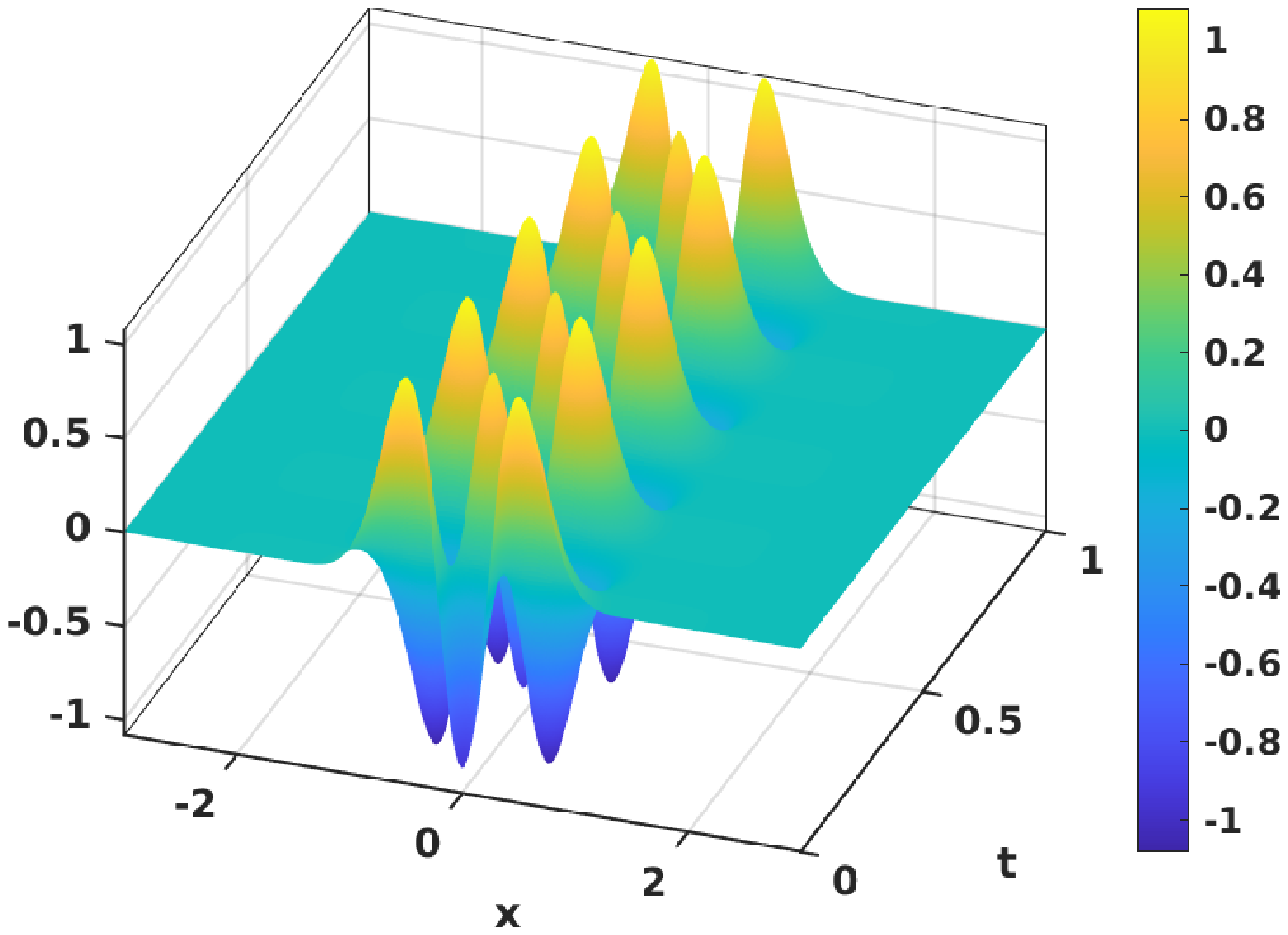}
    }
    \subfloat[$V(x) = -\text{sech}^2(x)$]{
        \includegraphics[width = .45\textwidth]{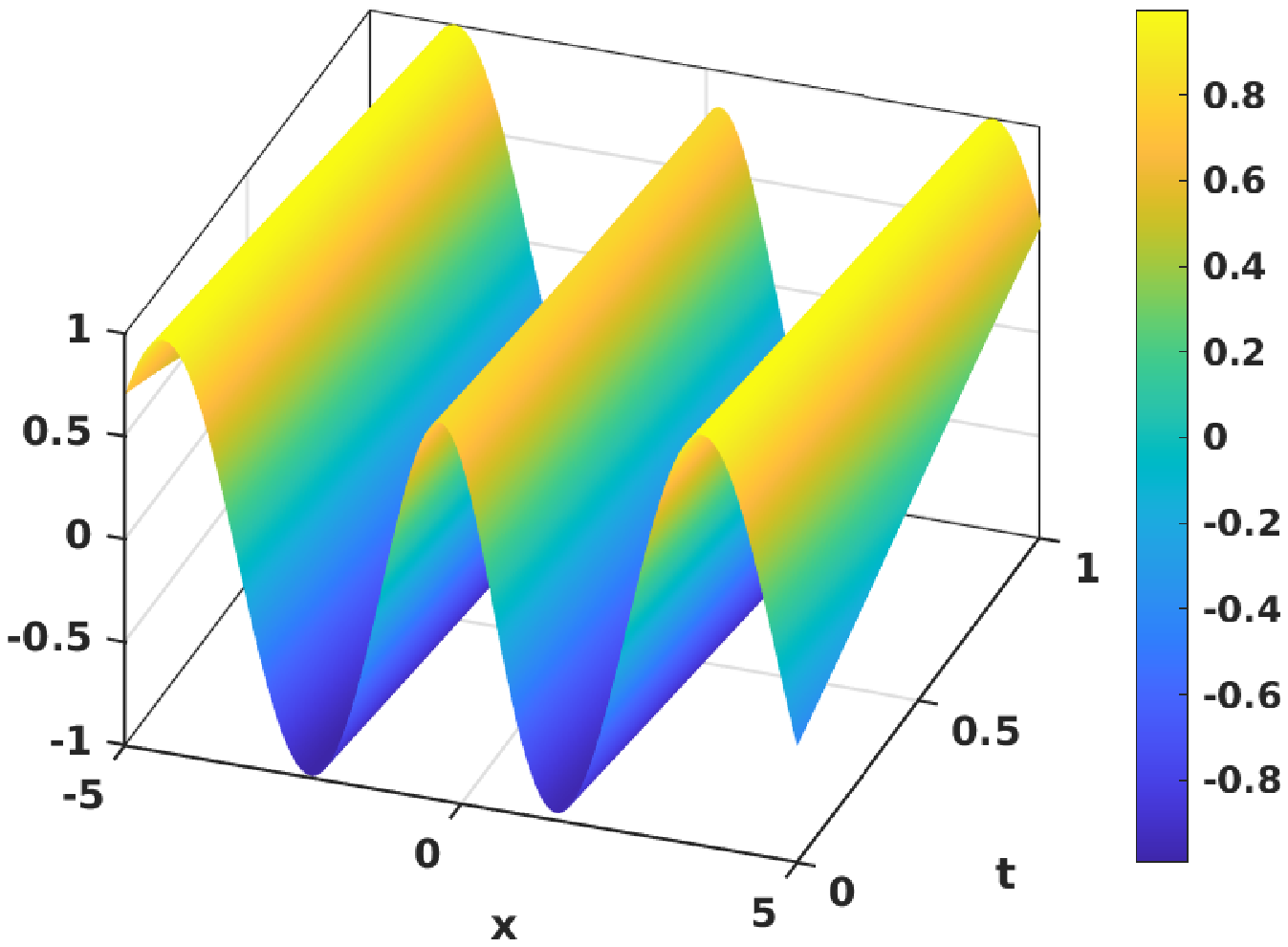}
    }\\
    \subfloat[$V(x) = 8\left(1 - \exp(-4 x)\right)^2$]{
        \includegraphics[width = .45\textwidth]{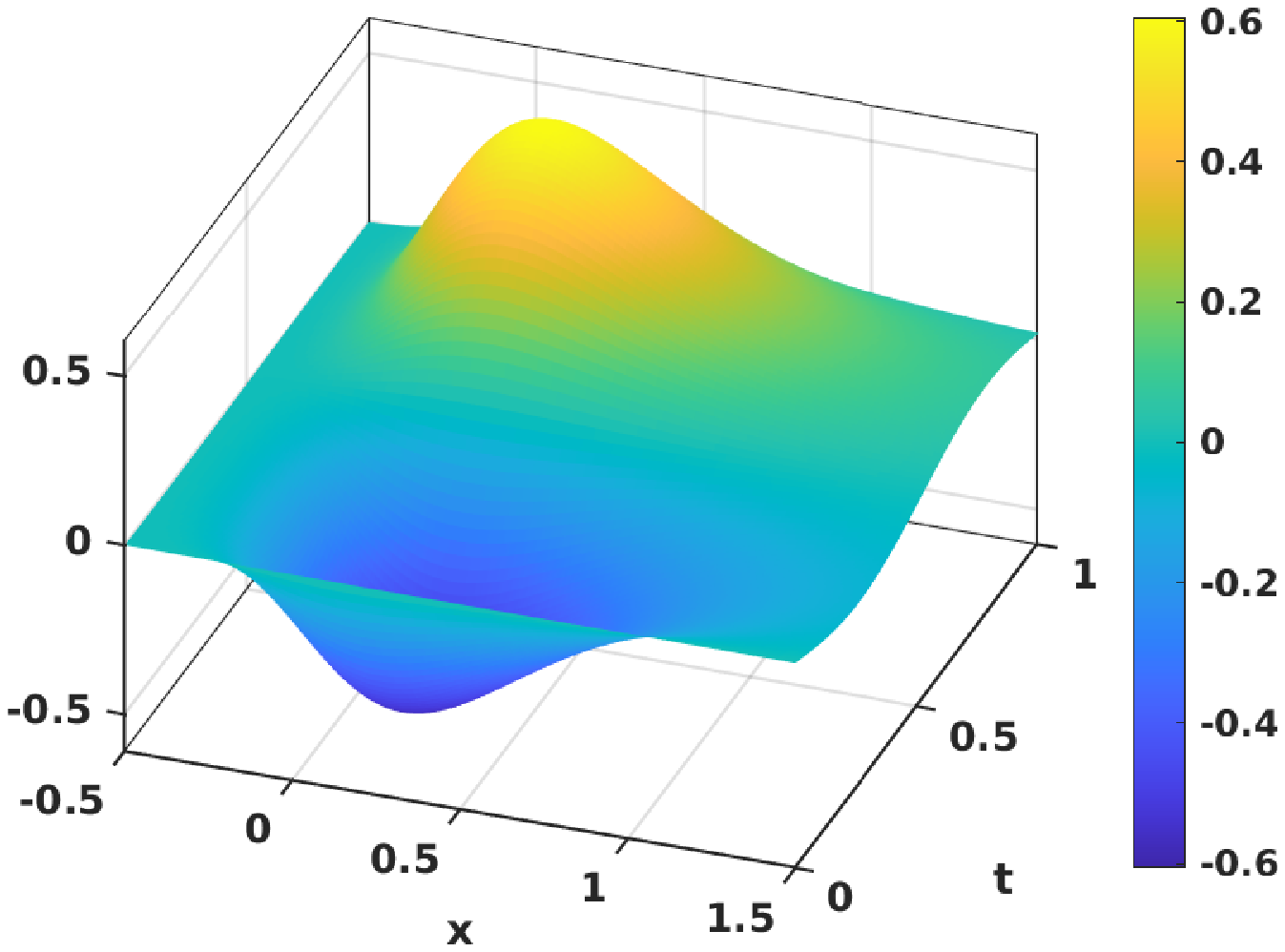}
    }
    \subfloat[Square-well potential ($V_* = 20$)]{
        \includegraphics[width = .45\textwidth]{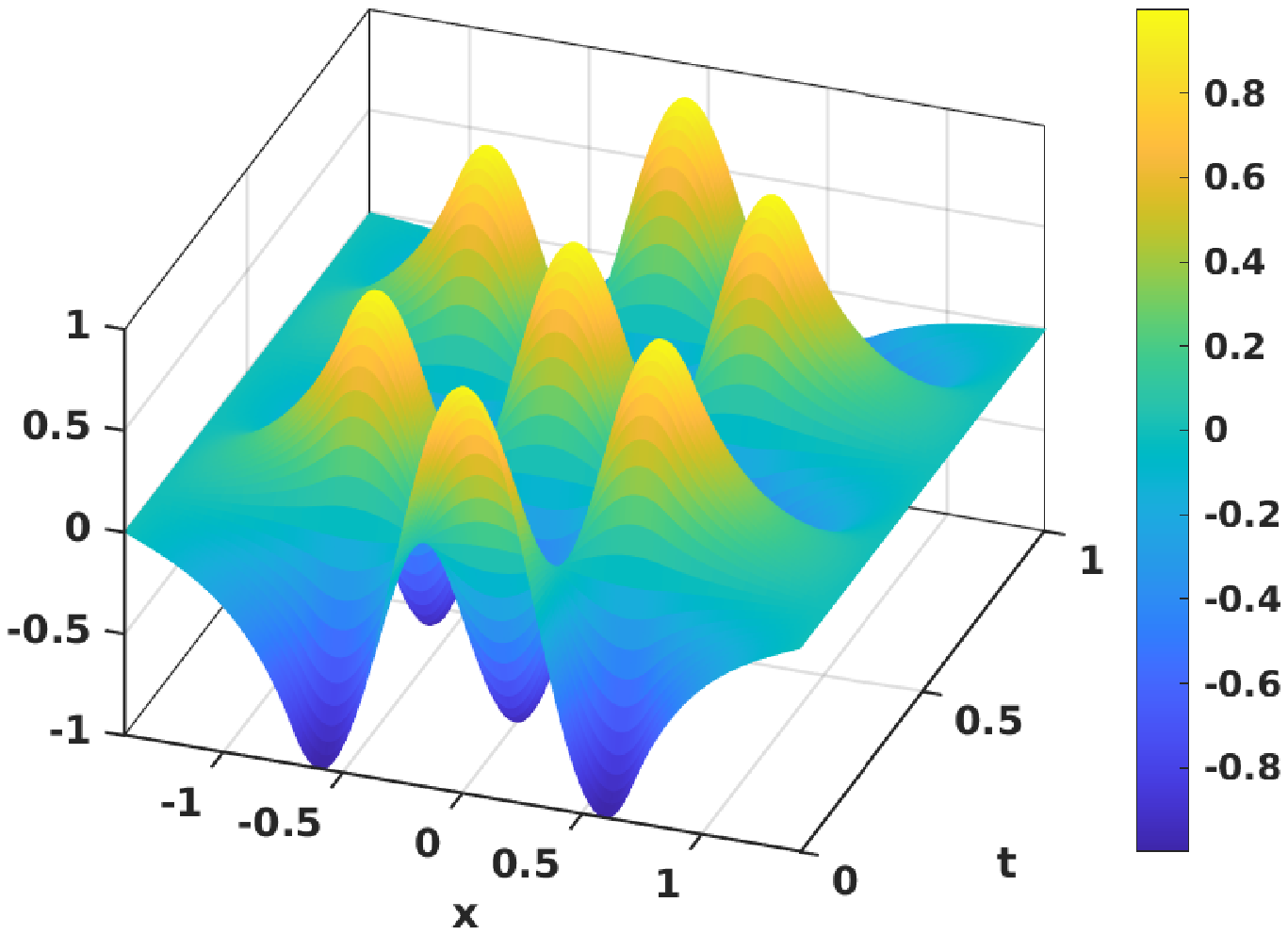}
    }
    \caption{Real part of the exact solutions for the~$(1+1)$-dimensional problems. \label{FIG::PLOT}}
\end{figure}

\subsubsection{Effect of stabilization and volume penalty terms \label{SECT::STABILIZATION}}

In this experiment we are interested in the effect of neglecting some of the terms in the variational formulation~\eqref{EQN::VARIATIONAL-DG}. To do so, we consider the~$(1 + 1)$-dimensional quantum harmonic oscillator problem with exact solution~\eqref{EQN::EXACT-SOLUTION-HARMONIC}. 
In Tables~\ref{TAB::STABILIZATION-HARMONIC-QT-1}--\ref{TAB::STABILIZATION-HARMONIC-QT-2} (quasi-Trefftz space) and~\ref{TAB::STABILIZATION-HARMONIC-FULL-1}--\ref{TAB::STABILIZATION-HARMONIC-FULL-2} (full polynomial space) we present the errors in the DG-norm obtained for the same sequence of meshes and approximation degrees as in the previous section, for different combinations of the stabilization terms~$\alpha, \beta$ and the volume penalty parameter~$\mu$. 
Although the proof of well-posedness of the method \eqref{EQN::VARIATIONAL-DG} relies on the assumption that $\alpha, \beta$ and $\mu$ are strictly positive, in our numerical experiments, the matrices of the arising linear systems are non-singular and optimal convergence rates are observed even when all these parameters are set to zero. Moreover, the errors obtained when~$\alpha = 0$ or~$\beta = 0$ are smaller as some terms in the definition \eqref{EQN::DG-NORM} of~$\Norm{\cdot}{\DG}$ vanish, while the presence of~$\mu$ seems to have just a mild effect in the results. Not shown here, similar effects were observed for the error in the~$L^2(\FT)$-norm.

\begin{table}[!ht]
\centering
\begin{tabular}{ccccccccc}
\hline
\multicolumn{9}{c}{$\mu = \max\{\hKt, \hKx\} $} \\
\hline 
$h$ & \multicolumn{2}{c}{$\alpha = \frac1{h_{F_\bx}}, \beta = h_{F_\bx}$} & \multicolumn{2}{c}{$\alpha = 0, \beta = 0$} & \multicolumn{2}{c}{$\alpha = \frac1{h_{F_\bx}}, \beta = 0$} & \multicolumn{2}{c}{$\alpha = 0, \beta = h_{F_\bx}$} \\
\hline
& DG error & Rate & DG error & Rate & DG error & Rate & DG error & Rate \\
\hline
\hline
\multicolumn{9}{c}{$p = 1$} \\
\hline
\hline
7.07e-02 & 1.00e+00 & \bf{---} & 9.81e-01 & \bf{---} & 1.01e+00 & \bf{---} & 1.00e+00 & \bf{---}\\
3.54e-02 & 7.67e-01 & \bf{0.39} & 4.76e-01 & \bf{1.04} & 6.72e-01 & \bf{0.58}  & 6.53e-01 & \bf{0.62}  \\
1.77e-02 & 4.40e-01 & \bf{0.80} & 2.14e-01 & \bf{1.15} & 3.62e-01 & \bf{0.89}  & 3.40e-01 & \bf{0.94}  \\
8.84e-03 & 2.29e-01 & \bf{0.94} & 1.01e-01 & \bf{1.08}  & 1.85e-01 & \bf{0.97} & 1.70e-01 & \bf{1.00}  \\
4.42e-03 & 1.16e-01 & \bf{0.98} & 4.96e-02 & \bf{1.03} & 9.31e-02 & \bf{0.99}  & 8.49e-02 & \bf{1.00}  \\
\hline
\hline
\multicolumn{9}{c}{$p = 2$} \\
\hline
\hline
7.07e-02 & 4.47e-01 & \bf{---} & 2.59e-01 & \bf{---} & 2.99e-01 & \bf{---} & 4.37e-01 & \bf{---} \\
3.54e-02 & 1.27e-01 & \bf{1.82} & 6.90e-02 & \bf{1.91} & 8.24e-02 & \bf{1.86}  & 1.20e-01 & \bf{1.87}  \\
1.77e-02 & 3.28e-02 & \bf{1.95} & 1.78e-02 & \bf{1.96} & 2.15e-02 & \bf{1.94}  & 3.05e-02 & \bf{1.97} \\
8.84e-03 & 8.29e-03 & \bf{1.98} & 4.50e-03 & \bf{1.98} & 5.48e-03 & \bf{1.97}  & 7.68e-03 & \bf{1.99}  \\
4.42e-03 & 2.08e-03 & \bf{1.99} & 1.13e-03 & \bf{1.99} & 1.38e-03 & \bf{1.98}  & 1.93e-03 & \bf{2.00}  \\ 
\hline
\hline
\multicolumn{9}{c}{$p = 3$} \\
\hline
\hline
7.07e-02 & 8.54e-02 & \bf{---} & 5.73e-02 & \bf{---} & 5.87e-02 & \bf{---} & 8.65e-02 & \bf{---} \\
3.54e-02 & 1.27e-02 & \bf{2.75} & 8.00e-03 & \bf{2.84} & 8.28e-03 & \bf{2.83}  & 1.27e-02 & \bf{2.77} \\
1.77e-02 & 1.77e-03 & \bf{2.84} & 1.08e-03 & \bf{2.89} & 1.12e-03 & \bf{2.88}  & 1.75e-03 & \bf{2.86}  \\
8.84e-03 & 2.35e-04 & \bf{2.91} & 1.42e-04 & \bf{2.93} & 1.48e-04 & \bf{2.93}  & 2.32e-04 & \bf{2.92} \\
4.42e-03 & 3.04e-05 & \bf{2.95} & 1.82e-05 & \bf{2.96} & 1.90e-05 & \bf{2.96}  & 2.99e-05 & \bf{2.96}  \\ 
\hline
\hline
\multicolumn{9}{c}{$p = 4$} \\
\hline
\hline
7.07e-02 & 1.06e-02 & \bf{---} & 9.36e-03 & \bf{---} & 9.27e-03 & \bf{---} & 1.08e-02 & \bf{---} \\
3.54e-02 & 7.93e-04 & \bf{3.74} & 6.56e-04 & \bf{3.84} & 6.64e-04 & \bf{3.80}  & 7.95e-04 & \bf{3.76}  \\
1.77e-02 & 5.97e-05 & \bf{3.73} & 4.59e-05 & \bf{3.84} & 4.66e-05 & \bf{3.83}  & 5.94e-05 & \bf{3.74}  \\
8.84e-03 & 4.42e-06 & \bf{3.76}  & 3.16e-06 & \bf{3.86} & 3.21e-06 & \bf{3.86}  & 4.39e-06 & \bf{3.76}  \\
4.42e-03 & 3.13e-07 & \bf{3.82} & 2.11e-07 & \bf{3.90}  & 2.14e-07 & \bf{3.90}  & 3.11e-07 & \bf{3.82} \\
\hline
\hline
\end{tabular}
\caption{$h$-convergence for the quasi-Trefftz version applied to the quantum harmonic oscillator problem with potential~$V(x) = 50x^2$ and exact solution~$\psi_2$ in~\eqref{EQN::EXACT-SOLUTION-HARMONIC} for different combinations of the stabilization parameters~$\alpha, \beta$ and  volume penalty parameter~$\mu\ne0$.
\label{TAB::STABILIZATION-HARMONIC-QT-1}}
\end{table}

\begin{table}[!ht]
\centering
\begin{tabular}{ccccccccc}
\hline
\multicolumn{9}{c}{$\mu = 0 $} \\
\hline 
$h$ & \multicolumn{2}{c}{$\alpha = \frac1{h_{F_\bx}}, \beta = h_{F_\bx}$} & \multicolumn{2}{c}{$\alpha = 0, \beta = 0$} & \multicolumn{2}{c}{$\alpha = \frac1{h_{F_\bx}}, \beta = 0$} & \multicolumn{2}{c}{$\alpha = 0, \beta = h_{F_\bx}$} \\
\hline
& DG error & Rate & DG error & Rate & DG error & Rate & DG error & Rate \\
\hline
\hline
\multicolumn{9}{c}{$p = 1$} \\
\hline
\hline
7.07e-02 & 1.04e+00 & \bf{---} & 1.16e+00 & \bf{---} & 1.07e+00 & \bf{---} & 1.09e+00 & \bf{---} \\
3.54e-02 & 7.78e-01 & \bf{0.43} & 5.02e-01 & \bf{1.21} & 6.84e-01 & \bf{0.64} & 6.69e-01 & \bf{0.70}  \\
1.77e-02 & 4.42e-01 & \bf{0.81} & 2.18e-01 & \bf{1.20} & 3.64e-01 & \bf{0.91}  & 3.42e-01 & \bf{0.97} \\
8.84e-03 & 2.29e-01 & \bf{0.95} & 1.02e-01 & \bf{1.09} & 1.85e-01 & \bf{0.97}  & 1.71e-01 & \bf{1.00} \\
4.42e-03 & 1.16e-01 & \bf{0.99} & 4.98e-02 & \bf{1.04} & 9.32e-02 & \bf{0.99}  & 8.50e-02 & \bf{1.01} \\
\hline
\hline
\multicolumn{9}{c}{$p = 2$} \\
\hline
\hline
7.07e-02 & 4.63e-01 & \bf{---} & 2.96e-01 & \bf{---} & 3.23e-01 & \bf{---} & 4.60e-01 & \bf{---}\\
3.54e-02 & 1.29e-01 & \bf{1.84} & 7.38e-02 & \bf{2.00} & 8.58e-02 & \bf{1.91}  & 1.23e-01 & \bf{1.90} \\
1.77e-02 & 3.31e-02 & \bf{1.97} & 1.84e-02 & \bf{2.01} & 2.19e-02 & \bf{1.97} & 3.09e-02 & \bf{1.99} \\
8.84e-03 & 8.33e-03 & \bf{1.99} & 4.58e-03 & \bf{2.00} & 5.54e-03 & \bf{1.99}  & 7.73e-03 & \bf{2.00} \\
4.42e-03 & 2.09e-03 & \bf{2.00} & 1.14e-03 & \bf{2.00} & 1.39e-03 & \bf{1.99}  & 1.93e-03 & \bf{2.00} \\
\hline
\hline
\multicolumn{9}{c}{$p = 3$} \\
\hline
\hline
7.07e-02 & 8.73e-02 & \bf{---} & 7.84e-02 & \bf{---} & 7.59e-02 & \bf{---} & 8.85e-02 & \bf{---} \\
3.54e-02 & 1.31e-02 & \bf{2.74} & 9.65e-03 & \bf{3.02} & 9.72e-03 & \bf{2.96} & 1.31e-02 & \bf{2.76} \\
1.77e-02 & 1.82e-03 & \bf{2.85} & 1.20e-03 & \bf{3.01} & 1.23e-03 & \bf{2.98} & 1.80e-03 & \bf{2.86} \\
8.84e-03 & 2.39e-04 & \bf{2.92} & 1.50e-04 & \bf{3.00} & 1.55e-04 & \bf{2.99}  & 2.36e-04 & \bf{2.93} \\
4.42e-03 & 3.07e-05 & \bf{2.96} & 1.87e-05 & \bf{3.00} & 1.95e-05 & \bf{2.99}  & 3.02e-05 & \bf{2.97} \\
\hline
\hline
\multicolumn{9}{c}{$p = 4$} \\
\hline
\hline
7.07e-02 & 1.09e-02 & \bf{---} & 1.71e-02 & \bf{---} & 1.56e-02 & \bf{---} & 1.12e-02 & \bf{---} \\
3.54e-02 & 7.97e-04 & \bf{3.77} & 9.77e-04 & \bf{4.13} & 9.60e-04 & \bf{4.02}  & 7.98e-04 & \bf{3.81} \\
1.77e-02 & 6.02e-05 & \bf{3.73} & 5.97e-05 & \bf{4.03} & 5.98e-05 & \bf{4.00}  & 5.99e-05 & \bf{3.73} \\
8.84e-03 & 4.50e-06 & \bf{3.74} & 3.71e-06 & \bf{4.01} & 3.74e-06 & \bf{4.00}  & 4.48e-06 & \bf{3.74} \\
4.42e-03 & 3.19e-07 & \bf{3.82} & 2.31e-07 & \bf{4.00} & 2.34e-07 & \bf{4.00}  & 3.17e-07 & \bf{3.82} \\ 
\hline
\hline
\end{tabular}
\caption{$h$-convergence for the quasi-Trefftz version applied to the quantum harmonic oscillator problem with potential~$V(x) = 50x^2$ and exact solution~$\psi_2$ in~\eqref{EQN::EXACT-SOLUTION-HARMONIC} for different combinations of the stabilization parameters~$\alpha, \beta$ and volume penalty parameter~$\mu=0$.
\label{TAB::STABILIZATION-HARMONIC-QT-2}}
\end{table}

\begin{table}[!ht]
\centering
\begin{tabular}{ccccccccc}
\hline
\multicolumn{9}{c}{$\mu = \max\{\hKt, \hKx\} $} \\
\hline 
$h$ & \multicolumn{2}{c}{$\alpha = \frac1{h_{F_\bx}}, \beta = h_{F_\bx}$} & \multicolumn{2}{c}{$\alpha = 0, \beta = 0$} & \multicolumn{2}{c}{$\alpha = \frac1{h_{F_\bx}}, \beta = 0$} & \multicolumn{2}{c}{$\alpha = 0, \beta = h_{F_\bx}$} \\
\hline
& DG error & Rate & DG error & Rate & DG error & Rate & DG error & Rate \\
\hline
\hline
\multicolumn{9}{c}{$p = 1$} \\
\hline
\hline
7.07e-02 & 1.00e+00 & \bf{---} & 9.81e-01 & \bf{---} & 1.01e+00 & \bf{---}\\
3.54e-02 & 7.67e-01 & \bf{0.39} & 4.76e-01 & \bf{1.04} & 6.72e-01 & \bf{0.58} & 1.00e+00 & \bf{---} \\
1.77e-02 & 4.40e-01 & \bf{0.80} & 2.14e-01 & \bf{1.15} & 3.62e-01 & \bf{0.89} & 3.40e-01 & \bf{0.94} \\
8.84e-03 & 2.29e-01 & \bf{0.94} & 1.01e-01 & \bf{1.08} & 1.85e-01 & \bf{0.97} & 1.70e-01 & \bf{1.00}  \\
4.42e-03 & 1.16e-01 & \bf{0.98} & 4.96e-02 & \bf{1.03} & 9.31e-02 & \bf{0.99} & 8.49e-02 & \bf{1.00} \\
\hline
\hline
\multicolumn{9}{c}{$p = 2$} \\
\hline
\hline
7.07e-02 & 4.46e-01 & \bf{---} & 2.55e-01 & \bf{---} & 2.96e-01 & \bf{---} & 4.34e-01 & \bf{---}\\
3.54e-02 & 1.27e-01 & \bf{1.81} & 6.88e-02 & \bf{1.89} & 8.22e-02 & \bf{1.85} & 1.20e-01 & \bf{1.86} \\
1.77e-02 & 3.28e-02 & \bf{1.95} & 1.77e-02 & \bf{1.95} & 2.15e-02 & \bf{1.94} & 3.05e-02 & \bf{1.97} \\
8.84e-03 & 8.29e-03 & \bf{1.98} & 4.50e-03 & \bf{1.98} & 5.48e-03 & \bf{1.97} & 7.68e-03 & \bf{1.99} \\
4.42e-03 & 2.08e-03 & \bf{1.99} & 1.13e-03 & \bf{1.99} & 1.38e-03 & \bf{1.98} & 1.93e-03 & \bf{2.00} \\ 
\hline
\hline
\multicolumn{9}{c}{$p = 3$} \\
\hline
\hline
7.07e-02 & 7.62e-02 & \bf{---} & 4.67e-02 & \bf{---} & 4.93e-02 & \bf{---} & 7.65e-02 & \bf{---}\\
3.54e-02 & 1.03e-02 & \bf{2.89} & 6.22e-03 & \bf{2.91} & 6.68e-03 & \bf{2.88} & 1.01e-02 & \bf{2.92} \\
1.77e-02 & 1.33e-03 & \bf{2.96} & 8.05e-04 & \bf{2.95} & 8.71e-04 & \bf{2.94} & 1.29e-03 & \bf{2.97} \\
8.84e-03 & 1.68e-04 & \bf{2.98} & 1.03e-04 & \bf{2.97} & 1.11e-04 & \bf{2.97} & 1.62e-04 & \bf{2.99} \\
4.42e-03 & 2.11e-05 & \bf{2.99} & 1.30e-05 & \bf{2.99} & 1.41e-05 & \bf{2.98} & 2.03e-05 & \bf{2.99} \\
\hline
\hline
\multicolumn{9}{c}{$p = 4$} \\
\hline
\hline
7.07e-02 & 8.63e-03 & \bf{---} & 6.05e-03 & \bf{---} & 6.14e-03 & \bf{---} & 8.74e-03 & \bf{---}\\
3.54e-02 & 5.82e-04 & \bf{3.89} & 3.95e-04 & \bf{3.94} & 4.10e-04 & \bf{3.90}  & 5.77e-04 & \bf{3.92} \\
1.77e-02 & 3.74e-05 & \bf{3.96} & 2.54e-05 & \bf{3.96} & 2.66e-05 & \bf{3.95} & 3.67e-05 & \bf{3.97} \\
8.84e-03 & 2.37e-06 & \bf{3.98} & 1.62e-06 & \bf{3.97} & 1.69e-06 & \bf{3.97} & 2.32e-06 & \bf{3.99} \\
4.42e-03 & 1.49e-07 & \bf{3.99} & 1.02e-07 & \bf{3.99} & 1.07e-07 & \bf{3.98} & 1.45e-07 & \bf{3.99} \\
\hline
\hline
\end{tabular}
\caption{$h$-convergence for the full polynomial version applied to the quantum harmonic oscillator problem with potential~$V(x) = 50x^2$ and exact solution~$\psi_2$ in~\eqref{EQN::EXACT-SOLUTION-HARMONIC} for different combinations of the stabilization parameters~$\alpha, \beta$ and volume penalty parameter~$\mu\ne0$.
\label{TAB::STABILIZATION-HARMONIC-FULL-1}}
\end{table}

\begin{table}[!ht]
\centering
\begin{tabular}{ccccccccc}
\hline
\multicolumn{9}{c}{$\mu = 0 $} \\
\hline 
$h$ & \multicolumn{2}{c}{$\alpha = \frac1{h_{F_\bx}}, \beta = h_{F_\bx}$} & \multicolumn{2}{c}{$\alpha = 0, \beta = 0$} & \multicolumn{2}{c}{$\alpha = \frac1{h_{F_\bx}}, \beta = 0$} & \multicolumn{2}{c}{$\alpha = 0, \beta = h_{F_\bx}$} \\
\hline
& DG error & Rate & DG error & Rate & DG error & Rate & DG error & Rate \\
\hline
\hline
\multicolumn{9}{c}{$p = 1$} \\
\hline
\hline
7.07e-02 & 1.04e+00 & \bf{---} & 1.16e+00 & \bf{---} & 1.07e+00 & \bf{---} & 1.09e+00 & \bf{---}\\
3.54e-02 & 7.78e-01 & \bf{0.43} & 5.02e-01 & \bf{1.21} & 6.84e-01 & \bf{0.64}  & 6.69e-01 & \bf{0.70} \\
1.77e-02 & 4.42e-01 & \bf{0.81} & 2.18e-01 & \bf{1.20} & 3.64e-01 & \bf{0.91} & 3.42e-01 & \bf{0.97}   \\
8.84e-03 & 2.29e-01 & \bf{0.95} & 1.02e-01 & \bf{1.09} & 1.85e-01 & \bf{0.97} & 1.71e-01 & \bf{1.00}  \\
4.42e-03 & 1.16e-01 & \bf{0.99} & 4.98e-02 & \bf{1.04} & 9.32e-02 & \bf{0.99} & 8.50e-02 & \bf{1.01}  \\
\hline
\hline
\multicolumn{9}{c}{$p = 2$} \\
\hline
\hline
7.07e-02 & 4.63e-01 & \bf{---} & 2.93e-01 & \bf{---} & 3.22e-01 & \bf{---} & 4.57e-01 & \bf{---}\\
3.54e-02 & 1.29e-01 & \bf{1.84} & 7.36e-02 & \bf{1.99} & 8.57e-02 & \bf{1.91} & 1.23e-01 & \bf{1.90}  \\
1.77e-02 & 3.31e-02 & \bf{1.97}  & 1.84e-02 & \bf{2.00} & 2.19e-02 & \bf{1.97} & 3.09e-02 & \bf{1.99} \\
8.84e-03 & 8.33e-03 & \bf{1.99} & 4.58e-03 & \bf{2.00} & 5.54e-03 & \bf{1.98} & 7.72e-03 & \bf{2.00}  \\
4.42e-03 & 2.09e-03 & \bf{2.00} & 1.14e-03 & \bf{2.00} & 1.39e-03 & \bf{1.99} & 1.93e-03 & \bf{2.00}  \\
\hline
\hline
\multicolumn{9}{c}{$p = 3$} \\
\hline
\hline
7.07e-02 & 8.09e-02 & \bf{---} & 5.42e-02 & \bf{---} & 5.54e-02 & \bf{---} & 8.19e-02 & \bf{---}\\
3.54e-02 & 1.06e-02 & \bf{2.93} & 6.74e-03 & \bf{3.01} & 7.12e-03 & \bf{2.96} & 1.04e-02 & \bf{2.97}  \\
1.77e-02 & 1.35e-03 & \bf{2.98} & 8.41e-04 & \bf{3.00} & 9.02e-04 & \bf{2.98} & 1.31e-03 & \bf{3.00}  \\
8.84e-03 & 1.69e-04 & \bf{2.99} & 1.05e-04 & \bf{3.00} & 1.13e-04 & \bf{2.99} & 1.64e-04 & \bf{3.00}  \\
4.42e-03 & 2.12e-05 & \bf{3.00} & 1.31e-05 & \bf{3.00} & 1.42e-05 & \bf{3.00} & 2.04e-05 & \bf{3.00}  \\
\hline
\hline
\multicolumn{9}{c}{$p = 4$} \\
\hline
\hline
7.07e-02 & 9.27e-03 & \bf{---} & 6.96e-03 & \bf{---} & 6.94e-03 & \bf{---} & 9.48e-03 & \bf{---}\\
3.54e-02 & 6.03e-04 & \bf{3.94} & 4.27e-04 & \bf{4.03}  & 4.39e-04 & \bf{3.98} & 5.99e-04 & \bf{3.99} \\
1.77e-02 & 3.81e-05 & \bf{3.98} & 2.66e-05 & \bf{4.01} & 2.76e-05 & \bf{3.99} & 3.74e-05 & \bf{4.00}  \\
8.84e-03 & 2.39e-06 & \bf{3.99} & 1.66e-06 & \bf{4.00} & 1.73e-06 & \bf{3.99} & 2.34e-06 & \bf{4.00}  \\
4.42e-03 & 1.50e-07 & \bf{4.00} & 1.04e-07 & \bf{4.00} & 1.08e-07 & \bf{4.00} & 1.46e-07 & \bf{4.00}  \\
\hline
\hline
\end{tabular}
\caption{$h$-convergence for the full polynomial version applied to the quantum harmonic oscillator problem with potential~$V(x) = 50x^2$ and exact solution~$\psi_2$ in~\eqref{EQN::EXACT-SOLUTION-HARMONIC} for different combinations of the stabilization parameters~$\alpha, \beta$ and volume penalty parameter~$\mu=0$.
\label{TAB::STABILIZATION-HARMONIC-FULL-2}}
\end{table}

\subsubsection{$p$-Convergence}

We now study numerically the~$p$-convergence of the method, i.e., for a fixed space--time mesh~$\Th$, we study the errors when increasing the polynomial degree~$p$. We consider the~$(1+1)$-dimensional problems above with the same parameters and the coarsest meshes for each case. In Figure~\ref{FIG::P-ERROR}, we compare the errors obtained for the method with the two choices for the discrete space~$\bVhp(\Th)$ analyzed in the previous sections: the full polynomial space~\eqref{EQN::POLYNOMIAL-SPACE} and the quasi-Trefftz polynomial space~\eqref{EQN::GLOBAL-QUASI-TREFFTZ}. As expected, for the quasi-Trefftz version we observe exponential decay of the error of order~$\ORDER{\mathrm{e}^{-bN_{dofs}}}$, where~$N_{dofs}$ denotes the total number of degrees of freedom. As for the full polynomial space, only root-exponential convergence~$\ORDER{\mathrm{e}^{-c\sqrt{N_{dofs}}}}$ is expected. The superiority of the quasi-Trefftz version is evident in all cases. 
Exponential convergence of space--time Trefftz and quasi-Trefftz schemes has been observed in several cases~\cite{ImbertGerard_Desperes_2014,Gomez_Moiola_2022,Banjai_Georgoulis_Lijoka_2017,Perugia_Schoeberl_Stocker_Wintersteiger_2020} but no proof is available yet (differently from the stationary case, \cite[\S3]{Hiptmair_Moiola_Perugia_2016}). In general, for a~$(d + 1)$-dimensional problem, we expect exponential convergence of order~$\mathcal{O}({e}^{-b \sqrt[d]{N_{DoFs}}})$ and $\mathcal{O}({e}^{-c \sqrt[(d+1)]{N_{DoFs}}})$ for the quasi-Trefftz and full-polynomial versions, respectively.

\begin{figure}[!ht]
    \centering
    \subfloat[Harmonic oscillator\label{FIG::P-ERROR-HO}]{\includegraphics[width = .45\textwidth]{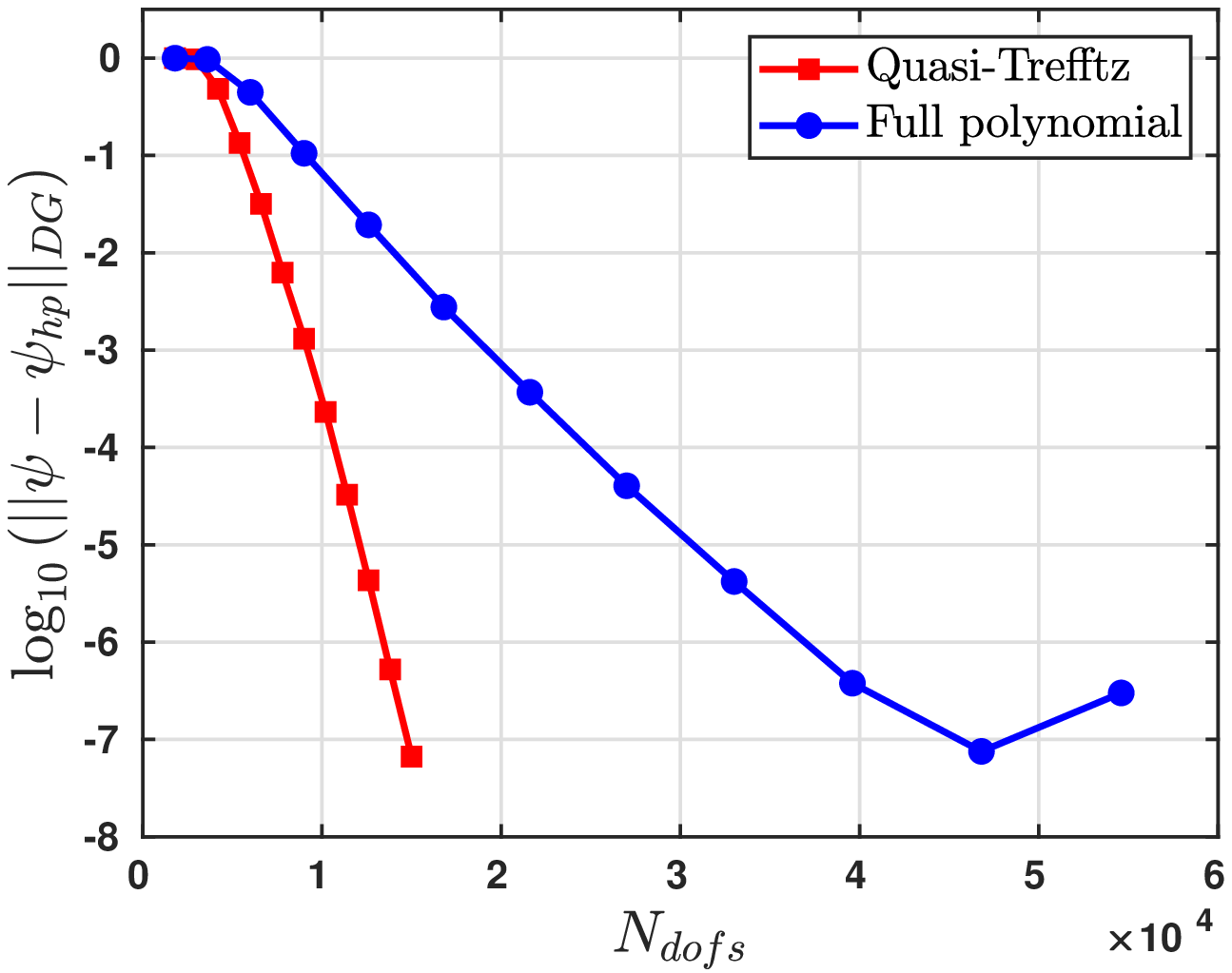} }
    \subfloat[Reflectionless potential]{\includegraphics[width = .45\textwidth]{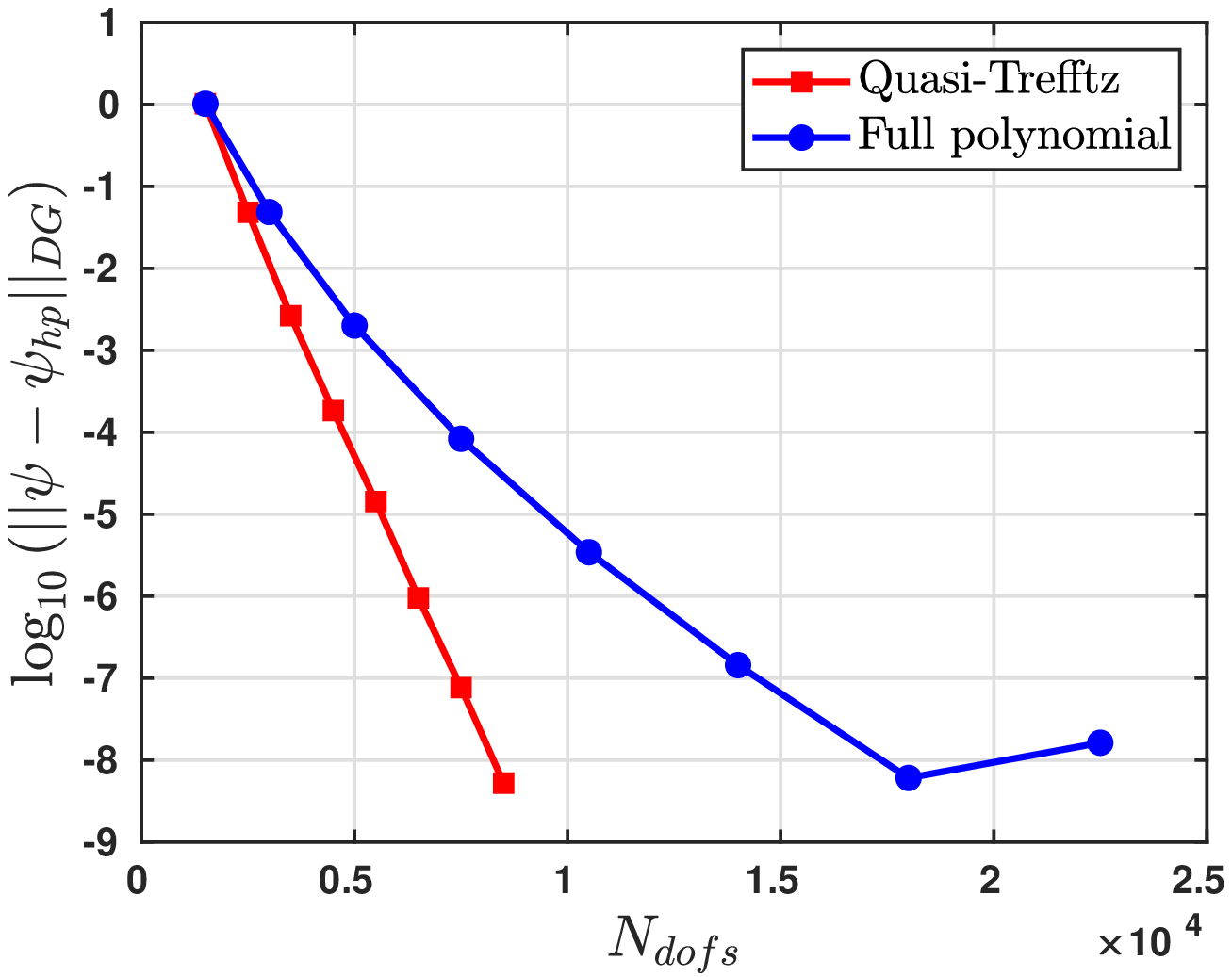} }
    \\
    \subfloat[Morse potential]{\includegraphics[width = .45\textwidth]{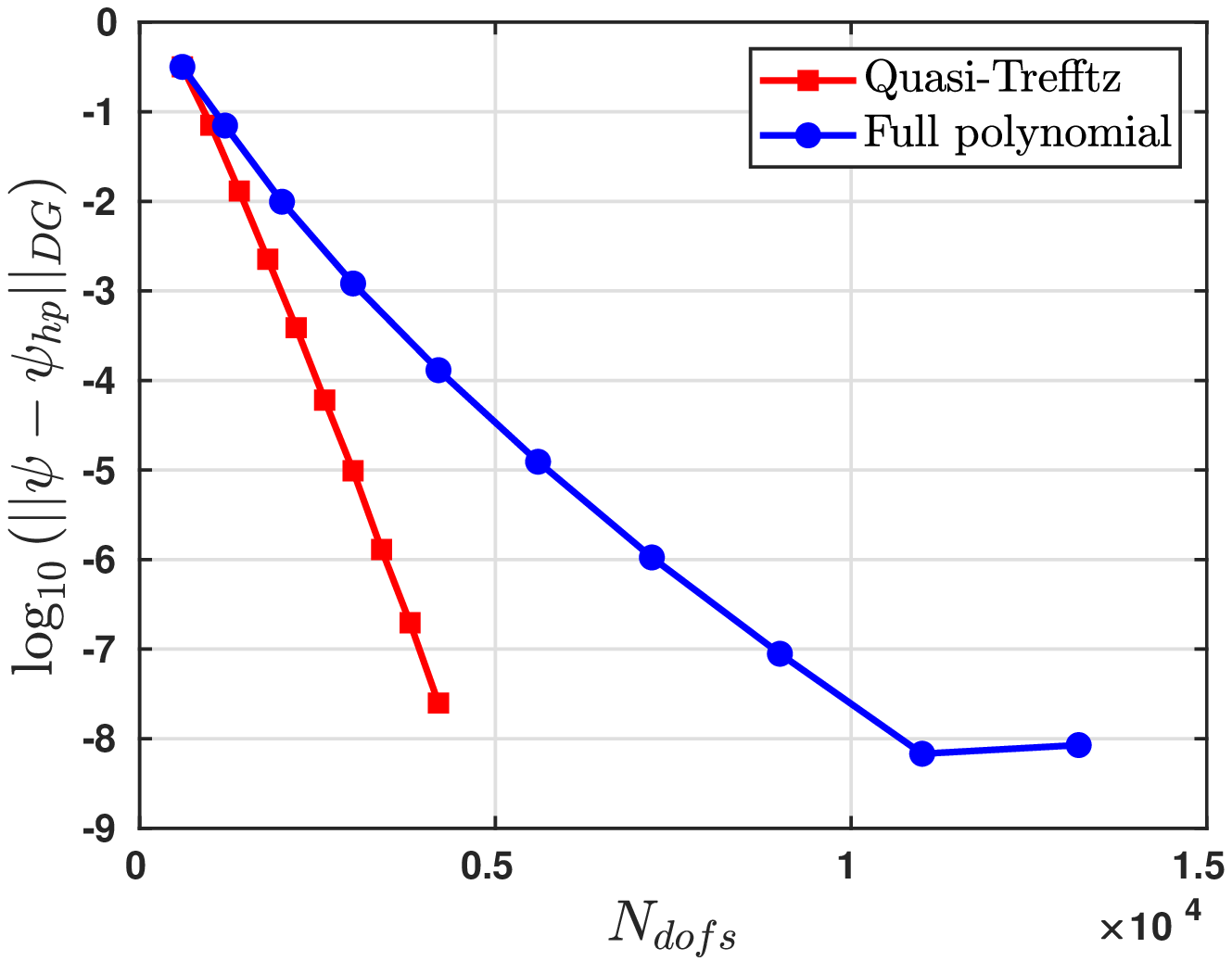}  }
    \subfloat[Square-well potential]{\includegraphics[width = .45\textwidth]{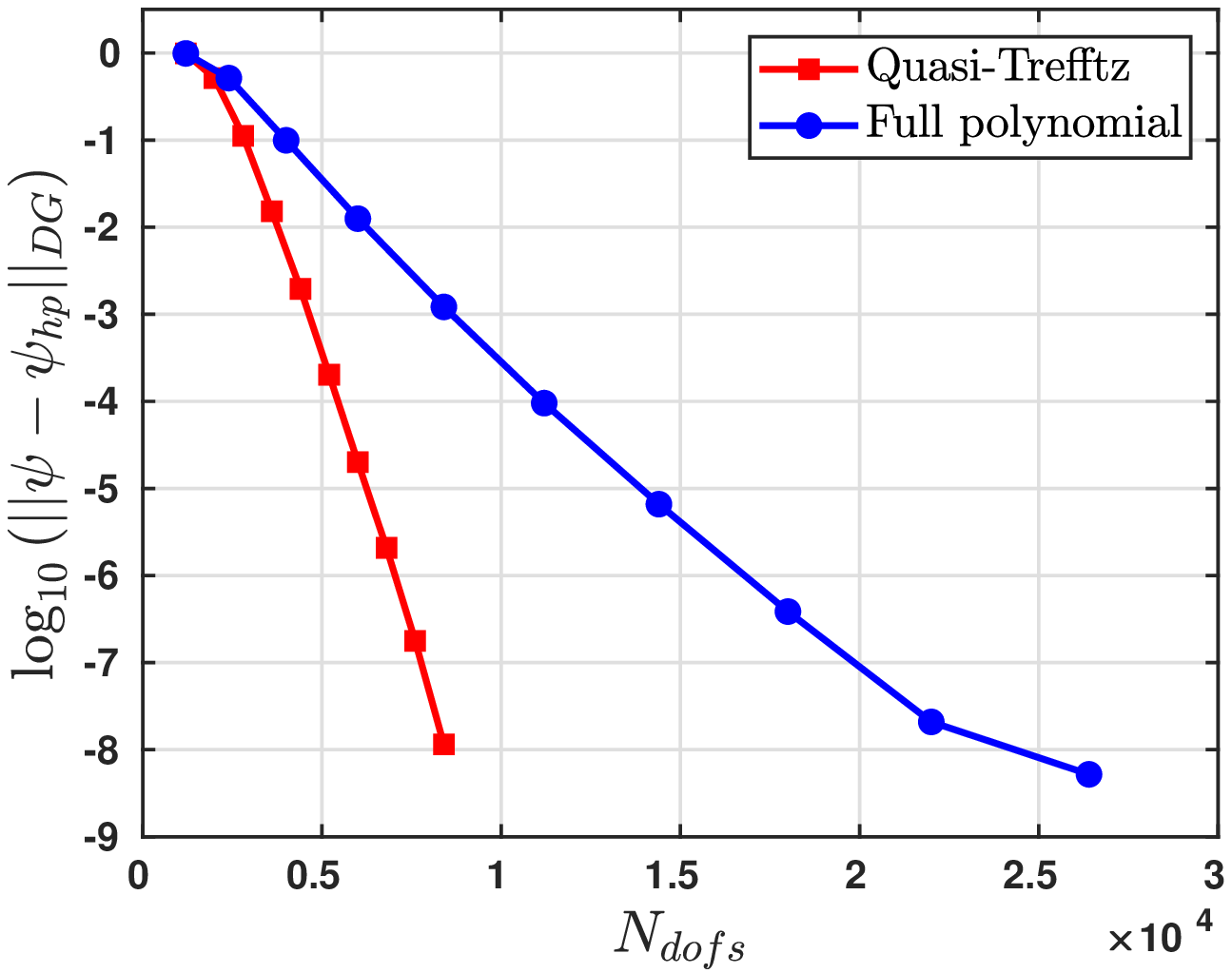}
    }
    \caption{$p$-convergence for the coarsest mesh in the $(1+1)$-dimensional problems. \label{FIG::P-ERROR}}
    \end{figure}


\subsubsection{Conditioning}
We now assess the conditioning of the stiffness matrix. 
In Figure~\ref{FIG::CONDITIONING} we compare the 2-condition number $\kappa_2(\cdot)$ for the stiffness matrix~$\mathbf{K}_n$ defined in Remark~\ref{REM::TIME-SLABS}, for the free particle problem~$V = 0$ on the space--time domain $\QT = (0, 1) \times (0, 1)$. 
We consider the proposed polynomial quasi-Trefftz space in~\eqref{EQN::GLOBAL-QUASI-TREFFTZ}, the full-polynomial space in~\eqref{EQN::POLYNOMIAL-SPACE} and the pure-Trefftz space of complex exponential wave functions~$\bVp{p}(\Th)$ proposed in~\cite{Gomez_Moiola_2022}. A basis~$\left\{\phi_\ell\right\}_{\ell = 1}^{2p + 1} \subset \bVp{p}(\Th)$ was defined in~\cite{Gomez_Moiola_2022} as
\begin{equation}
\label{EQN::COMPLEX-EXPONENTIAL-BASIS}
\phi_{\ell}(x, t) = \exp\left(i \left(\kappa_{\ell} x - \frac{\kappa_{\ell}^2}{2} t\right)\right), \qquad \ell = 1, \ldots, 2p + 1.
\end{equation}
We consider two choices for the parameters~$\kappa_\ell$: the arbitrary choice used in \cite{Gomez_Moiola_2022} $ \kappa_\ell = -p, \ldots, p$, and the choice $\kappa_\ell = 2\pi\ell/h_x$ which makes the basis orthogonal in each element.
The conditioning number~$\kappa_2(\mathbf{K})$ for the quasi-Trefftz space, the full polynomial space, and the Trefftz space with orthogonal basis asymptotically grows as~$\ORDER{h^{-1}}$ for all~$p \in \IN$, while for the Trefftz space with a non-orthogonal basis, asymptotically grows as~$\ORDER{h^{-(2p + 1)}}$.
Unfortunately, with higher dimensions and non-Cartesian elements, choosing the parameters and directions defining the basis functions~$\{\phi_{\ell}\}$ so as to obtain an orthogonal basis is more challenging.

\begin{figure}[!ht]
    \centering
    \subfloat[Quasi-Trefftz space~\eqref{EQN::GLOBAL-QUASI-TREFFTZ}]{
        \includegraphics[width = .45\textwidth]{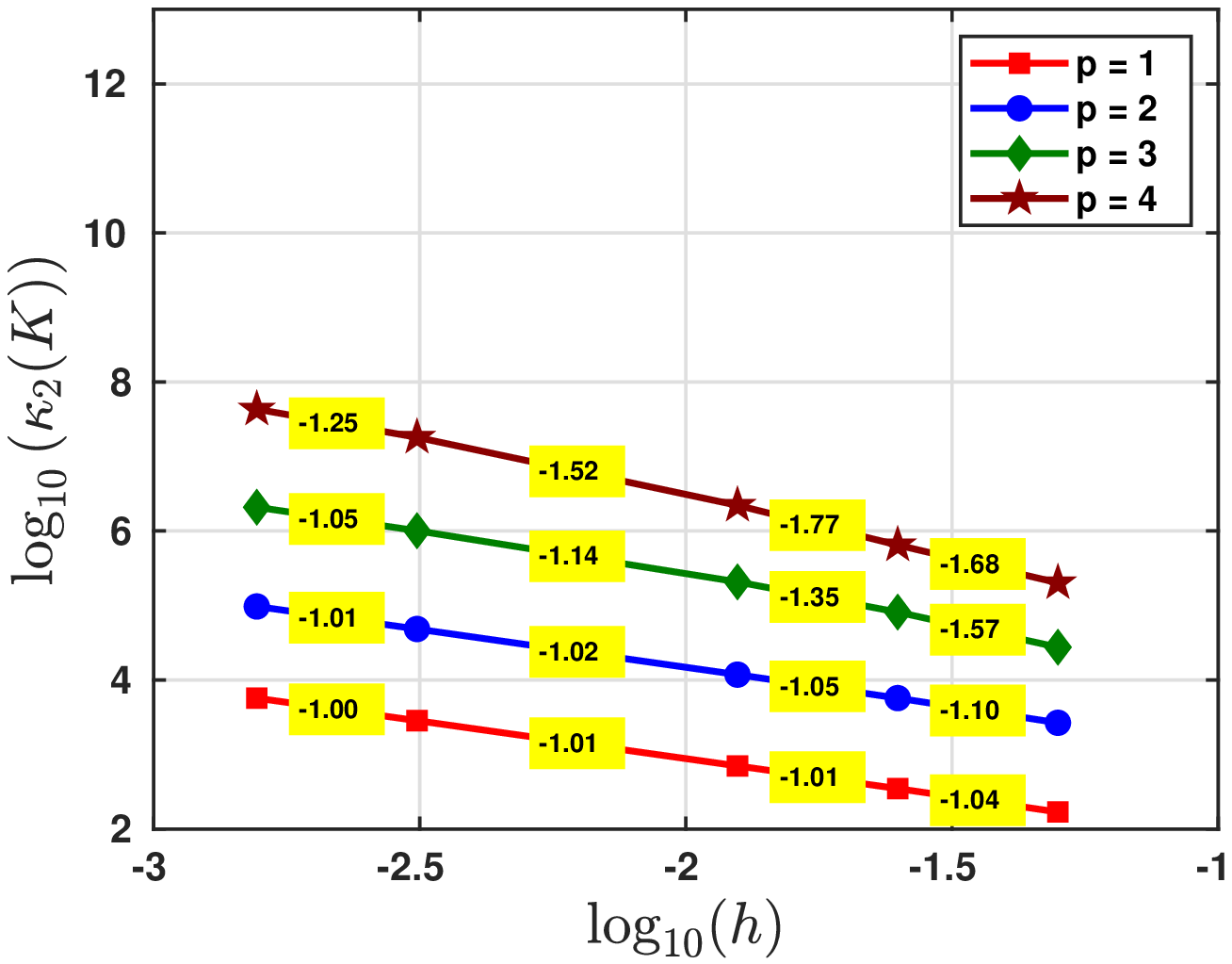}}
    \subfloat[Full polynomial space~\eqref{EQN::POLYNOMIAL-SPACE}]{
        \includegraphics[width = .45\textwidth]{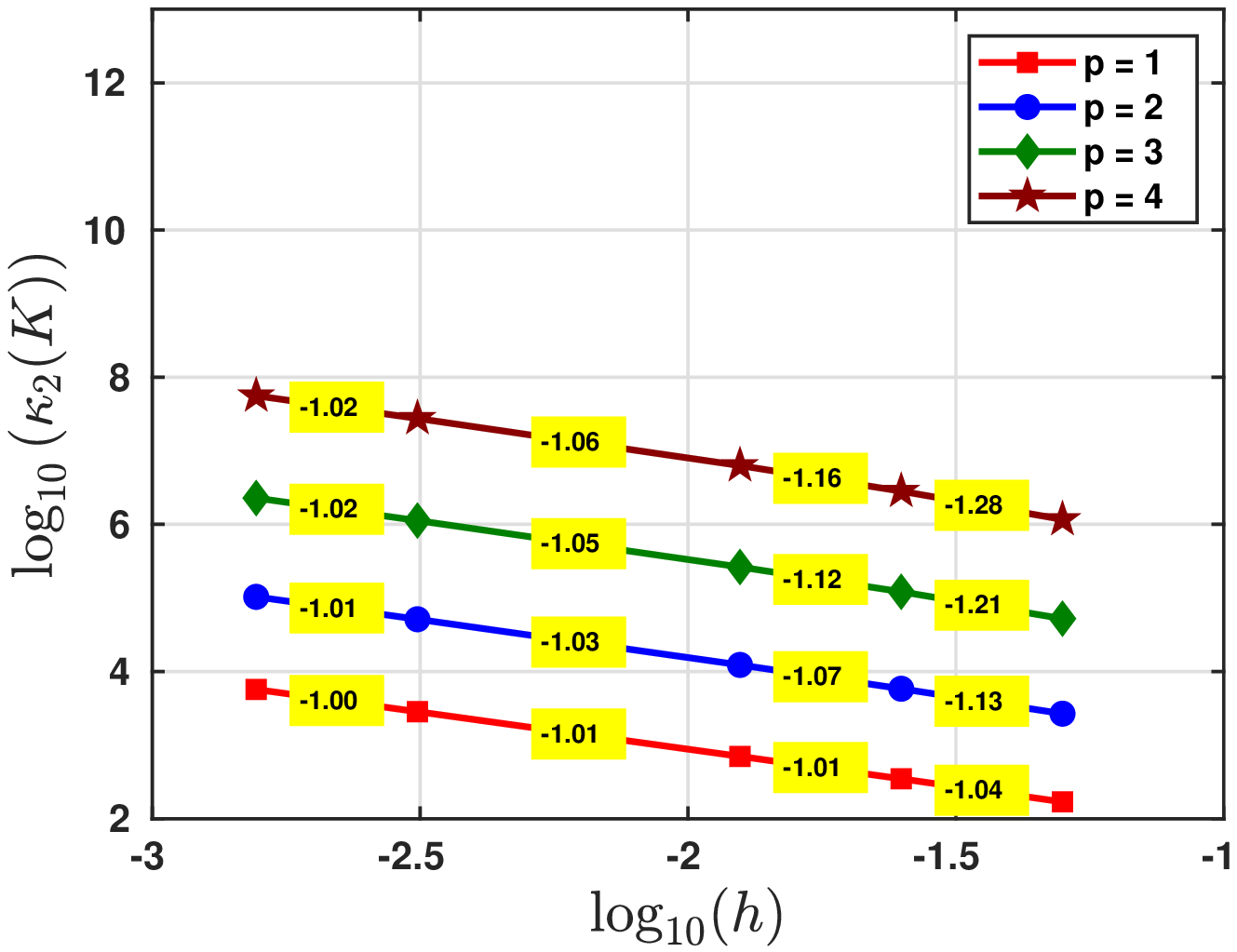}}
    \\
    \subfloat[Trefftz space~\cite{Gomez_Moiola_2022} (orthogonal basis)]{
        \includegraphics[width = .45\textwidth]{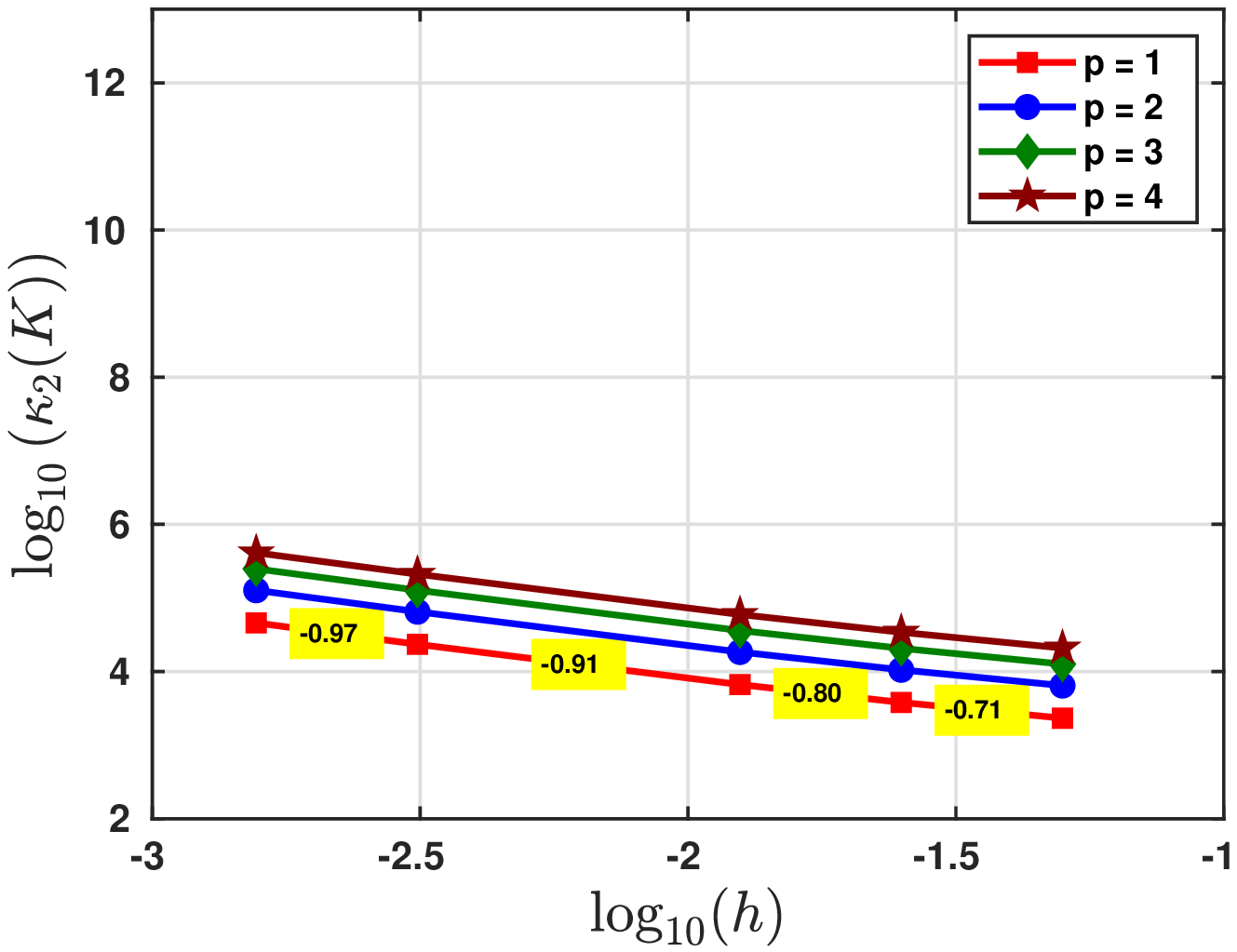}}
    \subfloat[Trefftz space~\cite{Gomez_Moiola_2022} (non-orthogonal basis)]{
        \includegraphics[width = .45\textwidth]{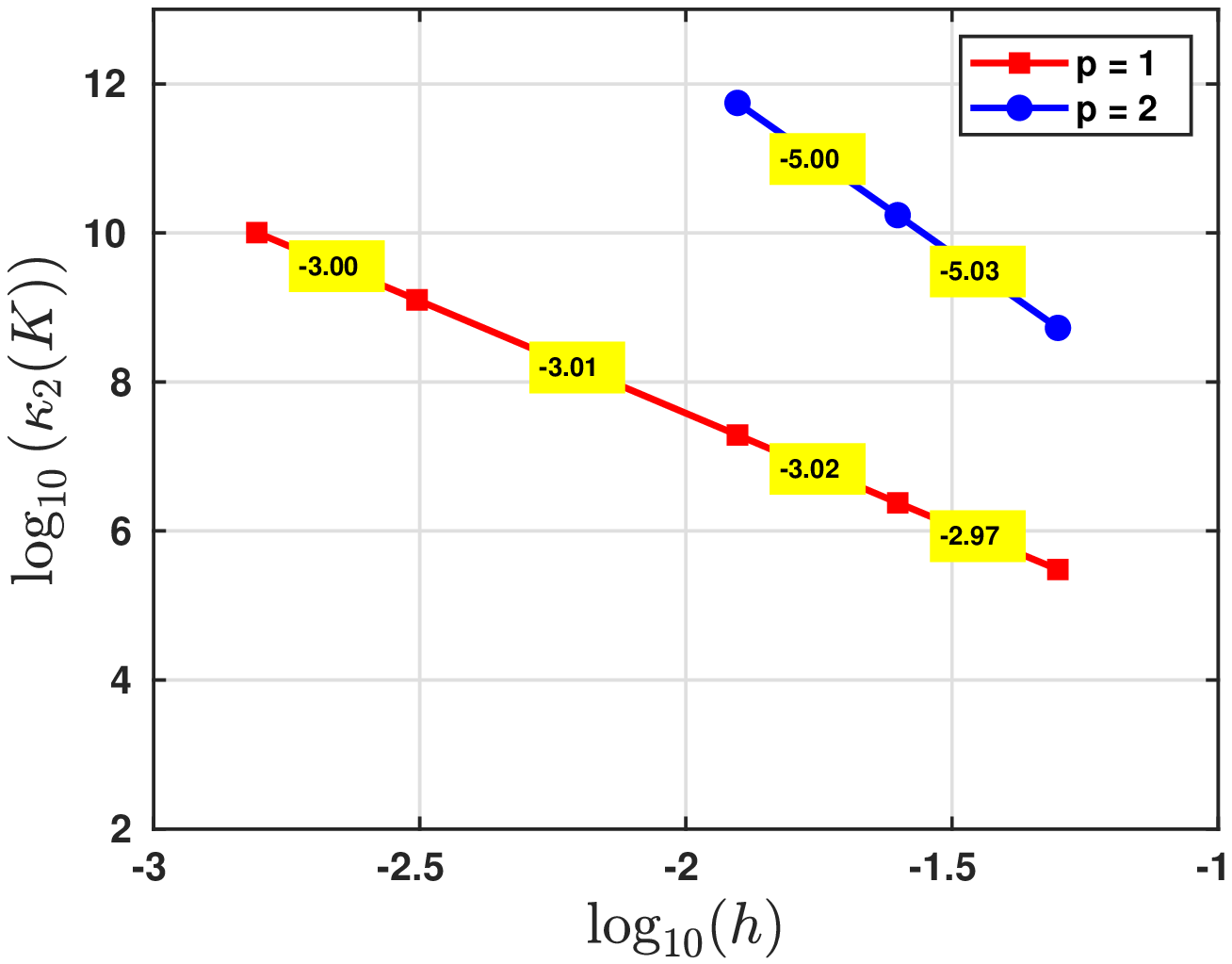}}
    \caption{Conditioning of the stiffness matrix for the DG method with different discrete spaces. \label{FIG::CONDITIONING}}
\end{figure}

\subsection{$(2+1)$-dimensional test cases}
\noindent We now present some numerical test for space dimension~$d = 2$. We recall that we use Cartesian space--time meshes with uniform partitions along each direction. 
\subsubsection{$h$-convergence}
\paragraph{Singular time-independent potential ($V(x, y) = 1 - 1/x^2 - 1/y^2$)} 
    
    We consider the $(2 + 1)$-dimensional problem on $\QT = (0, 1)^2 \times (0, 1)$ with exact solution  (see \cite{Subacsi_2002})
    \begin{equation}
    \label{EQN::EXACT-SOLUTION-RATIONAL-POTENTIAL}
        \psi(x, y, t) = x^2 y^2 e^{it}.
    \end{equation}
In Figure \ref{FIG::RATIONAL}, we show the errors obtained for a sequence of meshes with $h_x = h_y = h_t = 0.1$, $0.0667$, $0.05$, $0.04$ and different degrees of approximation $p$. As in the numerical results for the~$(1+1)$-dimensional problems, we obtain rates of convergence of order $\ORDER{h^p}$ in the DG norm, and $\ORDER{h^{p+1}}$ in the $L^2$ norm at the final time. 

\begin{figure}[!ht] 
    \centering
    \subfloat[DG error]{\includegraphics[width = .45\textwidth]{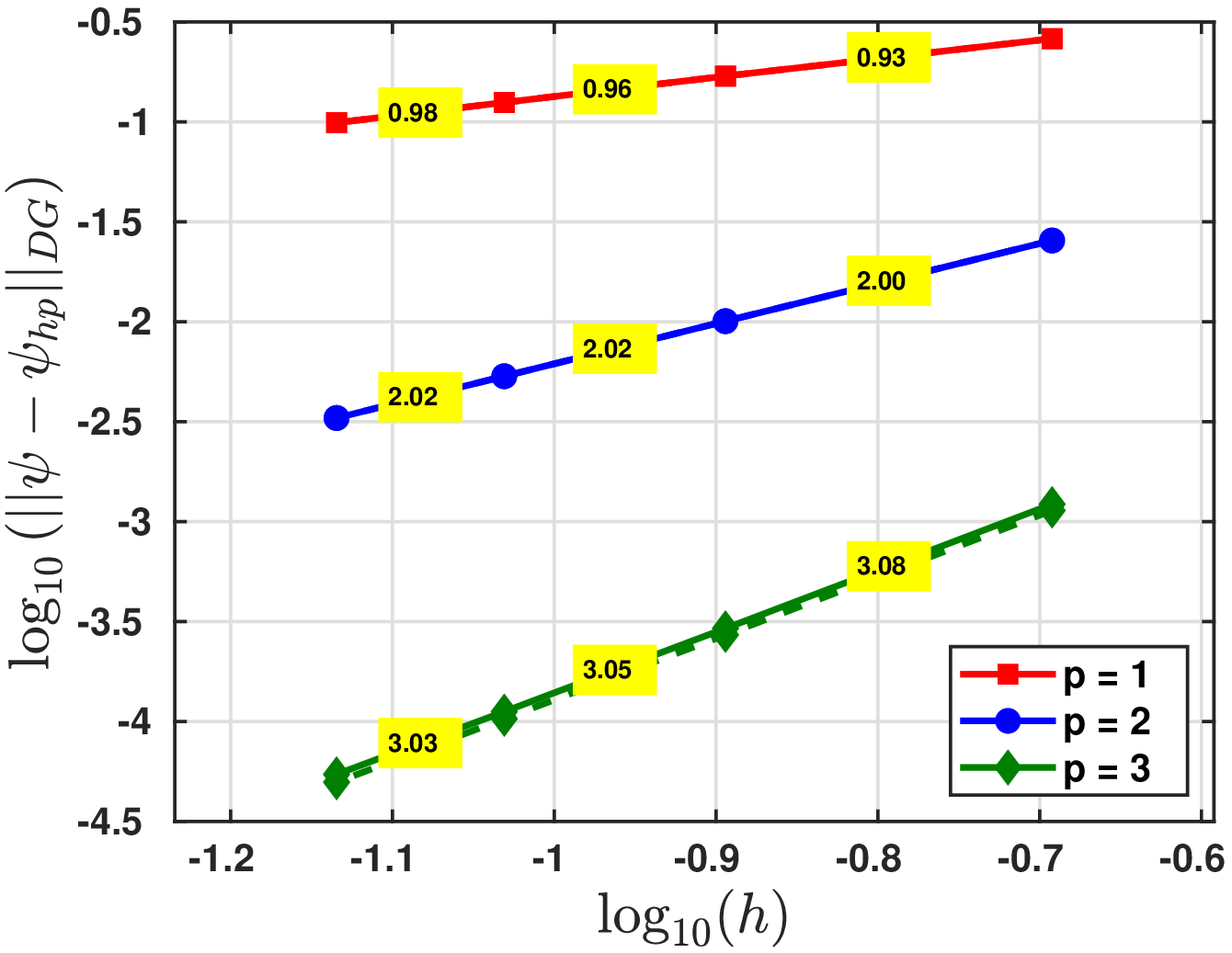}}
    \hspace{0.1in}
    \subfloat[$L^2$ error at $T = 1$]{\includegraphics[width = .45\textwidth]{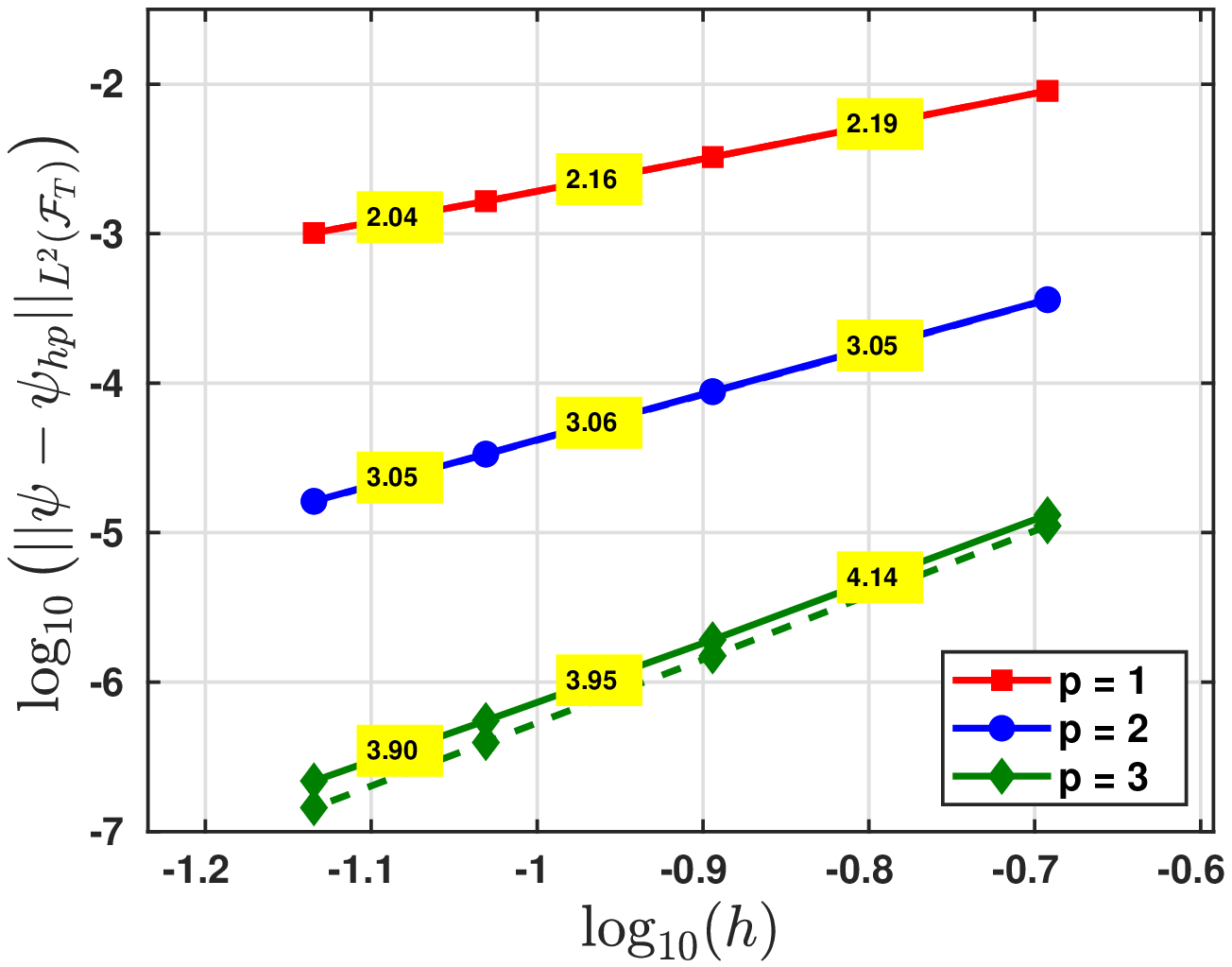}}
    \caption{$h$-convergence for the~$(2+1)$-dimensional problem with potential~$V(x, y) = 1-1/x^2 - 1/y^2$ and exact solution~\eqref{EQN::EXACT-SOLUTION-RATIONAL-POTENTIAL}. \label{FIG::RATIONAL}}
\end{figure}
    
\paragraph{Time-dependent~potential~($V(x, y, t) = 2\tanh^2(\sqrt{2} x) - 4(t - 1/2)^3 + 2\tanh^2(\sqrt{2}y) - 2$)} 
    
We now consider a manufactured problem with a time-dependent potential (see~\cite{Dehghan_Shokri_2007}). On the space--time domain~$\QT = (0, 1)^2 \times (0, 1)$ the exact solution is  
    \begin{equation}
    \label{EQN::EXACT-SOLUTION-TIME-DEPENDENT-POTENTIAL}
    \psi(x, y, t) = i \mathrm{e}^{i(t - 1/2)^4} \text{sech}(x) \text{sech}(y).
    \end{equation}
    In Figure \ref{FIG::TIME-DEPENDENT-POTENTIAL} we show the errors obtained for the sequence of meshes from the previous experiment, and optimal convergence is observed in both norms.
    
\begin{figure}[!ht]
    \centering
    \subfloat[DG error]{\includegraphics[width = .45\textwidth]{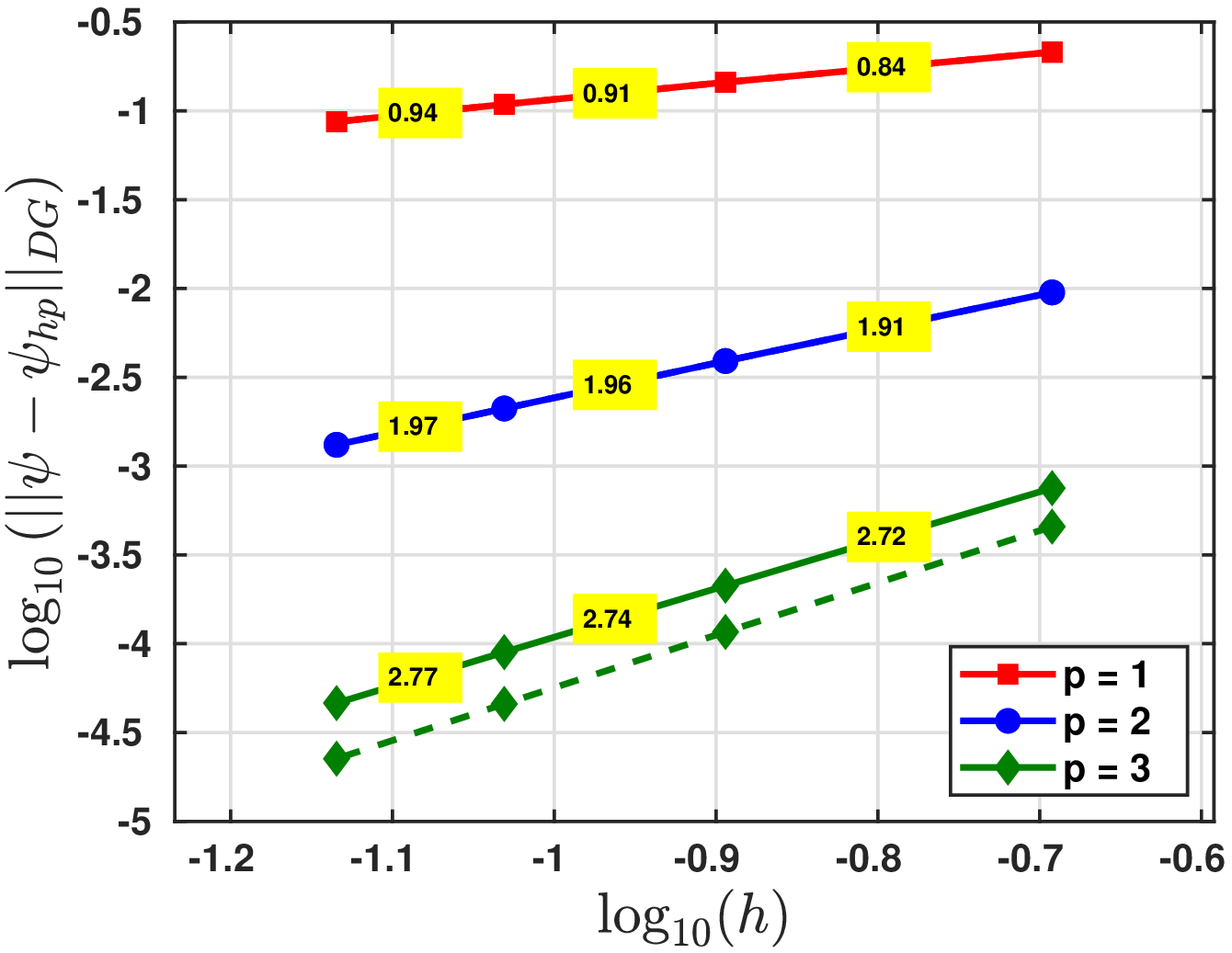}}
    \hspace{0.1in}
    \subfloat[$L^2$ error at $T = 1$]{\includegraphics[width = .45\textwidth]{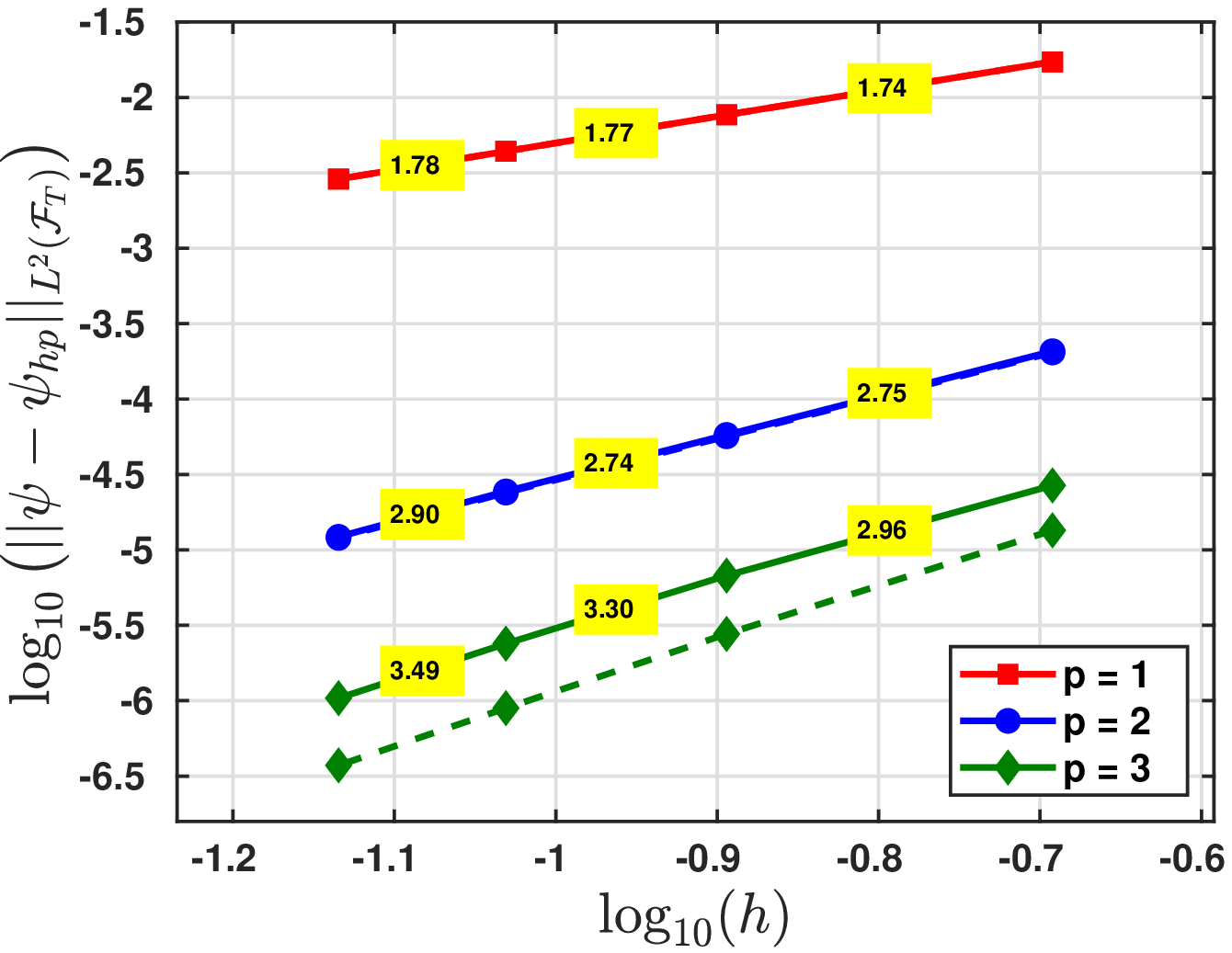}}
    \caption{$h$-convergence for the~$(2+1)$-dimensional problem with time dependent potential~$V(x, y, t) = 2\tanh^2\left(\sqrt{2} x\right) - 4\left(t - 1/2\right)^3 + 2\tanh^2\left(\sqrt{2}y\right) - 2$ and exact solution~\eqref{EQN::EXACT-SOLUTION-TIME-DEPENDENT-POTENTIAL}. \label{FIG::TIME-DEPENDENT-POTENTIAL}}
\end{figure}

\subsubsection{$p$-convergence}
In Figure \ref{FIG::P-CONVERGENCE-2D} we show the results obtained for the $p$-version of the method applied to the $(2+1)$-dimensional problems above, on the coarsest mesh. 
As expected, for the~$(2+1)$-dimensional case, the error of the quasi-Trefftz version decays root-exponentially as~$\ORDER{\mathrm{e}^{-b\sqrt{N_{dofs}}}}$.

\begin{figure}[!htp]
\centering
    \subfloat[Time independent potential~$V(x, y) = 1 - 1/x^2 - 1/y^2$]{
        \includegraphics[width = 0.45\textwidth]{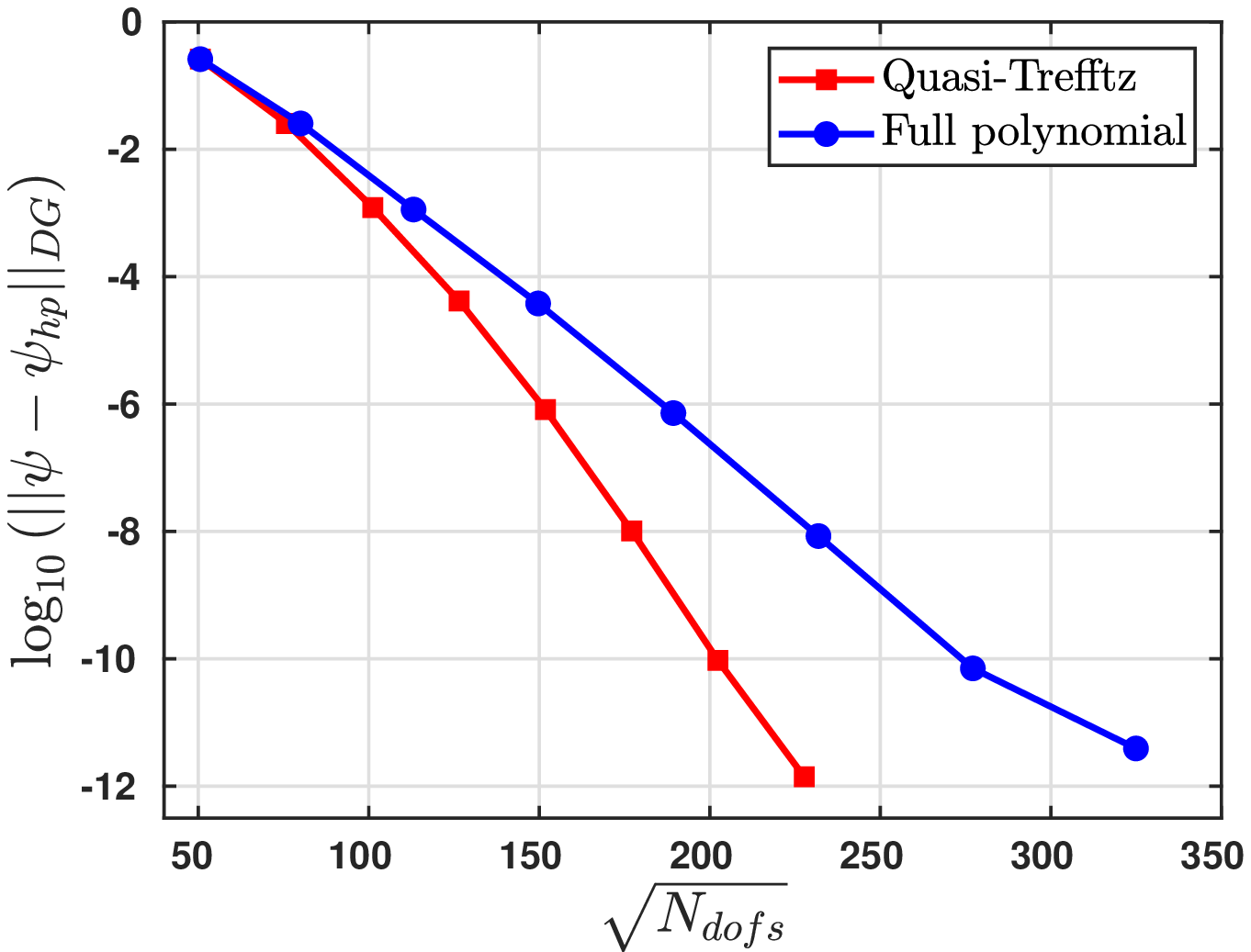}
    }
    \hspace{0.1in}
    \subfloat[Time dependent potential~$V(x, y, t) =$\\
    $2\tanh^2\left(\sqrt{2} x\right) - 4\left(t - 1/2\right)^3 + 2\tanh^2\left(\sqrt{2}y\right) - 2$]{
        \includegraphics[width = 0.45\textwidth]{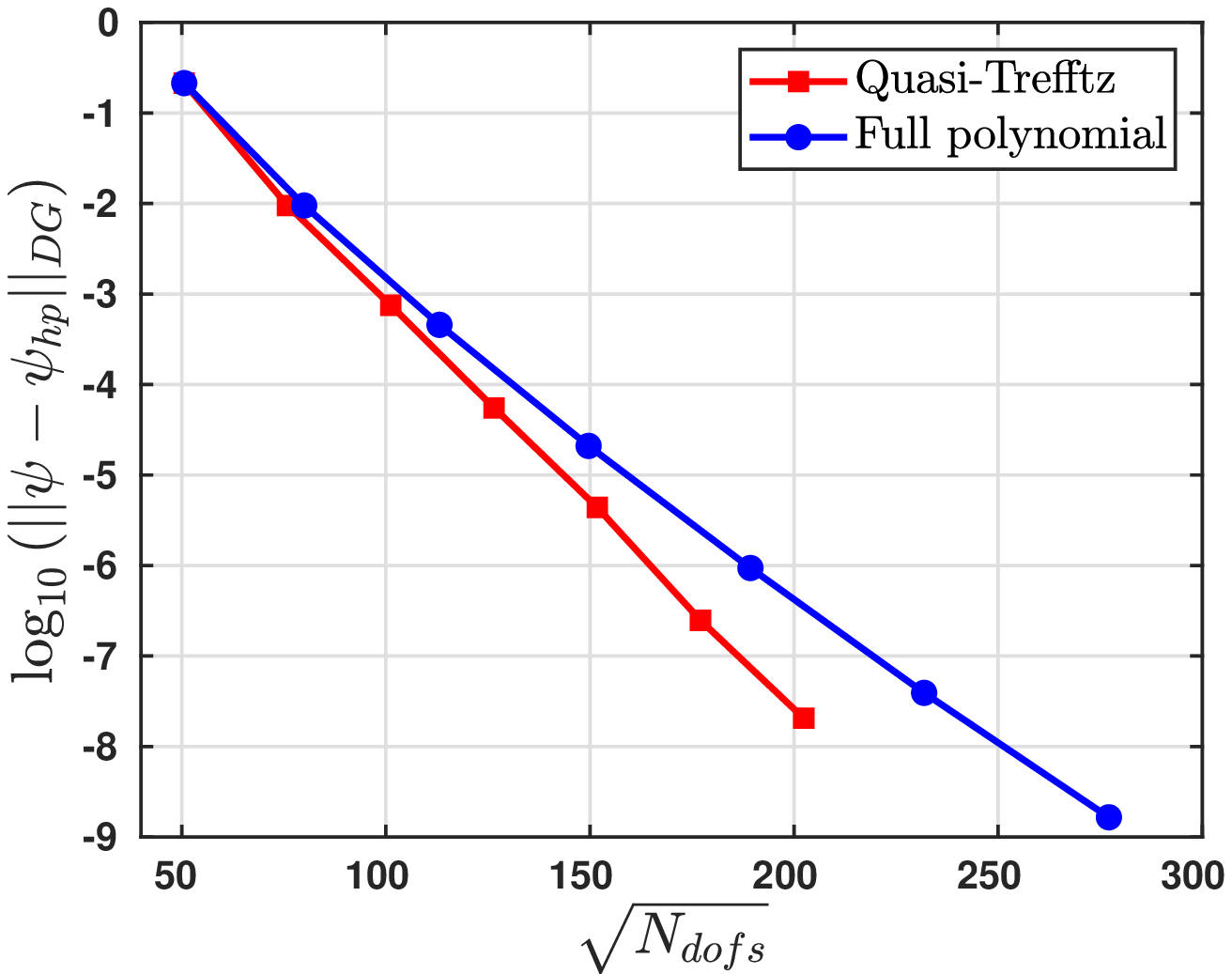}
    }
\caption{$p$-convergence for the $(2 + 1)$-dimensional problems. \label{FIG::P-CONVERGENCE-2D}}
\end{figure}

\section{Concluding remarks}
\label{SECT::CONCLUSIONS}

We have introduced a space--time ultra-weak discontinuous Galerkin discretization for the linear Schr\"o\-din\-ger equation with variable potential. 
The DG method is well-posed and quasi-optimal
in mesh-dependent norms
for any space dimension~$d\in \IN$, and for very general prismatic meshes and discrete spaces.
We proved optimal~$h$-convergence of order~$\ORDER{h^p}$, in such a mesh-dependent norm, for two choices of the discrete spaces: the space of piecewise polynomials, and a novel quasi-Trefftz polynomial space with much smaller dimension.
When the space--time mesh has a time-slab structure, the method allows for the decomposition of the resulting global linear system into a sequence of smaller problems on each time-slab: this is equivalent to an implicit time-stepping, possibly with local refinement in space--time.
We present several numerical experiments that validate the accuracy of the method for different potentials and high-order approximations.   


\end{document}